\documentclass[12pt]{amsart}
\usepackage{amsthm}
\usepackage{amsmath}
\usepackage{amssymb}
\usepackage{comment}
\usepackage{enumerate}
\usepackage{amsfonts}
\usepackage{mathtools}
\usepackage{afterpage}
\usepackage{hyperref}
\usepackage{todonotes}
\usepackage{algorithm,algorithmic}
\usepackage{multicol}
\usepackage{circuitikz}
\usepackage{tikz-cd}
\usetikzlibrary{math}
\usetikzlibrary{shapes.geometric}
\usepackage[utf8]{inputenc}
\usepackage[margin=1.1in]{geometry}

\usepackage[capitalize,nameinlink]{cleveref}
\crefname{theorem}{Theorem}{Theorems}
\crefname{lemma}{Lemma}{Lemmas}
\crefname{claim}{Claim}{Claims}

\newtheorem{theorem}{Theorem}[section]
\newtheorem{def-prop}[theorem]{Definition-Proposition}
\newtheorem{prop}[theorem]{Proposition}
\newtheorem{conjecture}[theorem]{Conjecture}
\newtheorem{lemma}[theorem]{Lemma}

\theoremstyle{definition}
\newtheorem{ex}[theorem]{Example}

\newtheorem{defin}[theorem]{Definition}

\theoremstyle{remark}
\newtheorem*{remark}{Remark}

\def\ShowAuthNotes{1}
\ifnum\ShowAuthNotes=1
\newcommand{\authnote}[2]{\textcolor{blue}{[{\footnotesize {\bf #1:} { {#2}}}]}}
\else
\newcommand{\authnote}[2]{}
\fi

\newcommand{\X}{\mathcal{X}}

\newcommand{\II}{\mathfrak{I}}
\newcommand{\Z}{\mathbb{Z}}
\newcommand{\cE}{\mathcal{E}}

\newcommand{\CR}{\mathrm{cr}}
\newcommand{\MM}{\mathcal{M}}
\newcommand{\cC}{\mathcal{C}}
\newcommand{\TNC}{\mathcal{TC}}
\newcommand{\NCM}{\mathcal{NM}}
\newcommand{\NP}{\mathcal{NP}}
\def\P{{\mathbf{P}}}
\def\C{{\mathbb C}}
\def\T{{\mathbf{T}}}

\DeclareMathOperator{\wt}{wt}
\DeclareMathOperator{\Gr}{Gr}
\DeclareMathOperator{\LG}{LG}
\DeclareMathOperator{\SL}{SL}
\DeclareMathOperator{\Sp}{Sp}
\DeclareMathOperator{\cyc}{cyc}

\title{Electrical networks and the Grove algebra}
\author{Yibo Gao}
\address{Department of Mathematics, University of Michigan, \mbox{Ann Arbor, MI 48109}}
\email{\href{mailto:gaoyibo@umich.edu}{{\tt gaoyibo@umich.edu}}}
\author{Thomas Lam}
\address{Department of Mathematics, University of Michigan, \mbox{Ann Arbor, MI 48109}}
\email{\href{mailto:tfylam@umich.edu}{{\tt tfylam@umich.edu}}}
\author{Zixuan Xu}
\address{Department of Mat7ematics, Massachusetts Institute of Technology, \mbox{Cambridge, MA 02139}}
\email{\href{mailto:zixuanxu@mit.edu}{{\tt zixuanxu@mit.edu}}}

\thanks{T.L.\ was supported by Grant No.~DMS-1953852 from the National Science Foundation.} 
\date{\today}

\begin{document}

\begin{abstract}
We study the ring of regular functions on the space of planar electrical networks, which we coin the \emph{grove algebra}.  This algebra is an electrical analogue of the Pl\"ucker ring studied classically in invariant theory.  We develop the combinatorics of double groves to study the grove algebra, and find a quadratic Gr\"obner basis for the grove ideal.
\end{abstract}
\maketitle

\section{Introduction}\label{sec:intro}
Electrical resistor networks are modeled by undirected graphs whose edges are given positive weights, and have been studied for two centuries.  The physical axioms governing such networks, Ohm's Law and Kirchhoff's Law, were discovered in the 19th century.  In the last few decades, planar resistor networks have been studied from a modern perspective in the works of Curtis--Ingerman--Morrow \cite{Curtis-Ingerman-Morrow}, de Verdi\`ere--Gitler--Vertigan \cite{deVerdiere-Gitler-Vertigan}, Kenyon--Wilson \cite{Kenyon-Wilson}, Lam--Pylyavskyy \cite{Lam-Pylyavskyy}, Lam \cite{lam2004electroid}, Chepuri--George--Speyer \cite{chepuri2021electrical}, and Bychkov--Gorbounov--Kazakov--Talalaev \cite{bychkov2021electrical}, and others.

\subsection{Grove coordinates}
In Kirchhoff's classical work, he gave formulae for the effective resistance of an electrical network by counting spanning trees.  We formalize the tree enumeration in terms of groves \cite{Kenyon-Wilson}.  Let $\Gamma$ be a planar electrical network with $n$ boundary nodes.  A \emph{grove} in $\Gamma$ is a spanning forest $F$ such that every connected component contains a boundary node.  Each grove has a noncrossing boundary partition $\sigma(F)$.

The grove coordinates $L_\sigma(\Gamma)$ count groves in $\Gamma$ with chosen boundary partition $\sigma$ and determine $\Gamma$ up to electrical equivalence (series-parallel and $Y-\Delta$ moves).  In \cite{lam2004electroid}, the second author constructed a compactification of the space of planar electrical networks with $n$ boundary vertices, using the grove coordinates $(L_\sigma)$.  Furthermore, an embedding of this compactification into the Grassmannian $\Gr(n-1,2n)$ is constructed; we denote the image of this compactification by $\X_{n,\geq 0}$ and let $\X_n \subset \Gr(n-1,2n)$ denote its Zariski-closure, an irreducible subvariety of the Grassmannian.  

The starting point of our work is the analogy between the Grassmannian $\Gr(k,n)$ and the variety $\X_n$, and between the Pl\"ucker coordinates $\Delta_I$ and the grove coordinates $L_\sigma$.

    \begin{figure}[t]
        \centering
        \begin{tabular}{|c|c|}
        \hline
             $\Gr(k,n)$ & $\X_n$  \\
             Totally nonnegative Grassmannian $\Gr(k,n)_{\geq 0}$ & Space of cactus networks $\X_{n,\geq 0}$ \\
             Dimer model or plabic graph & Cactus network $\Gamma$ \\
             \hline
             $k$-element subsets of $[n]$ & noncrossing partitions on $[n]$ \\
             Pl\"ucker coordinates $\Delta_I$ & Grove coordinates $L_\sigma$ \\
             Perfect matchings & Spanning forests \\
             \hline
             Partial noncrossing matchings $A_{k,n}$ & $3$-noncrossing matchings $\TNC_n$ \\
             Temperley-Lieb immanants $F_{\tau,T}$ & Bush basis $B_{\xi}$ \\
             Double dimers & Double groves \\
             \hline
             Semistandard tableaux $\T$ & Nested Dyck paths $\P \in \cC^{(d)}_n$\\
             Standard monomial basis $\Delta_{\T}$ & Standard grove basis $L_\P$ \\
             Dual canonical basis $H(\T)$ & Electrical canonical basis $E_\P$ ?? \\
             \hline
             Bounded affine permutations & Matchings on $[2n]$ \\
             Positroid variety $\Pi_f$ & Electroid variety $\X_\tau$ \\ 
             \hline
        \end{tabular}
        \caption{Parallels between $\X_n$ and $\Gr(k,n)$}
        \label{fig:analogy}
    \end{figure}

\subsection{Pl\"ucker algebra and Grove algebra}
The Pl\"ucker algebra is the homogeneous coordinate ring $R(k,n):=\C[\Gr(k,n)]$ of the Grassmannian.  It can be identified with the quotient of the polynomial ring $\C[\Delta_I]$ generated by the Pl\"ucker coordinates $\Delta_I$, labeled by $k$-element subsets $I \subset \binom{[n]}{k}$, modulo the Pl\"ucker ideal generated by the Pl\"ucker relations. The Pl\"ucker algebra $R(k,n)$ is also isomorphic to the ring of $\SL(k)$-invariants in the polynomial functions on a $k\times n$ matrix.  In this latter setting, the description of $R(k,n)$ in terms of Pl\"ucker coordinates and Pl\"ucker relations are known as the first and second fundamental theorems of invariant theory.  

We recall the following classical theorem regarding the basis for the Pl\"ucker ideal and the Pl\"ucker ring; see \cite{Sturmfels-White}. 


\begin{theorem}\label{thm:plucker_basis}\
\begin{enumerate}
    \item The Pl\"ucker relations form a quadratic Gr\"obner basis for the Pl\"ucker ideal.
    \item The degree $d$ homogeneous piece $R(k,n)_d$ of the Pl\"ucker ring has basis the set of standard monomials $\Delta_\T = \Delta_{T_1}\Delta_{T_2} \cdots \Delta_{T_d}$ as $(T_1,\ldots,T_d)$ varies over the columns of a semistandard Young tableaux $\T$ with rectangular shape $k \times d$ and entries in $\{1,2,\ldots,n\}$.
\end{enumerate}

\end{theorem}

In analogy to the objects involved in \cref{thm:plucker_basis}, we define the \emph{grove algebra} $G_{n}:=\C[\X_n]$ to be homogeneous coordinate ring of $\X_n$ and use $G_{d,n}$ denote the degree $d$ homogeneous piece of $G_n$.
There is a natural bijection $P \mapsto \sigma(P)$ (see \cref{sec:catalan}) from Dyck paths of semilength $2n$ to noncrossing partitions of $[n]$.  We abuse notation by writing $L_P$ for $L_{\sigma(P)}$. 
Let $\cC^{(d)}_n$ denote the set of $d$-tuples of nested Dyck paths (see \cref{sec:catalan}). For $\P = (P_1,\ldots,P_d) \in \cC^{(d)}_n$, we define the standard grove monomials
$$
L_\P:= L_{P_1}L_{P_2} \cdots L_{P_d}.
$$
In this work, we present a result analogous to \cref{thm:plucker_basis} for the grove algebra, stated as follows.

\begin{theorem}[{\cref{thm:tableaux-basis} and \cref{prop:dim_G_nd}}]\
\begin{enumerate}
    \item 
    The grove algebra $G_n$ is generated by the grove coordinates $L_\sigma$ modulo the ideal $\II_n$ generated by the quadratic relations $r_{P,Q}$ described in \cref{def:noncomparable_relation}.  The relations $r_{P,Q}$ give a Gr\"obner basis of the ideal $\II_n$, with respect to a term order described in \cref{sec:tableaux}.
    \item
    We have $\dim(G_{d,n}) = \prod_{1 \leq i \leq j \leq n-1} \frac{i+j+2d}{i+j}$ and $G_{d,n}$ has basis the set of
    standard grove monomials $L_\P$, $\P \in \cC^{(d)}_n$.
\end{enumerate}    
\end{theorem}

Chepuri, George, and Speyer \cite{chepuri2021electrical} and Bychkov, Gorbounov, Kazakov, and Talalaev \cite{bychkov2021electrical} have shown that the variety $\X_n$ is isomorphic to the Lagrangian Grassmannian $\LG(n-1,2n-2)$, extending earlier work of Lam and Pylyavskyy \cite{Lam-Pylyavskyy} relating electrical networks to the symplectic group.
By Borel-Weil theory, each homogeneous piece $G_{d,n}$ of $G_n$ is isomorphic to an irreducible representation of the symplectic group $\Sp(2n-2)$; see \cref{sec:Lagrangian}.  The relation between our description of $G_{d,n}$, including the standard grove basis $\{L_\P\}$, and the usual constructions from highest weight theory of $\Sp(2n-2)$ is far from being clear. It is an interesting problem to compare these approaches.

\subsection{Electrical canonical basis}
Recall that the standard monomial basis $\Delta_{\T}$ is not compatible with the cyclic symmetry of $\Gr(k,n)$. Lusztig's dual canonical basis of $R(k,n)$ is compatible with cyclic symmetry and exhibits remarkable positivity properties.  See \cite{lam2019cyclic} for the following result.

\begin{theorem}\label{thm:lam19_canonical_basis}
The space $R(k,n)_d$ has a dual canonical basis $H(\T)$ with the following properties:
\begin{enumerate}
\item For $d=1$, we have $H(\T)= \Delta_I$, where $I$ is the set of entries in the one-column tableau $\T$.
\item For $d=2$, the set $\{H(\T)\}$ 
is exactly the set of Temperley-Lieb immanants \cite{lam2015dimers} $\{F_{\tau,T} \mid (\tau, T)\in A_{k,n}\}$ which we describe in \cref{sec:F-basis}.
\item For any semistandard tableau $\T$, the function $H(\T)$ is a nonnegative function on the totally nonnegative Grassmannian $\Gr(k,n)_{\geq 0}$.
\item For any $T$, we have $\chi^*(H(\T)) = H(\chi(\T))$
where $\chi^*$ denotes the pullback map induced by the signed cyclic symmetry $\chi:\Gr(k, n)\to \Gr(k, n)$, and $\chi(\T)$ is the \emph{promotion} of $\T$.
\end{enumerate}
\end{theorem}
In \cite{lam2019cyclic} (see also \cite{Lam-CDM}) it is further shown that the dual canonical basis $H(\T)$ is compatible with restrictions to the homogeneous coordinate ring of positroid varieties $\Pi_f$ \cite{Knutson-Lam-Speyer}.

Similar to the standard monomial basis for $\Gr(k,n)$, the standard grove basis $L_\P$ is not compatible with the cyclic symmetry of planar electrical networks (or that of $\X_n$). Therefore we conjecture that there exists a canonical basis for the space $G_{d,n}$ analogous to the dual canonical basis $H(\T)$ for $R(k_n)_d$ in \cref{thm:lam19_canonical_basis}.

\begin{conjecture}
The space $G_{d,n}$ has an electrical canonical basis $E_\P$, $\P \in \cC^{(d)}_n$ with the following properties:
\begin{enumerate}
    \item For $d = 1$, we have $E_P = L_P$.
    \item The electrical canonical basis takes nonnegative values on the compactification of the space of electrical networks $\X_{n,\geq 0}$.
    \item Any monomial in the grove coordinates $L_\sigma$ expands positively into the electrical canonical basis $E_P$.
    \item There are actions of the cyclic group $\Z/2n\Z = \langle \chi \rangle$ on $\cC^{(d)}_n$ and on $\X_n$, preserving $\X_{n,\geq 0}$, such that $\chi^*(E_\P) = E_{\chi(\P)}$, where $\chi^*$ denotes the pullback on $\X_n$ induced by $\chi$.
\end{enumerate}
\end{conjecture}
The action of $\chi$ on $\cC^{(d)}_n$ is an electrical analogue of promotion on tableaux. For $d=1$, it corresponds to rotation on noncrossing matchings, under the bijection between Dyck paths and noncrossing matchings (\cref{sec:catalan}).

Lusztig has defined a totally nonnegative part $\LG(n-1,2n-2)_{\geq 0}$ of $\LG(n-1,2n-2)$; see \cite{Karpman}.  The spaces $\X_{n,\geq 0}$ and $\LG(n-1,2n-2)_{\geq 0}$ are both ``totally positive" versions of the Lagrangian Grassmannian.  However, they are quite different.  For example, they are cell complexes with a different number of cells.  We expect the analogy between $\LG(n-1,2n-2)_{\geq 0}$ and $\X_{n,\geq 0}$ to parallel the analogy between Lusztig's dual canonical basis and the electrical canonical basis.

In \cite{lam2004electroid}, a stratification of $\X_n$ by electroid varieties $\X_\tau$ is constructed, indexed by matchings $\tau$ on $2n$ points. We conjecture the electrical canonical basis to have the following properties with respect to $\X_\tau$.

\begin{conjecture}\
\begin{enumerate}
    \item For a matching $\tau$, the electrical canonical basis elements $E_\P$ that do not restrict to 0 on $\X_\tau$ form a basis of the homogeneous coordinate ring $\C[\X_\tau]$.
    \item For a matching $\tau$, if $E_\P$ is not identically 0 on $\X_\tau$, then it takes strictly positive values on $\X_{\tau,>0} = \X_\tau \cap \X_{n,\geq 0}$.
\end{enumerate}
\end{conjecture}

\subsection{Electrical canonical basis in degree two}
Using the combinatorics of electrical networks, we construct the electrical canonical basis in degree two, which we call the \emph{Bush basis}. As a consequence, we discover some remarkable combinatorial properties of \emph{double groves}.

There is a bijection between the set $\cC^{(2)}_n$ of pairs of nested Dyck paths of semilength $2n$, and the set of $3$-noncrossing matchings on $2n$ points, defined in \cref{sec:combinatorics-3noncrossing}.  We construct the Bush basis $B_\xi$ of $G_{2,n}$ labeled by $3$-noncrossing matchings $\xi$, defined by (\cref{def:B-basis})
$$
B_\xi(\Gamma):= \sum_{H \subset 2\Gamma} \alpha(H)_\xi \wt(H),
$$
where the summation is over double groves in $\Gamma$, and $\alpha(H)_\xi$ are certain nonnegative integers.  This definition is compatible with the natural rotation action on 3-noncrossing matchings.  In \cref{thm:L-into-B} we give a positive expansion of the quadratic monomials $L_\sigma L_{\sigma'}$ into the Bush basis $B_\xi$.

The relation between $L_\sigma$ and $\Delta_I$ is more than an analogy.  In \cite{lam2004electroid}, it is shown that the embedding $\iota:\X_n \hookrightarrow \Gr(n-1,2n)$ is given by the formula
$$
\iota^*(\Delta_I) = \sum_I M_{I\sigma} L_\sigma
$$
where $M_{I \sigma}$ is the concordance matrix defined in \cite{lam2004electroid} which we recall in \cref{sec:geometry}.  Likewise, we give in \cref{sec:F-basis} a positive formula expanding the Temperley-Lieb immanants $F_{\tau,T}$ in terms of the Bush basis $B_\xi$. This produces a commutative square of positive expansions:
\[\begin{tikzcd}[column sep=huge]
\Delta_I\Delta_J \arrow[r,"\text{Theorem~}\ref{thm:delta-to-F}"] \arrow[d,"\text{concordance}"] & F_{\tau,T} \arrow[d,"\text{Definition~}\ref{def:F_tau_T}"]  \\
L_{\sigma}L_{\sigma'} \arrow[r,"\text{Theorem~}\ref{thm:L-into-B}"] & B_{\xi}
\end{tikzcd}.\]

It would also be interesting to carry out our program in the setting of the Ising model and orthogonal Grassmannian \cite{galashin-pylyavskyy}.

\section{Background}\label{sec:background}
In this section, we introduce relevant background, following conventions in \cite{lam2004electroid}.
\subsection{Electrical networks and cactus networks}
A \emph{circular planar electrical network}, is a finite weighted undirected graph $\Gamma$ embedded into a disk, with boundary vertices labeled $\bar1,\bar2,\ldots,\bar n$ in clockwise order. From now on, we will simply refer to circular planar electrical networks as \emph{electrical networks}.

\emph{Cactus networks} are introduced by the second author \cite{lam2004electroid} in order to study the compactification of the space of electrical networks. Let $S$ be a circle with boundary vertices $\bar1,\ldots,\bar n$, and let $\zeta$ be a noncrossing partition of $[\bar n]$. Identifying the boundary vertices according to the parts of $\zeta$ gives a \emph{hollow cactus} $S_{\zeta}$, which can be viewed as a union of circles glued at some points. A \emph{cactus} is $S_{\zeta}$ together with its interior and a \emph{cactus network} $\Gamma$ is a weighted graph embedded into a cactus, which can be intuitively thought of as an electrical network where boundary vertices are identified in a manner given by $\zeta$. In particular, an electrical network is a cactus network with $\zeta=(\bar1|\bar2|\cdots|\bar{n})$. See \cref{fig:cactus-example} for an example.

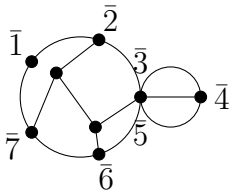
\begin{figure}[h!]
\centering
\begin{tikzpicture}[scale=0.4]
\draw (0,0) circle (2);
\node at (69:2.8) {$\bar2$};
\node at (140:2.8) {$\bar1$};
\node at (216:2.8) {$\bar7$};
\node at (288:2.8) {$\bar6$};
\node at (2,1.2) {$\bar3$};
\node at (2,-1.2) {$\bar5$};
\draw (3,0) circle (1);
\node at (4.7,0) {$\bar4$};
\node[circle, draw=black, fill=black, minimum size=1, scale=0.4,] (35) at (0:2) {};
\node[circle, draw=black, fill=black, minimum size=1, scale=0.4,] (2) at (72:2) {};
\node[circle, draw=black, fill=black, minimum size=1, scale=0.4,] (1) at (144:2) {};
\node[circle, draw=black, fill=black, minimum size=1, scale=0.4,] (7) at (216:2) {};
\node[circle, draw=black, fill=black, minimum size=1, scale=0.4,] (6) at (288:2) {};
\node[circle, draw=black, fill=black, minimum size=1, scale=0.4,] (4) at (4,0) {};
\node[circle, draw=black, fill=black, minimum size=1, scale=0.4,] (x) at (-0.8,0.8) {};
\node[circle, draw=black, fill=black, minimum size=1, scale=0.4,] (y) at (0.5,-1) {};
\draw (2) -- (x) -- (7);
\draw (x) -- (y) -- (6);
\draw (y) -- (35) -- (4);
\end{tikzpicture}
\caption{An example of a cactus network for $\zeta=(\bar1|\bar2|\bar3,\bar5|\bar4|\bar6|\bar7)$}
\label{fig:cactus-example}
\end{figure}

We can define a weight function $\wt$ on the edges in a network $\Gamma$ where we typically treat each value $\wt(e)$ as an indeterminate. For a subgraph $F\subset\Gamma$, define \[\wt(F)=\prod_{e\in F}\wt(e).\] 
We call the underlying unweighted graph of a cactus network a \emph{cactus graph}.

\subsection{Groves}
A \emph{grove} $F$ on a cactus network $\Gamma$ is a spanning forest, such that each component is connected to the boundary. The \emph{boundary partition} $\sigma(F)$ is a set partition of $[\bar n]$ that records which boundary vertices are in the same connected components of $F$. Note that since our network is planar, $\sigma(F)$ must be a noncrossing partition of $[\bar n]$ that coarsens $\zeta$, if $\Gamma$ lies in the cactus $S_{\zeta}$. 
For a noncrossing partition $\sigma$, we define the \emph{grove measurement} as
\[L_{\sigma}(\Gamma):=\sum_{F\:|\: \sigma(F)=\sigma}\wt(F)\]
where the summation is over all groves with boundary partition $\sigma$. 

\subsection{Medial graph and medial pairing}\label{sub:medial}
Given a cactus network $\Gamma$, we define its \emph{medial graph} $G(\Gamma)$ as follows. The vertices of $G(\Gamma)$ consist of vertices $t_1,t_2,\ldots,t_{2n}$ on the boundary in clockwise order such that the orginal boundary vertex $\bar i$ lies between $t_{2i-1}$ and $t_{2i}$, and vertices $t_e$'s for each edge $e$ in the network. The edges of $G(\Gamma)$ are joined as follows: we first join $t_e$ with $t_{e'}$ if $e$ and $e'$ share a vertex of $\Gamma$ and are incident to the same face. Then for each boundary vertex $\bar i$ of $\Gamma$, let the edges incident to $\bar i$ be $e_1,\ldots,e_k$ in counterclockwise order, i.e. $t_{2i-1}$ is closest to $e_1$ while $t_{2i}$ is closest to $e_{k}$, and join $t_{2i-1}$ with $t_{e_1}$ and $t_{2i}$ with $t_{e_k}$; in the case of $k=0$, i.e. there are no edges incident to $\bar i$ in $\Gamma$, we join $t_{2i-1}$ with $t_{2i}$ instead. Note that in $G(\Gamma)$, each $t_e$ has degree $4$ for $e\in E(\Gamma)$, and each $t_i$ has degree $1$ for $i\in[2n]$.

We then define the \emph{medial pairing} $\tau(\Gamma)$ on a medial graph $G(\Gamma)$. From a boundary vertex $t_i$, which has degree $1$ of $G(\Gamma)$, we trace through edges such that whenever we encounter a degree $4$ vertex $t_e$, we go straight through, ending at another boundary vertex $t_j$. This procedure forms $n$ \emph{strands} or \emph{wires} which naturally results in a matching $\tau(\Gamma)$ of $[2n]$. See \cref{fig:medial-example-01} for an example.
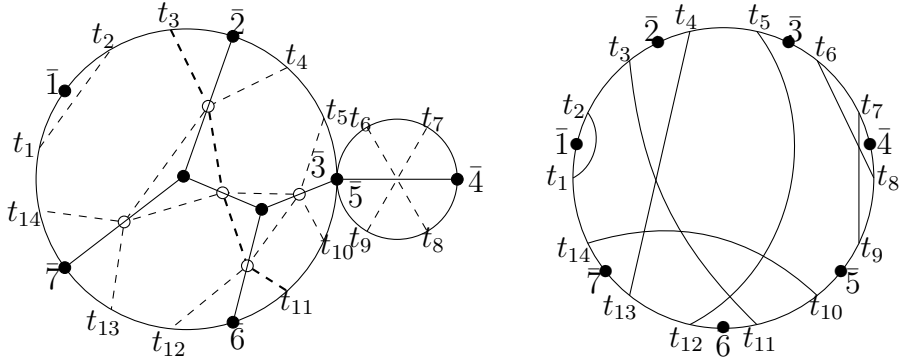
\begin{figure}[h!]
\centering
\begin{tikzpicture}[scale=0.4]
\def\r{5};
\def\rr{2};
\def\xa{-0.1};
\def\ya{0.1};
\def\xb{2.5};
\def\yb{-1.0};
\def\h{0.5};
\draw (0,0) circle (\r);
\draw (\r+\rr,0) circle (\rr);
\node[circle, draw=black, fill=black, minimum size=1, scale=0.4,] (35) at (\r,0) {};
\node[circle, draw=black, fill=black, minimum size=1, scale=0.4,] (2) at (\r*0.30901699437,\r*0.95105651629) {};
\node[circle, draw=black, fill=black, minimum size=1, scale=0.4,] (1) at (-\r*0.80901699437,\r*0.58778525229) {};
\node[circle, draw=black, fill=black, minimum size=1, scale=0.4,] (7) at (-\r*0.80901699437,-0.58778525229*\r) {};
\node[circle, draw=black, fill=black, minimum size=1, scale=0.4,] (6) at (0.30901699437*\r,-0.95105651629*\r) {};
\node[circle, draw=black, fill=black, minimum size=1, scale=0.4,] (4) at (\r+2*\rr,0) {};

\coordinate (t5) at (24:\r) {};
\node at (24:\r+\h) {$t_5$};
\coordinate (t4) at (48:\r) {};
\node at (48:\r+\h) {$t_4$};
\node at (72:\r+\h) {$\bar2$};
\coordinate (t3) at (96:\r) {};
\node at (96:\r+\h) {$t_3$};
\coordinate (t2) at (120:\r) {};
\node at (120:\r+\h) {$t_2$};
\node at (144:\r+\h) {$\bar1$};
\coordinate (t1) at (168:\r) {};
\node at (168:\r+\h) {$t_1$};
\coordinate (t14) at (192:\r) {};
\node at (192:\r+\h) {$t_{14}$};
\node at (216:\r+\h) {$\bar7$};
\coordinate (t13) at (240:\r) {};
\node at (240:\r+\h) {$t_{13}$};
\coordinate (t12) at (264:\r) {};
\node at (264:\r+\h) {$t_{12}$};
\node at (288:\r+\h) {$\bar6$};
\coordinate (t11) at (312:\r) {};
\node at (312:\r+\h) {$t_{11}$};
\coordinate (t10) at (336:\r) {};
\node at (336:\r+\h) {$t_{10}$};

\node[left] at (\r,\h) {$\bar3$};
\node[right] at (\r,-\h) {$\bar5$};
\node at (\r+0.4*\rr,\rr) {$t_6$};
\node at (\r+0.4*\rr,-\rr) {$t_9$};
\node[right] at (\r+2*\rr,0) {$\bar4$};
\node at (\r+1.6*\rr,\rr) {$t_7$};
\node at (\r+1.6*\rr,-\rr) {$t_8$};

\node[circle, draw=black, fill=black, minimum size=1, scale=0.4,] (a) at (\xa,\ya) {};
\node[circle, draw=black, fill=black, minimum size=1, scale=0.4,] (b) at (\xb,\yb) {};

\node[circle, draw=black, fill=white, minimum size=1, scale=0.4,] (te1) at (\r*0.30901699437/2+\xa/2,\r*0.95105651629/2+\ya/2) {};
\node[circle, draw=black, fill=white, minimum size=1, scale=0.4,] (te2) at (-\r*0.80901699437/2+\xa/2,-0.58778525229*\r/2+\ya/2) {};
\node[circle, draw=black, fill=white, minimum size=1, scale=0.4,] (te3) at (\xa/2+\xb/2,\ya/2+\yb/2) {};
\node[circle, draw=black, fill=white, minimum size=1, scale=0.4,] (te4) at (\r/2+\xb/2,\yb/2) {};
\node[circle, draw=black, fill=white, minimum size=1, scale=0.4,] (te5) at (0.30901699437*\r/2+\xb/2,-0.95105651629*\r/2+\yb/2) {};

\draw(2)--(a)--(7);
\draw(a)--(b)--(6);
\draw(b)--(35)--(4);

\draw[dashed](t1)--(t2);
\draw[dashed,thick](t3)--(te1)--(te3)--(te5)--(t11);
\draw[dashed](t4)--(te1)--(te2)--(t13);
\draw[dashed](t5)--(te4)--(te5)--(t12);
\draw[dashed](t10)--(te4)--(te3)--(te2)--(t14);
\draw[dashed](\r+0.5*\rr,0.86602540378*\rr)--(\r+1.5*\rr,-0.86602540378*\rr);
\draw[dashed](\r+0.5*\rr,-0.86602540378*\rr)--(\r+1.5*\rr,0.86602540378*\rr);
\end{tikzpicture}
\quad
\begin{tikzpicture}[scale=0.4]
\def\r{5};
\def\h{0.5};
\draw (0,0) circle (\r);
\coordinate (t8) at (0:\r) {};
\coordinate (t7) at (360/14:\r) {};
\coordinate (t6) at (2*360/14:\r) {};
\coordinate (t5) at (3*360/14:\r) {};
\coordinate (t4) at (4*360/14:\r) {};
\coordinate (t3) at (5*360/14:\r) {};
\coordinate (t2) at (6*360/14:\r) {};
\coordinate (t1) at (7*360/14:\r) {};
\coordinate (t14) at (8*360/14:\r) {};
\coordinate (t13) at (9*360/14:\r) {};
\coordinate (t12) at (10*360/14:\r) {};
\coordinate (t11) at (11*360/14:\r) {};
\coordinate (t10) at (12*360/14:\r) {};
\coordinate (t9) at (13*360/14:\r) {};
\node at (0.5*360/14:\r) {$\bullet$};
\node at (2.5*360/14:\r) {$\bullet$};
\node at (4.5*360/14:\r) {$\bullet$};
\node at (6.5*360/14:\r) {$\bullet$};
\node at (8.5*360/14:\r) {$\bullet$};
\node at (10.5*360/14:\r) {$\bullet$};
\node at (12.5*360/14:\r) {$\bullet$};
\node at (0.5*360/14:\r+\h) {$\bar4$};
\node at (2.5*360/14:\r+\h) {$\bar3$};
\node at (4.5*360/14:\r+\h) {$\bar2$};
\node at (6.5*360/14:\r+\h) {$\bar1$};
\node at (8.5*360/14:\r+\h) {$\bar7$};
\node at (10.5*360/14:\r+\h) {$\bar6$};
\node at (12.5*360/14:\r+\h) {$\bar5$};
\node at (0:\r+\h) {$t_8$};
\node at (360/14:\r+\h) {$t_7$};
\node at (2*360/14:\r+\h) {$t_6$};
\node at (3*360/14:\r+\h) {$t_5$};
\node at (4*360/14:\r+\h) {$t_4$};
\node at (5*360/14:\r+\h) {$t_3$};
\node at (6*360/14:\r+\h) {$t_2$};
\node at (7*360/14:\r+\h) {$t_1$};
\node at (8*360/14:\r+\h) {$t_{14}$};
\node at (9*360/14:\r+\h) {$t_{13}$};
\node at (10*360/14:\r+\h) {$t_{12}$};
\node at (11*360/14:\r+\h) {$t_{11}$};
\node at (12*360/14:\r+\h) {$t_{10}$};
\node at (13*360/14:\r+\h) {$t_9$};
\draw[bend right=50] (t1) to (t2);
\draw[bend right=20] (t3) to (t11);
\draw(t4)--(t13);
\draw[bend left=50] (t5) to (t12);
\draw(t6)--(t8);
\draw(t7)--(t9);
\draw[bend right=30] (t10) to (t14);
\end{tikzpicture}
\caption{A cactus network $\Gamma$ with its medial graph $G(\Gamma)$ (left) and medial pairing $\tau(\Gamma)=\{(1,2),(3,11),(4,13),(5,12),(6,8),(7,9),(10,14)\}$ (right)}
\label{fig:medial-example-01}
\end{figure}

For later sections, it might be more convenient to think about medial graphs and medial pairings to be on a circle instead of on a cactus. To reconstruct $\Gamma$ from its medial graph $G(\Gamma)$, either drawn on a cactus or drawn on a circle, we first place a boundary vertex $\bar i$ between $t_{2i-1}$ and $t_{2i}$ for $i\in[n]$. Note that the graph $G$ divides the interior of the disk or the cactus into regions. We color the regions with black and white so that adjacent regions have different colors, and that the regions with boundary vertices $\bar i$'s are colored white. Contract the boundary vertices inside the same regions to form a cactus. Place a vertex in each white interior region and connect two vertices if the two corresponding regions in $G$ share a vertex. 

A medial graph is \emph{lensless} if no wire has a self intersection, and no two wires cross twice, i.e. there are no lens that look like \begin{tikzpicture}[scale=0.5]
\draw[bend right=50](-1,0.2)to(1,0.2);\draw[bend left=50](-1,-0.2)to(1,-0.2);
\end{tikzpicture}. A cactus network $\Gamma$ is \emph{critical}, or \emph{reduced}, if its medial graph is lensless. In this case, the number of crossings in the medial pairing $\tau(\Gamma)$ equals the number of edges in $\Gamma$. We note that for any matching $\tau$ on $[2n]$, there exists a reduced cactus network $\Gamma$ such that $\tau(\Gamma)=\tau$, by drawing $\tau$ in a reduced way on a disk, viewing it as a medial graph and applying the procedure above. More specifically, we recall the following \cref{prop:Y-delta} from \cite{lam2004electroid}.

\begin{prop}[Proposition 2.12 of \cite{lam2004electroid}]\label{prop:Y-delta}
We have:
\begin{enumerate}
\item any matching on $[2n]$ can be realized as a medial pairing of some reduced cactus graph;
\item any two reduced cactus graphs with the same medial pairing can be obtained from each other via $Y-\Delta$ moves.
\end{enumerate}
\end{prop}

Given a cactus network $\Gamma$ with its medial graph $G(\Gamma)$, we define its \emph{dual network} $\Gamma^\vee$ as follows: we keep only the information of $G(\Gamma)$, put $\bar i$ between $t_{2i}$ and $t_{2i+1}$ in clockwise order where the indices are taken modulo $2n$, and recover a cactus network $\Gamma^{\vee}$ via the procedure described above, called the \emph{dual network} of $\Gamma$. See \cref{fig:dual-network-ex1} for an example. If we take the dual twice, $(\Gamma^{\vee})^{\vee}$ will be $\Gamma$ with indices shifted by $1$ in the clockwise order.
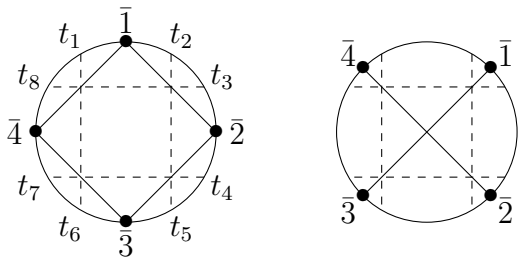
\begin{figure}[h!]
\centering
\begin{tikzpicture}[scale=0.4]
\def\r{3};
\def\x{10};
\draw (0,0) circle (\r);
\node at (\r,0) {$\bullet$};
\node at (-\r,0) {$\bullet$};
\node at (0,\r) {$\bullet$};
\node at (0,-\r) {$\bullet$};
\node at (\r+0.7,0) {$\bar2$};
\node at (-\r-0.7,0) {$\bar4$};
\node at (0,\r+0.7) {$\bar1$};
\node at (0,-\r-0.7) {$\bar3$};
\node at (30:\r+0.7) {$t_3$};
\node at (60:\r+0.7) {$t_2$};
\node at (120:\r+0.7) {$t_1$};
\node at (150:\r+0.7) {$t_8$};
\node at (210:\r+0.7) {$t_7$};
\node at (240:\r+0.7) {$t_6$};
\node at (300:\r+0.7) {$t_5$};
\node at (330:\r+0.7) {$t_4$};
\draw(\r,0)--(0,\r)--(-\r,0)--(0,-\r)--(\r,0);
\draw[dashed](30:\r)--(150:\r);
\draw[dashed](330:\r)--(210:\r);
\draw[dashed](60:\r)--(300:\r);
\draw[dashed](120:\r)--(240:\r);

\draw (\x,0) circle (\r);
\node at (\x+1.41421356237*\r/2,1.41421356237*\r/2) {$\bullet$};
\node at (\x+1.41421356237*\r/2,-1.41421356237*\r/2) {$\bullet$};
\node at (\x-1.41421356237*\r/2,1.41421356237*\r/2) {$\bullet$};
\node at (\x-1.41421356237*\r/2,-1.41421356237*\r/2) {$\bullet$};
\node at (\x+0.70710678118*3.7,0.70710678118*3.7) {$\bar1$};
\node at (\x+0.70710678118*3.7,-0.70710678118*3.7) {$\bar2$};
\node at (\x-0.70710678118*3.7,0.70710678118*3.7) {$\bar4$};
\node at (\x-0.70710678118*3.7,-0.70710678118*3.7) {$\bar3$};
\draw(\x+1.41421356237*\r/2,1.41421356237*\r/2)--(\x-1.41421356237*\r/2,-1.41421356237*\r/2);
\draw(\x+1.41421356237*\r/2,-1.41421356237*\r/2)--(\x-1.41421356237*\r/2,1.41421356237*\r/2);
\draw[dashed](\x-0.5*\r,0.86602540378*\r)--(\x-0.5*\r,-0.86602540378*\r);
\draw[dashed](\x+0.5*\r,0.86602540378*\r)--(\x+0.5*\r,-0.86602540378*\r);
\draw[dashed](\x+0.86602540378*\r,0.5*\r)--(\x-0.86602540378*\r,0.5*\r);
\draw[dashed](\x+0.86602540378*\r,-0.5*\r)--(\x-0.86602540378*\r,-0.5*\r);
\end{tikzpicture}
\caption{A network $\Gamma$ and its dual $\Gamma^{\vee}$}
\label{fig:dual-network-ex1}
\end{figure}

The dual graphs $\Gamma$ and $\Gamma^{\vee}$ have the same number of edges.  Furthermore, each edge of $\Gamma$ intersects a unique edge of $\Gamma^\vee$ (and vice-versa), inducing a natural bijection between the edges of $\Gamma$ and of $\Gamma^\vee$. Each pair of intersecting edges corresponds to an intersection in the medial graph $G(\Gamma)$. 

\subsection{Catalan objects and their bijections}\label{sec:catalan}
We need to use all of the following objects: Dyck paths, noncrossing matchings and noncrossing partitions, all of which are enumerated by the Catalan numbers $C_n$. For easier reference, we summarize the notations we use in the following:
\begin{itemize}
    \item[-]  $\mathcal{C}_n$:= the set of Dyck paths of semilength $n$.
    \item[-]  $\MM_n$:= the set of matchings on $1,2,\ldots,2n$.
    \item[-]  $\NCM_n$:= the set of noncrossing matchings on $1,2,\ldots,2n$.
    \item[-]  $\TNC_n$:= the set of $3$-noncrossing matchings on $1,2,\ldots,2n$.
    \item[-]  $\NP_n$:= the set of noncrossing partitions of $[\bar{n}] = \{\bar{1}, \bar{2},\dots, \bar{n}\}$.
\end{itemize}
Now we start our discussion on these Catalan objects.
\begin{defin}
A \textit{Dyck path} of semilength $n$ is a walk from $(0,0)$ to $(2n,0)$ consisting of $n$ \textit{upsteps} $(1,1)$ and $n$ \textit{downsteps} $(1,-1)$ that stays above the $x$-axis. 
\end{defin}
We denote the set of Dyck paths of semilength $n$ by $\mathcal{C}_n$. There is a natural poset structure on Dyck paths: for two Dyck paths of semilength $n$, we say that $P\leq Q$ if the path $P$ stays weakly below $Q$. We can also represent a Dyck path $P$ by a vector $v(P)\in\{0,1\}^{2n}$ such that $v(P)_j=1$ if the $j$th step of $P$ is an upstep and $v(P)_j=0$ if the $j$th step of $P$ is a downstep. Then there are always as many $1$'s as $0$'s within $v(P)_1,\ldots,v(P)_j$ for all $j=1,\ldots,2n$. We write $P\preccurlyeq Q$ if $v(P)$ is lexicographically (weakly) smaller than $v(Q)$. Note that $\preccurlyeq$ is a total order and that if $P\leq Q$ then $P\preccurlyeq Q$. 

For a positive integer $k$, let \[\mathcal{C}_n^{(k)}=\{P_1\leq P_2\leq\cdots\leq P_k\}\subset \mathcal{C}_n^{k}\]
be the set of chains of length $k$ in the poset on Dyck paths of semilength $n$. In particular, $\mathcal{C}_n^{(1)}=\mathcal{C}_n$ and $\mathcal{C}_n^{(2)}$ is the set of pairs of comparable Dyck paths $(P\leq Q)$. 

Another important family of combinatorial objects that we need is \emph{matchings}. Denote the set of matchings on $1,2,\ldots,2n$ by $\MM_n$. For a pair $(i,j)$ in a matching $M$ with $i<j$, we say that $i$ is the \emph{left endpoint} and $j$ is the \emph{right endpoint}. Let $\CR(M)$ denote the set of crossings of $M$, given by
\[\CR(M)=\{(a,b,c,d)\:|\: a<b<c<d,\ (a,c),(b,d)\in M\}.\] A \emph{strand diagram} of a matching $M\in\MM_n$ is a drawing of $M$ by $n$ arcs, called \emph{strands} or \emph{wires}, either on a disk or a straight line such that no two strands cross twice. We consider the strand diagrams up to homotopy so for example, if $M=\{(1,2),(3,4),\ldots,(2n-1,2n)\}$, then there is a unique strand diagram of $M$. 
\begin{defin}
A matching $M$ of $1,2,\ldots,2n$ is called $k$-\textit{noncrossing} if there do not exist $k$ pairs in $M$ that cross pairwise. Denote the set of $2$-noncrossing matchings, i.e. noncrossing matchings by $\NCM_n$ and the set of $3$-noncrossing matchings by $\TNC_n$.
\end{defin}
We now describe a bijection between Dyck paths $\mathcal{C}_n$ and noncrossing matchings $\NCM_n$. For a Dyck path $P\in\mathcal{C}_n$, label its steps (edges) by $1,2,\ldots,2n$ from left to right and pair up the upsteps with downsteps of the same height with no steps in between to obtain a noncrossing matching $M\in\NCM_n$.  For the inverse, given a noncrossing matching $M$, construct a Dyck path $P$ so that the $i$-th step is a upstep if $i$ is a left endpoint in $M$, and a downstep if $i$ is a right endpoint in $M$. See \cref{fig:bijection-Dyck-NCM} for an example.
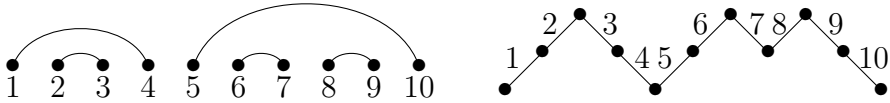
\begin{figure}[h!]
\centering
\begin{tikzpicture}[scale=0.6]
\node at (1,0) {$\bullet$};
\node at (2,0) {$\bullet$};
\node at (3,0) {$\bullet$};
\node at (4,0) {$\bullet$};
\node at (5,0) {$\bullet$};
\node at (6,0) {$\bullet$};
\node at (7,0) {$\bullet$};
\node at (8,0) {$\bullet$};
\node at (9,0) {$\bullet$};
\node at (10,0) {$\bullet$};
\node[below] at (1,0) {$1$};
\node[below] at (2,0) {$2$};
\node[below] at (3,0) {$3$};
\node[below] at (4,0) {$4$};
\node[below] at (5,0) {$5$};
\node[below] at (6,0) {$6$};
\node[below] at (7,0) {$7$};
\node[below] at (8,0) {$8$};
\node[below] at (9,0) {$9$};
\node[below] at (10,0) {$10$};
\draw[bend left=70] (1,0) to (4,0);
\draw[bend left=70] (2,0) to (3,0);
\draw[bend left=70] (5,0) to (10,0);
\draw[bend left=70] (8,0) to (9,0);
\draw[bend left=70] (6,0) to (7,0);
\end{tikzpicture}
\quad
\begin{tikzpicture}[scale=0.5]
\node at (0,0) {$\bullet$};
\node at (1,1) {$\bullet$};
\node at (2,2) {$\bullet$};
\node at (3,1) {$\bullet$};
\node at (4,0) {$\bullet$};
\node at (5,1) {$\bullet$};
\node at (6,2) {$\bullet$};
\node at (7,1) {$\bullet$};
\node at (8,2) {$\bullet$};
\node at (9,1) {$\bullet$};
\node at (10,0) {$\bullet$};
\draw(0,0)--(2,2)--(4,0)--(6,2)--(7,1)--(8,2)--(10,0);
\node[above] at (0.2, 0.3) {$1$};
\node[above] at (1.2, 1.2) {$2$};
\node[above] at (2.8,1.2) {$3$};
\node[above] at (3.65,0.3) {$4$};
\node[above] at (4.25,0.3) {$5$};
\node[above] at (5.2,1.2) {$6$};
\node[above] at (6.7,1.2) {$7$};
\node[above] at (7.3,1.2) {$8$};
\node[above] at (8.8,1.2) {$9$};
\node[above] at (9.8,0.3) {$10$};

\end{tikzpicture}
\caption{Example of a Dyck path of semilength 5 and its corresponding noncrossing matching under the bijection described above}
\label{fig:bijection-Dyck-NCM}
\end{figure}

Th cover relations in the poset structure on Dyck paths can be described as follows: $P\lessdot P'$ if $P'$ can be obtained from $P$ by changing an upstep followed by a downstep, to a downstep followed by an upstep. The poset structure on noncrossing matchings is inherited from that on Dyck paths. And the cover relations can be described similarly: $M\lessdot M'$ if we can change two pairs $(a,b)$ and $(b+1,c)$ in $M$ to $(a,c)$ and $(b,b+1)$ in $M'$, where $a<b<b+1<c$. 

Our third Catalan object is the set of noncrossing partitions. 
\begin{defin}
A partition $\sigma$ of $[\bar n]=\{\bar1,\bar2,\ldots,\bar n\}$ is called \emph{noncrossing} if there do not exist $a<b<c<d$ such that $\bar a$ and $\bar c$ are in one part of $\sigma$ while $\bar b$ and $\bar d$ are in another. Denote the set of noncrossing partitions as $\NP_n$. 
\end{defin}
The natural bijections between $\NP_n$ and $\NCM_n$ and $\mathcal{C}_n$ are also easy to describe. 

We first start by describing the bijection between $\NCM_n$ and $\NP_n$. Let $M\in\NCM_n$ be a noncrossing matching on $[2n]$. We draw the strand diagram $M$ either on a circle or on a straight line, dividing into smaller regions, put $\bar i$ between $2i-1$ and $2i$, for $i=1,\ldots,n$.  We obtain a noncrossing partition $\sigma$ of $[\bar n]$ whose parts consist of the points $\bar i$ in the same region.  The inverse bijection from $\NP_n$ to $\NCM_n$ is more conveniently described inductively. Let $\sigma$ be a noncrossing partition, then $\sigma$ must contain a part of the form $(a+1,\ldots,a+k)$ for some $a\geq0$ and $k\geq1$. We remove this part and add a partial matching by connecting $2a+1$ with $2a+2k$, $2a+2$ with $2a+3$, $2a+4$ with $2a+5$ and so on up to $2a+2k-2$ with $2a+2k-1$. We repeat this process until we obtain a noncrossing matching.

Now we describe the bijection between $\mathcal{C}_n$ and $\NP_n$. Let $P\in\mathcal{C}_n$ be a Dyck path. There are $2n+1$ lattice points in $P$, including $(0,0)$ and $(2n,0)$. Label the $2i$-th lattice point in $P$ by $\bar i$, for $i=1,\ldots,n$, so that $(1,1)$ is labeled by $\bar 1$. Then we identify $\bar i$ and $\bar j$ to be in the same part of the noncrossing partition $\sigma$ if they are on the same horizontal line, not separated by the Dyck path $P$. To go the other way, we apply the following recursive procedure. Let $\sigma$ be a noncrossing partition and let the part containing $1$ be $\{1=a_0<a_1<\cdots<a_k=m\}$ with $k\geq0$. This means that each part of $\sigma$ either has all entries at most $m$, or has all entries at least $m+1$. We deal with these two kinds separately and attach the two resulting Dyck paths at coordinate $(2m,0)$. Thus assume $m=n$. Similarly, all other parts of $\sigma$ have their entries strictly between $a_i$ and $a_{i+1}$, and let $P^{(i)}$ be the Dyck path obtained from such $i$, after adjusting the entries from $\{a_i+1,\ldots,a_{i+1}-1\}$ to $\{1,\ldots,a_{i+1}-a_i-1\}$. Attach each $P^{(i)}$ from $(2a_i+2,2)$ to $(2a_{i+1}-2,2)$ and we have the desired Dyck path $P\in\mathcal{C}_n$.

See \cref{fig:bijection-NP-Dyck-NCM} for an illustration of the bijections between $\NP_n, \NCM_n$ and $\mathcal{C}_n$.
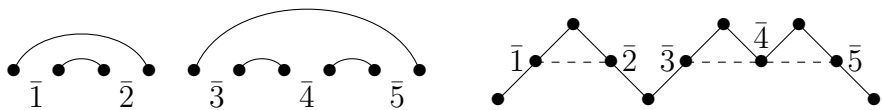
\begin{figure}[h!]
\centering
\begin{tikzpicture}[scale=0.6]
\node at (1,0) {$\bullet$};
\node at (2,0) {$\bullet$};
\node at (3,0) {$\bullet$};
\node at (4,0) {$\bullet$};
\node at (5,0) {$\bullet$};
\node at (6,0) {$\bullet$};
\node at (7,0) {$\bullet$};
\node at (8,0) {$\bullet$};
\node at (9,0) {$\bullet$};
\node at (10,0) {$\bullet$};
\node[below] at (1.5,0) {$\bar1$};
\node[below] at (3.5,0) {$\bar2$};
\node[below] at (5.5,0) {$\bar3$};
\node[below] at (7.5,0) {$\bar4$};
\node[below] at (9.5,0) {$\bar5$};
\draw[bend left=70] (1,0) to (4,0);
\draw[bend left=70] (2,0) to (3,0);
\draw[bend left=70] (5,0) to (10,0);
\draw[bend left=70] (8,0) to (9,0);
\draw[bend left=70] (6,0) to (7,0);
\end{tikzpicture}
\quad
\begin{tikzpicture}[scale=0.5]
\node at (0,0) {$\bullet$};
\node at (1,1) {$\bullet$};
\node at (2,2) {$\bullet$};
\node at (3,1) {$\bullet$};
\node at (4,0) {$\bullet$};
\node at (5,1) {$\bullet$};
\node at (6,2) {$\bullet$};
\node at (7,1) {$\bullet$};
\node at (8,2) {$\bullet$};
\node at (9,1) {$\bullet$};
\node at (10,0) {$\bullet$};
\draw(0,0)--(2,2)--(4,0)--(6,2)--(7,1)--(8,2)--(10,0);
\node[left] at (1,1.1) {$\bar1$};
\node[right] at (3,1.1) {$\bar2$};
\node[left] at (5,1.1) {$\bar3$};
\node[above] at (7,1.1) {$\bar4$};
\node[right] at (9,1.1) {$\bar5$};
\draw[dashed](1,1)--(3,1);
\draw[dashed](5,1)--(9,1);
\end{tikzpicture}
\caption{The noncrossing matching (left) and the Dyck path (right) in bijection with the noncrossing partition $\sigma=(\bar 1\bar 2|\bar3\bar4\bar5)$}
\label{fig:bijection-NP-Dyck-NCM}
\end{figure}

For a noncrossing partition $\sigma$ on $[\bar n]$, one has a \emph{dual noncrossing} partition $\tilde\sigma$ on $[\tilde n]$ obtained by drawing $\sigma$ on a disk, connecting boundary vertices in the same parts, putting $\tilde i$ between $\bar i$ and $\overline{i{+}1}$ and identifying the partition $\tilde\sigma$ as the regions. See \cref{fig:dual-partition-ex1}. We also often denote $\sigma$ as a partition on $1,3,\ldots,2n-1$ and $\tilde\sigma$ as a partition on $2,4,\ldots,2n$.
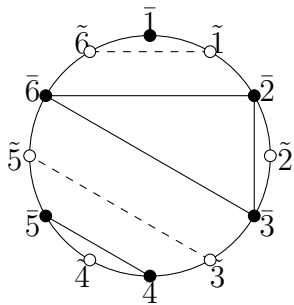
\begin{figure}[h!]
\centering
\begin{tikzpicture}[scale=0.4]
\def\r{4};
\draw (0,0) circle (\r);
\coordinate (b1) at (90:\r);
\coordinate (b2) at (30:\r);
\coordinate (b3) at (330:\r);
\coordinate (b4) at (270:\r);
\coordinate (b5) at (210:\r);
\coordinate (b6) at (150:\r);
\coordinate (t1) at (60:\r);
\coordinate (t2) at (0:\r);
\coordinate (t3) at (300:\r);
\coordinate (t4) at (240:\r);
\coordinate (t5) at (180:\r);
\coordinate (t6) at (120:\r);
\node at (90:\r+0.5) {$\bar1$};
\node at (30:\r+0.5) {$\bar2$};
\node at (330:\r+0.5) {$\bar3$};
\node at (270:\r+0.5) {$\bar4$};
\node at (210:\r+0.5) {$\bar5$};
\node at (150:\r+0.5) {$\bar6$};
\node at (60:\r+0.5) {$\tilde1$};
\node at (0:\r+0.5) {$\tilde2$};
\node at (300:\r+0.5) {$\tilde3$};
\node at (240:\r+0.5) {$\tilde4$};
\node at (180:\r+0.5) {$\tilde5$};
\node at (120:\r+0.5) {$\tilde6$};
\draw(b2)--(b3)--(b6)--(b2);
\draw(b4)--(b5);
\draw[dashed](t1)to(t6);
\draw[dashed](t3)to(t5);
\node[circle, draw=black, fill=black, minimum size=1, scale=0.4,] at (b1) {};
\node[circle, draw=black, fill=black, minimum size=1, scale=0.4,] at (b2) {};
\node[circle, draw=black, fill=black, minimum size=1, scale=0.4,] at (b3) {};
\node[circle, draw=black, fill=black, minimum size=1, scale=0.4,] at (b4) {};
\node[circle, draw=black, fill=black, minimum size=1, scale=0.4,] at (b5) {};
\node[circle, draw=black, fill=black, minimum size=1, scale=0.4,] at (b6) {};
\node[circle, draw=black, fill=white, minimum size=1, scale=0.4,] at (t1) {};
\node[circle, draw=black, fill=white, minimum size=1, scale=0.4,] at (t2) {};
\node[circle, draw=black, fill=white, minimum size=1, scale=0.4,] at (t3) {};
\node[circle, draw=black, fill=white, minimum size=1, scale=0.4,] at (t4) {};
\node[circle, draw=black, fill=white, minimum size=1, scale=0.4,] at (t5) {};
\node[circle, draw=black, fill=white, minimum size=1, scale=0.4,] at (t6) {};
\end{tikzpicture}
\caption{A noncrossing partition $\sigma=(\bar1|\bar2\bar3\bar6|\bar4\bar5)$ with its dual $\tilde\sigma=(\tilde1\tilde6|\tilde2|\tilde3\tilde5|\tilde4)$}
\label{fig:dual-partition-ex1}
\end{figure}
It is easy to obtain that $|\sigma|+|\tilde\sigma|=n+1$, where $|\sigma|$ is the number of parts of $\sigma$ (Lemma 2.2 of \cite{lam2004electroid}).

As for notations, we typically use $P,Q$ for Dyck paths, $\tau,M,\xi$ for matchings, and $\sigma$ for noncrossing partitions.  Also, denote these bijections by $\sigma:\mathcal{C}_n\rightarrow \NP_n$, $\sigma:\NCM_n\rightarrow \NP_n$, $\tau:\mathcal{C}_n\rightarrow \NCM_n$, $\tau:\NP_n\rightarrow\NCM_n$. 

\section{Combinatorics of $3$-noncrossing matchings}\label{sec:combinatorics-3noncrossing}

\subsection{3-noncrossing matchings and pairs of comparable Dyck paths}
In this section, we describe a bijection between $3$-noncrossing matchings $\TNC_n$ of $n$ and pairs of comparable Dyck paths $\mathcal{C}_n^{(2)}$ in two languages, one formulated as in \cite{chen2007crossings} and the other one to be used later. 

We first formulate the construction in Chen et al. \cite{chen2007crossings}. For a matching $M\in\MM_n$, associate a sequence of standard Young tableaux $T_0,T_1,\ldots,T_{2n}$, starting with $T_{2n}=\emptyset$, and for $j<2n$, 
\begin{itemize}
\item if $j$ is the right end point of a pair $(i,j)$ in $M$, insert $i$ into $T_{j+1}$ to obtain $T_j$, in the sense of row insertion in RSK (see for example \cite{ec1});
\item if $j$ is the left end point of a pair $(j,k)$ in $M$, delete $j$ from $T_{j+1}$ to obtain $T_j$. 
\end{itemize}
For a $k$-noncrossing matching $M$, $k\geq2$, Chen et al. \cite{chen2007crossings} showed that the tableaux $T_0,\ldots,T_{2n}$ contain at most $k-1$ rows. Thus, for $M\in\TNC_n$, let the shape of the corresponding standard Young tableau $T_{j}$ be $(x_j\geq y_j)$. Note that the $y_j$'s can be $0$. Define the map $\varphi: \TNC_n\to \mathcal{C}_n^{(2)}$ as $\varphi(M):=(P_1\leq P_2)$, where $P_1$ is the Dyck path that passes through the coordinates $\{(j,x_j-y_j)\}_{j=0}^{2n}$ and $P_2$ is the Dyck path that passes through the coordinates $\{(j,x_j+y_j)\}_{j=0}^{2n}$. See \cref{fig:bijection-chen-3noncrossing-pair} for an example.

\begin{figure}[h!]
\centering
\begin{tikzpicture}[scale=0.6]
\node at (1,0) {$\bullet$};
\node at (2,0) {$\bullet$};
\node at (3,0) {$\bullet$};
\node at (4,0) {$\bullet$};
\node at (5,0) {$\bullet$};
\node at (6,0) {$\bullet$};
\node at (7,0) {$\bullet$};
\node at (8,0) {$\bullet$};
\node at (9,0) {$\bullet$};
\node at (10,0) {$\bullet$};
\node[below] at (1,0) {$1$};
\node[below] at (2,0) {$2$};
\node[below] at (3,0) {$3$};
\node[below] at (4,0) {$4$};
\node[below] at (5,0) {$5$};
\node[below] at (6,0) {$6$};
\node[below] at (7,0) {$7$};
\node[below] at (8,0) {$8$};
\node[below] at (9,0) {$9$};
\node[below] at (10,0) {$10$};
\draw[bend left=50] (1,0) to (8,0);
\draw[bend left=50] (2,0) to (4,0);
\draw[bend left=50] (3,0) to (10,0);
\draw[bend left=50] (5,0) to (9,0);
\draw[bend left=50] (6,0) to (7,0);
\end{tikzpicture}
\quad
\begin{tikzpicture}[scale=0.5]
\def\x{0.5}
\node at (0,0) {$\bullet$};
\node at (1,1) {$\bullet$};
\node at (2,2) {$\bullet$};
\node at (3,3) {$\bullet$};
\node at (4,2) {$\bullet$};
\node at (5,3) {$\bullet$};
\node at (6,4) {$\bullet$};
\node at (7,3) {$\bullet$};
\node at (8,2) {$\bullet$};
\node at (9,1) {$\bullet$};
\node at (10,0) {$\bullet$};
\draw(0,0)--(3,3)--(4,2)--(6,4)--(10,0);
\node at (0,0-\x) {$\bullet$};
\node at (1,1-\x) {$\bullet$};
\node at (2,2-\x) {$\bullet$};
\node at (3,1-\x) {$\bullet$};
\node at (4,0-\x) {$\bullet$};
\node at (5,1-\x) {$\bullet$};
\node at (6,2-\x) {$\bullet$};
\node at (7,1-\x) {$\bullet$};
\node at (8,2-\x) {$\bullet$};
\node at (9,1-\x) {$\bullet$};
\node at (10,0-\x) {$\bullet$};
\draw(0,0-\x)--(2,2-\x)--(4,0-\x)--(6,2-\x)--(7,1-\x)--(8,2-\x)--(10,0-\x);
\end{tikzpicture}
\

\

\begin{tikzpicture}[scale=1.0]
\node (a0) at (0,0) {$\emptyset$};
\node (a1) at (1,0) {$1$};
\node (a2) at (2,0) {$12$};
\node[align=left] (a3) at (3,0) {$12$\\$3$};
\node[align=left] (a4) at (4,0) {$1$\\$3$};
\node[align=left] (a5) at (5,0) {$15$\\$3$};
\node[align=left] (a6) at (6,0) {$156$\\$3$};
\node[align=left] (a7) at (7,0) {$15$\\$3$};
\node (a8) at (8,0) {$35$};
\node (a9) at (9,0) {$3$};
\node (a10) at (10,0) {$\emptyset$};
\draw[->](a0)--(a1);
\draw[->](a1)--(a2);
\draw[->](a2)--(a3);
\draw[->](a3)--(a4);
\draw[->](a4)--(a5);
\draw[->](a5)--(a6);
\draw[->](a6)--(a7);
\draw[->](a7)--(a8);
\draw[->](a8)--(a9);
\draw[->](a9)--(a10);
\end{tikzpicture}
\caption{A bijection between $\TNC_n$ and $\mathcal{C}^{(2)}_n$ by Chen et al. \cite{chen2007crossings}, with the sequence of standard Young tableau $T_0,T_1,\ldots,T_{2n}$ shown below.}
\label{fig:bijection-chen-3noncrossing-pair}
\end{figure}
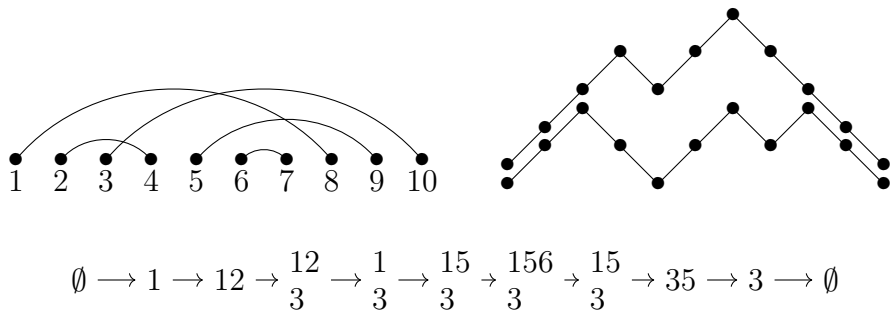

\begin{theorem}[Corollary 5.4 of \cite{chen2007crossings}]
The map $\varphi$ defined above is a bijection between $\TNC_n$ and $\mathcal{C}^{(2)}_n$.
\end{theorem}

We give another description of the bijection $\varphi$ via resolutions of crossings. Recall that the set of crossings of $M\in\MM_n$ is denoted $\CR(M)$. For a vector $v\in\{0,1\}^{\CR(M)}$, define the \textit{resolution of crossings} of $M$ with respect to $v$ to be the noncrossing matching obtained from a planar drawing of $M$ by resolving each crossing $x=(a,b,c,d)\in\CR(M)$ locally by connecting the strand $a$ to $b$ and $c$ to $d$ if $v_x=0$, or connecting $a$ to $d$ and $b$ to $c$ if $v_x=1$. Denote this resolution by $M(v)$, which is a noncrossing matching. We note that in the planar drawing of such a resolution, it is possible for a connected component of strands that do not connect to the vertices $1,2,\ldots,2n$ to appear, and we temporarily allow such resolutions. Details regarding this issue will be discussed in \cref{sec:circular-basis}. Let $\mathbf{0}$ denote the all $0$ vector and $\mathbf{1}$ denote the all $1$ vector. We are particularly interested in the resolution of crossings $M(\mathbf{0})$ and $M(\mathbf{1})$.

Examples of resolutions of crossings are shown in \cref{fig:resolution-crossings}.
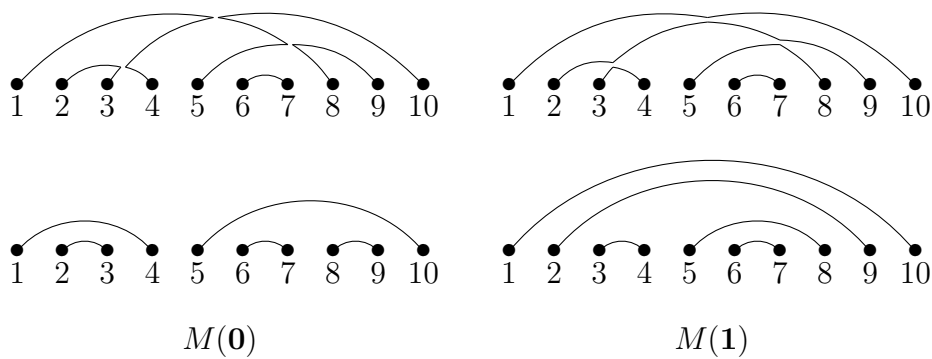
\begin{figure}[h!]
\centering
\begin{tikzpicture}[scale=0.6]
\node at (1,0) {$\bullet$};
\node at (2,0) {$\bullet$};
\node at (3,0) {$\bullet$};
\node at (4,0) {$\bullet$};
\node at (5,0) {$\bullet$};
    \node at (6,0) {$\bullet$};
\node at (7,0) {$\bullet$};
\node at (8,0) {$\bullet$};
\node at (9,0) {$\bullet$};
\node at (10,0) {$\bullet$};
\node[below] at (1,0) {$1$};
\node[below] at (2,0) {$2$};
\node[below] at (3,0) {$3$};
\node[below] at (4,0) {$4$};
\node[below] at (5,0) {$5$};
\node[below] at (6,0) {$6$};
\node[below] at (7,0) {$7$};
\node[below] at (8,0) {$8$};
\node[below] at (9,0) {$9$};
\node[below] at (10,0) {$10$};
\draw[bend left=28] (1,0) to (5.35,1.5);
\draw[bend left=12] (5.45,1.5) to (7,0.9);
\draw[bend left=11] (7.1,0.9) to (8,0);
\draw[bend left=30] (2,0) to (3.3,0.4);
\draw[bend left=20] (3.4,0.4) to (4,0);
\draw[bend left=5] (3,0) to (3.3,0.4);
\draw[bend left=18] (3.4,0.4) to (5.35,1.5);
\draw[bend left=28] (5.45,1.5) to (10,0);
\draw[bend left=25] (5,0) to (7.0,0.9);
\draw[bend left=25] (7.1,0.9) to (9,0);
\draw[bend left=50] (6,0) to (7,0);
\end{tikzpicture}
\quad
\begin{tikzpicture}[scale=0.6]
\node at (1,0) {$\bullet$};
\node at (2,0) {$\bullet$};
\node at (3,0) {$\bullet$};
\node at (4,0) {$\bullet$};
\node at (5,0) {$\bullet$};
\node at (6,0) {$\bullet$};
\node at (7,0) {$\bullet$};
\node at (8,0) {$\bullet$};
\node at (9,0) {$\bullet$};
\node at (10,0) {$\bullet$};
\node[below] at (1,0) {$1$};
\node[below] at (2,0) {$2$};
\node[below] at (3,0) {$3$};
\node[below] at (4,0) {$4$};
\node[below] at (5,0) {$5$};
\node[below] at (6,0) {$6$};
\node[below] at (7,0) {$7$};
\node[below] at (8,0) {$8$};
\node[below] at (9,0) {$9$};
\node[below] at (10,0) {$10$};
\draw[bend left=28] (1,0) to (5.4,1.5);
\draw[bend left=11] (5.4,1.4) to (7,1);
\draw[bend left=11] (7,0.9) to (8,0);
\draw[bend left=30] (2,0) to (3.3,0.5);
\draw[bend left=20] (3.3,0.4) to (4,0);
\draw[bend left=5] (3,0) to (3.3,0.4);
\draw[bend left=17] (3.3,0.5) to (5.4,1.4);
\draw[bend left=28] (5.4,1.5) to (10,0);
\draw[bend left=25] (5,0) to (7,0.9);
\draw[bend left=25] (7,1) to (9,0);
\draw[bend left=50] (6,0) to (7,0);
\end{tikzpicture}

\begin{tikzpicture}[scale=0.6]
\node at (1,0) {$\bullet$};
\node at (2,0) {$\bullet$};
\node at (3,0) {$\bullet$};
\node at (4,0) {$\bullet$};
\node at (5,0) {$\bullet$};
\node at (6,0) {$\bullet$};
\node at (7,0) {$\bullet$};
\node at (8,0) {$\bullet$};
\node at (9,0) {$\bullet$};
\node at (10,0) {$\bullet$};
\node[below] at (1,0) {$1$};
\node[below] at (2,0) {$2$};
\node[below] at (3,0) {$3$};
\node[below] at (4,0) {$4$};
\node[below] at (5,0) {$5$};
\node[below] at (6,0) {$6$};
\node[below] at (7,0) {$7$};
\node[below] at (8,0) {$8$};
\node[below] at (9,0) {$9$};
\node[below] at (10,0) {$10$};
\draw[bend left=50] (1,0) to (4,0);
\draw[bend left=50] (2,0) to (3,0);
\draw[bend left=50] (5,0) to (10,0);
\draw[bend left=50] (6,0) to (7,0);
\draw[bend left=50] (8,0) to (9,0);
\node at (5.5, -2) {$M(\mathbf{0})$};
\end{tikzpicture}
\quad
\begin{tikzpicture}[scale=0.6]
\node at (1,0) {$\bullet$};
\node at (2,0) {$\bullet$};
\node at (3,0) {$\bullet$};
\node at (4,0) {$\bullet$};
\node at (5,0) {$\bullet$};
\node at (6,0) {$\bullet$};
\node at (7,0) {$\bullet$};
\node at (8,0) {$\bullet$};
\node at (9,0) {$\bullet$};
\node at (10,0) {$\bullet$};
\node[below] at (1,0) {$1$};
\node[below] at (2,0) {$2$};
\node[below] at (3,0) {$3$};
\node[below] at (4,0) {$4$};
\node[below] at (5,0) {$5$};
\node[below] at (6,0) {$6$};
\node[below] at (7,0) {$7$};
\node[below] at (8,0) {$8$};
\node[below] at (9,0) {$9$};
\node[below] at (10,0) {$10$};
\draw[bend left=50] (1,0) to (10,0);
\draw[bend left=50] (2,0) to (9,0);
\draw[bend left=50] (3,0) to (4,0);
\draw[bend left=50] (5,0) to (8,0);
\draw[bend left=50] (6,0) to (7,0);
\node at (5.5, -2) {$M(\mathbf{1})$};
\end{tikzpicture}
\caption{Examples for resolution of crossings with $M$ shown in \cref{fig:resolution-crossings}.}
\label{fig:resolution-crossings}
\end{figure}

\begin{theorem}\label{thm:equality-of-bijection}
Let $M \in \TNC_n$.  Then we have $\varphi(M)=(M(\mathbf{0}),M(\mathbf{1}))$, where $M(\mathbf{0})$ and $M(\mathbf{1})$ are identified with Dyck paths via the bijection $\tau^{-1}:\NCM_n \to \mathcal{C}_n$.
\end{theorem}
\begin{proof}
We first recall the bijection between noncrossing matchings $\NCM_n$ and Dyck paths $\mathcal{C}_n$. A noncrossing matching $M\in\NCM_n$ is in bijection with a Dyck path $P\in\mathcal{C}_n$ if for all $j=1,\ldots,2n$, $j$ is a left endpoint in $M$ if and only if the $j$-th step is an upstep in $P$. 

For a $3$-noncrossing matching $M\in\TNC_n$, write $\varphi(M)=(P_0\leq P_1)$ for notation purposes of this proof. Let $T_0,\ldots,T_{2n}$ be the sequence of tableaux corresponding to $M$ as defined above in this section. By construction, $T_j$ has one more entry than $T_{j-1}$ if and only if $j$ is a left endpoint in $M$. At the same time, we see from the construction of $M(\mathbf{1})$ that $M(\mathbf{1})$ and $M$ have the same set of left endpoints and right endpoints (see \cref{fig:resolution-crossings} for a reference). Recall that the height of $P_1$ is given by the sizes of $T_j$'s, and by the bijection between noncrossing matchings and Dyck paths, we directly see that $P_1=M(\mathbf{1})$. 

We use induction on $n$ and on $|\CR(M)|$ to show that $P_0$ and $M(\mathbf{0})$ are in bijection under $\tau$. The base case $n=1$ is trivial. The other base case is when $\CR(M)=\emptyset$, i.e. $M$ is noncrossing. Then the tableau $T_0,\ldots,T_{2n}$ all consist of a single row, so we can see that $P_0=P_1$ is indeed in bijection with $M(\mathbf{0})=M(\mathbf{1})=M$ under $\tau$. 

For the inductive step, as $1$ is a left endpoint in $M$ and $2n$ is a right endpoint in $M$, there must exist a $j\in\{1,\ldots,2n-1\}$ such that $j$ is a left endpoint and $j+1$ is a right endpoint in $M$. There are two possible cases: either $(j, j+1)$ are paired in the matching $M$, or they are not paired in $M$. In the first case where $(j,j+1)$ are paired $M$, let $M'\in \TNC_{n-1}$ be obtained from $M$ by deleting the pair $(j,j+1)$ from $M$ and flattening other entries. By the induction hypothesis, we have $\varphi(M')=(P_0'\leq P_1')$ where $P_0'$ is in bijection with $M'(\mathbf{0})$. We know that $M'(\mathbf{0})$ is obtained from $M(\mathbf{0})$ by removing the pair $(j,j+1)$, so to show that $P_0$ and $M(\mathbf{0})$ are in bijection given that $P_0'$ and $M'(\mathbf{0})$ are in bijection, it suffices to show that $P_0'$ is obtained from $P_0$ by deleting an upstep at $j$ and a downstep at $j+1$. By definition of $T_0,\ldots,T_{2n}$, $T_{j+1}$ is obtained from $T_{j+2}$ by inserting $j$, which is strictly larger than other values in $T_{j+2}$ so $j$ gets inserted as the rightmost entry of the first row in $T_{j+1}$. Then $T_j$ is obtained from $T_{j+1}$ by deleting $j$ from $T_{j+1}$, so $T_j=T_{j+2}$. This means that step $j$ of $P_0$ is an upstep and step $j+1$ is a downstep. Moreover, the sequence of tableaux for $M'$ is $T_0,\ldots,T_{j}=T_{j+2},T_{j+3},\ldots,T_{2n}$ so we conclude that $P_0'$ is obtained from $P_0$ by deleting step $j$ and $j+1$ and we are done.

In the second case where $(j,j+1)$ are not paired in the matching $M$. Suppose that $j$ is paired with $b$ and $j+1$ is paired with $a$ in $M$, where $b>j+1$ and $a<j$. Let $M'$ be obtained from $M$ by swapping the roles of $j$ and $j+1$, i.e. $M'$ is obtained from $M$ by replacing the pair $(a,j+1)$ and $(j,b)$ with $(a,j)$ and $(j+1,b)$. We show that $M'\in \TNC_n$. If there exist three strands $A,B,C$ in $M'$ that intersect pairwise, then at least one of them is either $(a,j)$ or $(j+1,b)$. Say $A=(a,j)$. Then as $(j+1,b)$ does not intersect with $(a,j)$, $(j+1,b)$ does not belong to one of these strands. This means that $(a,j+1)$, $B$ and $C$ also intersect pairwise in $M$, contradicting $M\in \TNC_n$. Let $T_0',\ldots,T_{2n}'$ be the sequence of tableau associated to $M'$ and let $\varphi(M')=(P_0'\leq P_1')$. Let the shape of $T_i$ be $(x_i,y_i)$ and the shape of $T_{i}'$ be $(x_i',y_i')$ for $i=1,\ldots,2n$. The definition of $M(\mathbf{0})$ gives $M(\mathbf{0})=M'(\mathbf{0})$, as $M'$ can be viewed as resolving one crossing from $M$. As $|cr(M')|<|cr(M)|$, by induction hypothesis, $P_0'$ and $M'(\mathbf{0})$ are in bijection. To show that $P_0$ and $M(\mathbf{0})$ are in bijection, it suffices to show that $P_0=P_0'$, i.e. $x_i-y_i=x_i'-y_i'$ for $i=1,\ldots,2n.$

It is clear that $T_i=T_i'$ if $i>b$, and that $T_i$ has the same shape as $T_i'$ if $j+2\leq i\leq b$, while replacing the value $j+1$ in $T_i'$ by $j$. For simplicity, we use $\overline{j}$ to mean $j$ in the tableaux $T_i$, and $j+1$ in $T_i'$, for $i=1,\ldots,2n$. We see that $\overline{j}$ is the largest value in $T_{j+2}$ and $T_{j+2}'$. Now, $T_{j+1}$ is obtained from $T_{j+2}$ by inserting $a$ and $T_{j}$ is obtained from $T_{j+1}$ by removing $\overline{j}$, while $T_{j+1}'$ is obtained from $T_{j+2}'$ by removing $\overline{j}$ and $T_{j}'$ is obtained from $T_{j+1}'$ by inserting $a$. If $\overline{j}$ is in the second row of $T_{j+2}$ and $T_{j+2}'$, then $a$ must be inserted in the first row of $T_{j+2}$, since otherwise, insertion of $a$ will bump some value $c<\overline{j}$ to the second row of $T_{j+2}$, resulting in another bump that creates a third row, contradicting $T_{j+1}$ having at most $2$ rows. In this case, $(x_{j+1},y_{j+1})=(x_{j+2}+1,y_{j+2})$ while $(x_{j+1}',y_{j+1}')=(x_{j+2},y_{j+2}-1)$, satisfying $x_{j+1}-y_{j+1}=x_{j+1}'-y_{j+1}'$. Moreover, removing $\overline{j}$ and inserting $a$ commute on $T_{j+2}$ and $T_{j+2}'$, so $T_i=T_i'$ for $i\leq j$. If $\overline{j}$ is in the first row of $T_{j+2}$ and $T_{j+2}'$, inserting $a$ will increase the length of the second row, so $(x_{j+1},y_{j+1})=(x_{j+2},y_{j+2}+1)$ while $(x_{j+1}',y_{j+1}')=(x_{j+2}-1,y_{j+2})$, satisfying $x_{j+1}-y_{j+1}=x_{j+1}'-y_{j+1}'$. Moreover, whether $a$ bumps $\overline{j}$ to the second row or $a$ bumps some other value to the second row, we direcly see that removing $\overline{j}$ and inserting $a$ commute. As a result, $x_i-y_i=x_i'-y_i'$ in all cases, concluding our proof.
\end{proof}

\begin{prop}\label{prop:resolution-1-maximum}
For $M\in\TNC_n$ and $v\in\{0,1\}^{\CR(M)}$, $M(v)\leq M(\mathbf{1})$, with equality if and only if $v=\mathbf{1}$.
\end{prop}
\begin{proof}
For a matching $M\in\TNC_n$ and $k=1,\ldots,2n$, let $l_k(M)$ denote the number of left endpoints among $1,2,\ldots,k$ in $M$, and let $r_k(M)$ denote the number of right endpoints among $1,2,\ldots,k$. Recall that the corresponding Dyck path for $M$ is the path that passes through $(k,l_k(M)-r_k(M))$ for $k=0,\ldots,2n$. Thus, to show $M'\leq M$, it suffices to show that $l_k(M')\leq l_k(M)$ for all $k$. 

Consider a planar drawing of $M$ inside the rectangle $[1,2n]\times [0,1]$, such that for a pair $(a,b)$ in $M$ where $a<b$, we connect the coordinate $(a,0)$ to $(b,0)$ via a curve $\gamma_{a,b}:[0,1]\rightarrow[a,b]\times [0,1]$ such that $\gamma_{a,b}(0)=(a,0)$, $\gamma_{a,b}(1)=(b,0)$ and that the projection of $\gamma_{a,b}$ onto the first coordinate is bijective, i.e. we can view this curve as going from left to right. The matching shown in \cref{fig:bijection-chen-3noncrossing-pair} is such an example. We view $M(v)$ as modifying and regrouping the curves in a small neighborhood around each crossing in $M$ (see the top of \cref{fig:resolution-crossings}). Note that the crossings of $M$ decompose the curves into segments. For each crossing $z$, there are four segments of curves touching $z$ in the planar drawing of $M$, from the direction of NW, SW, NE and SE.

For $k=1,\ldots,2n$, the vertical line $x=k+0.5$ intersects the strands in the drawing of $M$ with a left endpoint in $1,\ldots,k$ and a right endpoint in $k+1,\ldots,n$. There are $l_k(M)-r_k(M)$ number of such strands. Let $\xi_1,\ldots,\xi_h$ be the segments that intersect $x=k+0.5$, where $h=l_k(M)-r_k(M)$. Now, in any resolution of crossings $M(v)$, there are $l_k(M(v))-r_k(M(v))$ strands with a left endpoint in $1,\ldots,k$ and a right endpoint in $k+1,\ldots,n$. Each such strand needs to contain at least one segment of $\xi_1,\ldots,\xi_h$, meaning that $l_k(M(v))-r_k(M(v))\leq h$ so $l_k(M(v))\leq l_k(M)=l_k(M(\mathbf{1}))$. This means $M(v)\leq M(\mathbf{1})$ as desired.
\end{proof}

\subsection{$3$-noncrossing matchings and cactus networks}
\begin{prop}\label{prop:3-noncrossing}
Let $\tau\in\MM_n$ be a matching. The followings are equivalent:
\begin{enumerate}
\item $\tau\in\TNC_n$ is a $3$-noncrossing matching;
\item all interior regions of (any) strand diagram of $\tau$ have at least $4$ sides;
\item $\tau$ has a unique strand diagram;
\item there is a reduced cactus network $\Gamma$ whose medial pairing is $\tau$ such that $\Gamma$ does not have any triangular faces and all interior vertices of $\Gamma$ have degree at least $4$.
\end{enumerate}
\end{prop}
\begin{proof}
We prove (1)$\Rightarrow$(4)$\Rightarrow$(3)$\Rightarrow$(2)$\Rightarrow$(1).

(1)$\Rightarrow$(4). Assume $\tau$ is a $3$-noncrossing matching. Take any strand diagram / medial graph of $\tau$ and recover a reduced cactus network $\Gamma$ as described in \cref{sec:background}. As $\Gamma$ is reduced, no interior vertex has degree $2$. If $\Gamma$ has a triangular face with edges $e_1,e_2,e_3$, then this medial graph $G(\tau)$ has a triangular face given by $t_{e_1},t_{e_2},t_{e_3}$; if $\Gamma$ has an interior vertex of degree $3$, then it comes from a triangular face in the strand diagram. Both cases imply that there exists a triangular face in the strand diagram of $\tau$, which corresponds to three strands that cross pairwise, contradicting $\tau$ being $3$-noncrossing.

(4)$\Rightarrow$(3). If $\tau$ has two (or more) strand diagrams, then any corresponding reduced cactus network admits a $Y-\Delta$ move, by \cref{prop:Y-delta}. But $\Gamma$ does not have any interior vertex of degree $3$ or a triangular face. This is a contradiction.

(3)$\Rightarrow$(2). If any strand diagram of $\tau$ has a region with $3$ sides, i.e. a triangular face, then we can apply a local move of the form: \begin{tikzpicture}[scale=0.1]\draw(0,0)--(4,4);\draw(0,1)--(4,1);\draw(0,4)--(4,0);\end{tikzpicture} $\leftrightarrow$ \begin{tikzpicture}[scale=0.1]\draw(0,0)--(4,4);\draw(0,3)--(4,3);\draw(0,4)--(4,0);\end{tikzpicture}. This implies that $\tau$ has more than one strand diagram.

(2)$\Rightarrow$(1). If $\tau$ is not $3$-noncrossing, then in a strand diagram of $\tau$, there are three strands that cross pairwise, forming an interior triangle. This triangle may not be a face, but any other strand that passes through this triangle creates a smaller triangle.  Thus any strand diagram of $\tau$ contains an interior triangular face. 
\end{proof}
For a $3$-noncrossing matching $\xi\in\TNC$, let $\Gamma(\xi)$ be the corresponding reduced cactus network, which is unique by \cref{prop:3-noncrossing}.

\section{The Grove algebra $G_n$ and the Bush basis $\{B_{\xi}\}$ for $G_{n,2}$}\label{sec:circular-basis}
Let $\II_n\subset\Z[L_{\sigma}\:|\:\sigma\in\NP_n]$ be the ideal of polynomials in the variables $\{L_{\sigma}\}_{\sigma\in\NP_n}$ that vanish on all electrical networks $\Gamma$, or equivalently, on all cactus networks $\Gamma$. Define the \textit{grove algebra} (over $\Z$) to be \[G_n:=\Z[L_{\sigma}\:|\: \sigma\in\NP_n]/\II_n\] which is a graded algebra by the degree of the polynomials, where we assign each variable $L_{\sigma}$ to have degree $1$. Let the $d$-th graded piece of $G_n$ be $G_{n,d}$.  We will describe $G_n \otimes \C$ in more algebro-geometric terms in \cref{sec:geometry}.

In this section, we introduce a basis $\{B_\xi\:|\: \xi\in\TNC_n\}$ of $G_{n,2}$ indexed by 3-noncrossing matchings that exhibits cyclic symmetry. We first define $B_\xi(\Gamma)$ for any cactus network $\Gamma$ and then lift them to $G_{n,2}$.

\subsection{Definition of the Bush basis $\{B_{\xi}\}$}
\begin{defin}
A \textit{double grove} $H$ on boundary vertices $\bar1,\bar2,\ldots,\bar n$ is a cactus graph where each edge has multiplicity $1$ or $2$ (which we call a \textit{double edge}). 
\end{defin}

Let $\Z\TNC_n$ denote the free module with basis elements $\TNC_n$.  For a double grove $H$, we define an invariant $\alpha(H)\in\Z\TNC_n$ by the following procedure. 
\begin{defin}\label{def:alpha-H}
For a cactus graph $H$ with multiplicity (i.e. edges of $H$ are allowed to have multiplicities being positive integers), $\alpha(H)$ is defined recursively as follows:
\begin{enumerate}
\item[(0)] If some edge of $H$ has multiplicity $\geq3$, then $\alpha(H)=0$.
\item If all edges of $H$ have multiplicity $1$ and the medial pairing $\tau(H)$ is a $3$-noncrossing matching $\xi$, i.e. $H$ is reduced and does not have any interior vertex of degree equal to $3$ or any triangular faces (see \cref{prop:3-noncrossing}), then $\alpha(H)=\xi$.
\item If $H$ has loops, then $\alpha(H)=0$.
\item If $H$ has a double edge $e$, then $\alpha(H)=\alpha(H')$, where $H'$ is obtained from $H$ by contracting $e$.
\item If $H$ has edges $e_1,\ldots,e_r$ of multiplicity $1$ between the same pair of vertices, then $\alpha(H)=0$ if $r\geq3$, and $\alpha(H)=2\alpha(H')$ where $H'$ is obtained from $H$ by contracting $e_1$ (and removing $e_2$) if $r=2$. 
\item If $H$ has an interior leaf or an interior vertex with degree $0$, then $\alpha(H)=0$.
\item If $H$ has an interior vertex $v$ of degree $2$ with distinct neighbors, then $\alpha(H)=2\alpha(H')$ where $H'$ is obtained from $H$ by removing $v$.
\item If $H$ has an interior vertex $v$ of degree $3$ with distinct neighbors $v_1,v_2,v_3$, then $\alpha(H)=\alpha(H_1)+\alpha(H_2)+\alpha(H_3)$ where $H_i$ is obtained from $H$ by removing $v$ together with the edges incident to it and adding the edge $(v_{i+1},v_{i+2})$ where the indices are taken modulo $3$.
\item If $H$ has a triangular face (not necessarily in the interior) with vertices $v_1,v_2,v_3$, then $\alpha(H)=\alpha(H_1)+\alpha(H_2)+\alpha(H_3)$ where $H_i$ is obtained from $H$ by identifying $v_{i}$ with $v_{i+1}$ and replacing the three edges by a single edge between $v_i=v_{i+1}$ and $v_{i+2}$.
\end{enumerate}
For $\xi\in\TNC_n$, let $\alpha(H)_{\xi}\in\mathbb{Z}$ denote the coefficient of $\xi$ in $\alpha(H)$.
\end{defin}
A visualization of some moves in \cref{def:alpha-H} is shown in \cref{fig:alpha-H-steps}.
\begin{figure}[h!]
\centering
\begin{tikzpicture}[scale=0.4]
\def\x{7}
\draw (0,0) circle (2);
\node at (-1,0) {$\bullet$};
\node at (1,0) {$\bullet$};
\draw(-1,0)--(-1.6,0);
\draw(-1,0)--(-1.5196,0.3);
\draw(-1,0)--(-1.5196,-0.3);
\draw(1,0)--(1.6,0);
\draw(1,0)--(1.5196,0.3);
\draw(1,0)--(1.5196,-0.3);
\draw[bend left=30] (-1,0) to (1,0);
\draw[bend right=30] (-1,0) to (1,0);

\draw (\x,0) circle (2);
\node at (\x,0) {$\bullet$};
\draw(\x,0)--(\x-0.6,0);
\draw(\x,0)--(\x-0.5196,0.3);
\draw(\x,0)--(\x-0.5196,-0.3);
\draw(\x,0)--(\x+0.6,0);
\draw(\x,0)--(\x+0.5196,0.3);
\draw(\x,0)--(\x+0.5196,-0.3);

\node at (0.5*\x,0) {$\rightarrow\ 2\cdot$};
\node at (0.5*\x,-3) {Move (4)};
\end{tikzpicture}
\qquad
\begin{tikzpicture}[scale=0.4]
\def\x{7}
\draw (0,0) circle (2);
\node at (-1,0) {$\bullet$};
\node at (1,0) {$\bullet$};
\draw(-1,0)--(-1.6,0);
\draw(-1,0)--(-1.5196,0.3);
\draw(-1,0)--(-1.5196,-0.3);
\draw(1,0)--(1.6,0);
\draw(1,0)--(1.5196,0.3);
\draw(1,0)--(1.5196,-0.3);
\node at (0,0) {$\bullet$};
\draw(-1,0)--(1,0);

\draw (\x,0) circle (2);
\node at (\x-1,0) {$\bullet$};
\node at (\x+1,0) {$\bullet$};
\draw(\x-1,0)--(\x-1.6,0);
\draw(\x-1,0)--(\x-1.5196,0.3);
\draw(\x-1,0)--(\x-1.5196,-0.3);
\draw(\x+1,0)--(\x+1.6,0);
\draw(\x+1,0)--(\x+1.5196,0.3);
\draw(\x+1,0)--(\x+1.5196,-0.3);

\node at (0.5*\x,0) {$\rightarrow\ 2\cdot$};
\node at (0.5*\x,-3) {Move (6)};
\end{tikzpicture}

\

\begin{tikzpicture}[scale=0.4]
\def\x{6}
\node at (0,0) {$\bullet$};
\node[right] at (0,1) {$v_1$};
\draw (0,0) circle (2);
\node at (0,1) {$\bullet$};
\node at (0.866,-0.5) {$\bullet$};
\node[below] at (0.866,-0.5) {$v_3$};
\node at (-0.866,-0.5) {$\bullet$};
\node[below] at (-0.866,-0.5) {$v_2$};
\draw(0,0)--(0,1);
\draw(0.866,-0.5)--(0,0)--(-0.866,-0.5);

\node[right] at (\x,1) {$v_1$};
\draw (\x,0) circle (2);
\node at (\x,1) {$\bullet$};
\node at (\x+0.866,-0.5) {$\bullet$};
\node[below] at (\x+0.866,-0.5) {$v_3$};
\node at (\x-0.866,-0.5) {$\bullet$};
\node[below] at (\x-0.866,-0.5) {$v_2$};
\draw(\x-0.866,-0.5)--(\x,1);

\node[right] at (2*\x,1) {$v_1$};
\draw (2*\x,0) circle (2);
\node at (2*\x,1) {$\bullet$};
\node at (2*\x+0.866,-0.5) {$\bullet$};
\node[below] at (2*\x+0.866,-0.5) {$v_3$};
\node at (2*\x-0.866,-0.5) {$\bullet$};
\node[below] at (2*\x-0.866,-0.5) {$v_2$};
\draw(2*\x+0.866,-0.5)--(2*\x,1);

\node[right] at (3*\x,1) {$v_1$};
\draw (3*\x,0) circle (2);
\node at (3*\x,1) {$\bullet$};
\node at (3*\x+0.866,-0.5) {$\bullet$};
\node[below] at (3*\x+0.866,-0.5) {$v_3$};
\node at (3*\x-0.866,-0.5) {$\bullet$};
\node[below] at (3*\x-0.866,-0.5) {$v_2$};
\draw(3*\x+0.866,-0.5)--(3*\x-0.866,-0.5);

\node at (0.5*\x,0) {$\rightarrow$};
\node at (1.5*\x,0) {$+$};
\node at (2.5*\x,0) {$+$};
\node at (1.5*\x,-3) {Move (7)};
\end{tikzpicture}

\

\begin{tikzpicture}[scale=0.4]
\def\x{6}
\node[right] at (0,1) {$v_1$};
\draw (0,0) circle (2);
\node at (0,1) {$\bullet$};
\node at (0.866,-0.5) {$\bullet$};
\node[below] at (0.866,-0.5) {$v_3$};
\node at (-0.866,-0.5) {$\bullet$};
\node[below] at (-0.866,-0.5) {$v_2$};
\draw(0,1)--(0.866,-0.5)--(-0.866,-0.5)--(0,1);

\node at (\x-0.433,0.25) {$\bullet$};
\node[above] at (\x-0.433,0.25) {$v_1{=}v_2$};
\draw (\x,0) circle (2);
\node at (\x+0.866,-0.5) {$\bullet$};
\node[below] at (\x+0.866,-0.5) {$v_3$};
\draw(\x-0.433,0.25)--(\x+0.866,-0.5);

\node at (2*\x+0.433,0.25) {$\bullet$};
\node[above] at (2*\x+0.433,0.25) {$v_1{=}v_3$};
\draw (2*\x,0) circle (2);
\node at (2*\x-0.866,-0.5) {$\bullet$};
\node[below] at (2*\x-0.866,-0.5) {$v_2$};
\draw(2*\x+0.433,0.25)--(2*\x-0.866,-0.5);

\draw(3*\x,0) circle (2);
\node at (3*\x,1) {$\bullet$};
\node[right] at (3*\x,1) {$v_1$};
\node at (3*\x,-0.5) {$\bullet$};
\node[below] at (3*\x,-0.5) {$v_2{=}v_3$};
\draw(3*\x,1)--(3*\x,-0.5);

\node at (0.5*\x,0) {$\rightarrow$};
\node at (1.5*\x,0) {$+$};
\node at (2.5*\x,0) {$+$};
\node at (1.5*\x,-3) {Move (8)};
\end{tikzpicture}
\caption{Moves (4), (6), (7) and (8) of \cref{def:alpha-H}}
\label{fig:alpha-H-steps}
\end{figure}
It is clear that the procedure in \cref{def:alpha-H} terminates, as either the number of vertices or the number of edges strictly decreases after step. However, it is not immediately clear that $\alpha(H)$ is well-defined, i.e. the result of the recursion is independent of choices of moves. we show in \cref{lem:alpha-H-well-defined} that indeed $\alpha(H)$ is well-defined.  

\begin{lemma}\label{lem:alpha-H-well-defined}
For a double grove $H$, $\alpha(H)$ is well-defined. In other words, $\alpha(H)$ does not depend on the order that we apply the moves in \cref{def:alpha-H}.
\end{lemma}
\begin{proof}
By Newman's lemma (also commonly known as the diamond lemma) \cite{newman1942theories}, it suffices to show that for a cactus graph $H$ with multiplicities, if there are two reduction moves $m$ and $m'$ that can be applied to $H$ resulting in $\alpha(H)$ being assigned to $A=\sum c_i\alpha(H_i)$ and $A'=\sum c_i'\alpha(H_i')$, then there are sequences of reduction moves that can be applied to $A$ and $A'$ respectively, such that the end results are the same. Depending on which types of moves $m$ and $m'$ are and which edges are involved in these moves, the number of cases to check is humongous, but most of them are very straightforward, including the situations where $m$ and $m'$ involve disjoint sets of edges and any combinations of moves (0) - (6).  We check a few critical cases in this proof.

Case 1: $m$ and $m'$ are moves (7) applied to vertices $v$ and $v'$ respectively. Notice that if $v$ and $v'$ are not adjacent, then moves (7) at $v$ and $v'$ commute. So it suffices to look at the following local configuration in $H$, and apply the calculation in \cref{fig:case1-of-alpha-H-well-defined}, where we apply moves (7), (7), (6) in \cref{def:alpha-H} after move $m$. We see that the same result can be achieved if we apply move $m'$ first, as the end result in \cref{fig:case1-of-alpha-H-well-defined} is symmetric horizontally. 
\begin{figure}[h!]
\centering
\begin{tikzpicture}[scale=0.3]
\def\x{7}
\def\y{4.5}
\def\h{5}
\def\m{8}
\def\eps{1}
\draw (0,0) circle (2);
\node at (-1,-1) {$\bullet$};
\node at (1,-1) {$\bullet$};
\node at (-1,1) {$\bullet$};
\node at (1,1) {$\bullet$};
\node at (-0.5,0) {$\bullet$};
\node at (0.5,0) {$\bullet$};
\draw(-1,1)--(-0.5,0)--(-1,-1);
\draw(1,1)--(0.5,0)--(1,-1);
\draw(-0.5,0)--(0.5,0);
\node[left] at (-0.5,0) {$v$};
\node[right] at (0.5,0) {$v'$};

\node at (\m-1,-1) {$\bullet$};
\node at (\m-1,1) {$\bullet$};
\node at (\m+1,-1) {$\bullet$};
\node at (\m+1,1) {$\bullet$};
\node at (\m+0.5,0) {$\bullet$};
\draw(\m-1,1)--(\m+0.5,0)--(\m+1,1);
\draw(\m+0.5,0)--(\m+1,-1);
\node[above] at (0.5*\m,0) {$m$};
\draw[->](3,0)--(\m-3,0);
\draw (\m,0) circle (2);

\node at (\x+\m-1,-1) {$\bullet$};
\node at (\x+\m-1,1) {$\bullet$};
\node at (\x+\m+1,-1) {$\bullet$};
\node at (\x+\m+1,1) {$\bullet$};
\node at (\x+\m+0.5,0) {$\bullet$};
\draw(\x+\m-1,-1)--(\x+\m+0.5,0)--(\x+\m+1,1);
\draw(\x+\m+0.5,0)--(\x+\m+1,-1);
\node at (\m+0.5*\x,0) {$+$};
\draw (\m+\x,0) circle (2);

\node at (2*\x+\m-1,-1) {$\bullet$};
\node at (2*\x+\m-1,1) {$\bullet$};
\node at (2*\x+\m+1,-1) {$\bullet$};
\node at (2*\x+\m+1,1) {$\bullet$};
\node at (2*\x+\m+0.5,0) {$\bullet$};
\draw(2*\x+\m-1,-1)--(2*\x+\m-1,1);
\draw(2*\x+\m+1,1)--(2*\x+\m+0.5,0)--(2*\x+\m+1,-1);
\node at (\m+1.5*\x,0) {$+$};
\draw (\m+2*\x,0) circle (2);

\node at (\m-1,-1-\h) {$\bullet$};
\node at (\m-1,1-\h) {$\bullet$};
\node at (\m+1,-1-\h) {$\bullet$};
\node at (\m+1,1-\h) {$\bullet$};
\draw(\m-1,1-\h)--(\m+1,1-\h);
\node[above] at (0.5*\m-1.5,-\h) {(7),(7),(6)};
\draw[->](0,-\h)--(\m-3,-\h);

\node at (\m-1+\y,-1-\h) {$\bullet$};
\node at (\m-1+\y,1-\h) {$\bullet$};
\node at (\m+1+\y,-1-\h) {$\bullet$};
\node at (\m+1+\y,1-\h) {$\bullet$};
\draw(\m-1+\y,1-\h)--(\m+1+\y,-1-\h);
\node at (\m+0.5*\y,-\h) {$+$};

\node at (\m-1+2*\y,-1-\h) {$\bullet$};
\node at (\m-1+2*\y,1-\h) {$\bullet$};
\node at (\m+1+2*\y,-1-\h) {$\bullet$};
\node at (\m+1+2*\y,1-\h) {$\bullet$};
\draw(\m+1+2*\y,1-\h)--(\m+1+2*\y,-1-\h);
\node at (\m+1.5*\y,-\h) {$+$};

\node at (\m-1+3*\y,-1-\h) {$\bullet$};
\node at (\m-1+3*\y,1-\h) {$\bullet$};
\node at (\m+1+3*\y,-1-\h) {$\bullet$};
\node at (\m+1+3*\y,1-\h) {$\bullet$};
\draw(\m+1+3*\y,1-\h)--(\m-1+3*\y,-1-\h);
\node at (\m+2.5*\y,-\h) {$+$};

\node at (\m-1+4*\y,-1-\h) {$\bullet$};
\node at (\m-1+4*\y,1-\h) {$\bullet$};
\node at (\m+1+4*\y,-1-\h) {$\bullet$};
\node at (\m+1+4*\y,1-\h) {$\bullet$};
\draw(\m-1+4*\y,-1-\h)--(\m+1+4*\y,-1-\h);
\node at (\m+3.5*\y,-\h) {$+$};

\node at (\m-1+5*\y,-1-\h) {$\bullet$};
\node at (\m-1+5*\y,1-\h) {$\bullet$};
\node at (\m+1+5*\y,-1-\h) {$\bullet$};
\node at (\m+1+5*\y,1-\h) {$\bullet$};
\draw(\m+1+5*\y,-1-\h)--(\m+1+5*\y,1-\h);
\node at (\m+4.5*\y,-\h) {$+$};

\node at (\m-1+6*\y+\eps,-1-\h) {$\bullet$};
\node at (\m-1+6*\y+\eps,1-\h) {$\bullet$};
\node at (\m+1+6*\y+\eps,-1-\h) {$\bullet$};
\node at (\m+1+6*\y+\eps,1-\h) {$\bullet$};
\draw(\m-1+6*\y+\eps,-1-\h)--(\m-1+6*\y+\eps,1-\h);
\node at (\m+5.5*\y+0.5*\eps,-\h) {$+\ 2\cdot$};

\node at (\m-1,-1-2*\h) {$\bullet$};
\node at (\m-1,1-2*\h) {$\bullet$};
\node at (\m+1,-1-2*\h) {$\bullet$};
\node at (\m+1,1-2*\h) {$\bullet$};
\draw(\m-1,1-2*\h)--(\m+1,1-2*\h);
\node[above] at (0.5*\m,-2*\h) {$=$};

\node at (\m-1+\y,-1-2*\h) {$\bullet$};
\node at (\m-1+\y,1-2*\h) {$\bullet$};
\node at (\m+1+\y,-1-2*\h) {$\bullet$};
\node at (\m+1+\y,1-2*\h) {$\bullet$};
\draw(\m-1+\y,-1-2*\h)--(\m+1+\y,-1-2*\h);
\node at (\m+0.5*\y,-2*\h) {$+$};

\node at (\m-1+2*\y,-1-2*\h) {$\bullet$};
\node at (\m-1+2*\y,1-2*\h) {$\bullet$};
\node at (\m+1+2*\y,-1-2*\h) {$\bullet$};
\node at (\m+1+2*\y,1-2*\h) {$\bullet$};
\draw(\m-1+2*\y,1-2*\h)--(\m+1+2*\y,-1-2*\h);
\node at (\m+1.5*\y,-2*\h) {$+$};

\node at (\m-1+3*\y,-1-2*\h) {$\bullet$};
\node at (\m-1+3*\y,1-2*\h) {$\bullet$};
\node at (\m+1+3*\y,-1-2*\h) {$\bullet$};
\node at (\m+1+3*\y,1-2*\h) {$\bullet$};
\draw(\m+1+3*\y,1-2*\h)--(\m-1+3*\y,-1-2*\h);
\node at (\m+2.5*\y,-2*\h) {$+$};

\node at (\m-1+4*\y+\eps,-1-2*\h) {$\bullet$};
\node at (\m-1+4*\y+\eps,1-2*\h) {$\bullet$};
\node at (\m+1+4*\y+\eps,-1-2*\h) {$\bullet$};
\node at (\m+1+4*\y+\eps,1-2*\h) {$\bullet$};
\draw(\m-1+4*\y+\eps,1-2*\h)--(\m-1+4*\y+\eps,-1-2*\h);
\node at (\m+3.5*\y+0.5*\eps,-2*\h) {$+\ 2\cdot$};

\node at (\m-1+5*\y+2*\eps,-1-2*\h) {$\bullet$};
\node at (\m-1+5*\y+2*\eps,1-2*\h) {$\bullet$};
\node at (\m+1+5*\y+2*\eps,-1-2*\h) {$\bullet$};
\node at (\m+1+5*\y+2*\eps,1-2*\h) {$\bullet$};
\draw(\m+1+5*\y+2*\eps,1-2*\h)--(\m+1+5*\y+2*\eps,-1-2*\h);
\node at (\m+4.5*\y+1.5*\eps,-2*\h) {$+\ 2\cdot$};
\end{tikzpicture}
\caption{Case 1 of \cref{lem:alpha-H-well-defined}}
\label{fig:case1-of-alpha-H-well-defined}
\end{figure}

Case 2: $m$ and $m'$ are moves (8) applied to the faces $v,x,y$ and faces $v',x,y$ respectively. We have a similar calculation shown in \cref{fig:case2-of-alpha-H-well-defined}. Note that $v,x,y,v'$ may be incident to other edges and we will not draw these edges for simplicity of the visualization. We conclude that applying $m$ first or $m'$ first can result in the same configuration via symmetry.
\begin{figure}[h!]
\centering
\begin{tikzpicture}[scale=0.3]
\def\x{8}
\def\y{4.5}
\def\h{5}
\def\m{9}
\def\eps{2}
\draw (0,0) circle (2.4);
\node at (-1,0) {$\bullet$};
\node at (1,0) {$\bullet$};
\node at (0,1) {$\bullet$};
\node at (0,-1) {$\bullet$};
\draw(-1,0)--(0,1)--(0,-1)--(-1,0);
\draw(0,1)--(1,0)--(0,-1);
\node[above] at (0,1) {$x$};
\node[right] at (1,0) {$v'$};
\node[below] at (0,-1) {$y$};
\node[left] at (-1,0) {$v$};
\draw[->](3,0)--(\m-3,0);
\node[above] at (0.5*\m,0) {$m$};

\node at (\m-1,1) {$\bullet$};
\node at (\m+1,0) {$\bullet$};
\node at (\m+0,-1) {$\bullet$};
\draw(\m-1,1)--(\m+1,0)--(\m+0,-1)--(\m-1,1);
\node[above] at (\m-1,1) {$v{=}x$};
\node[right] at (\m+1,0) {$v'$};
\node[below] at (\m,-1) {$y$};

\node at (\m-1+\x,-1) {$\bullet$};
\node at (\m+\x,1) {$\bullet$};
\node at (\m+1+\x,0) {$\bullet$};
\node[below] at (\m-1+\x,-1) {$v{=}y$};
\draw(\m-1+\x,-1)--(\m+\x,1)--(\m+1+\x,0)--(\m-1+\x,-1);
\node[above] at (\m+\x,1) {$v'$};
\node[right] at (\m+1+\x,0) {$x$};
\node at (\m+0.5*\x,0) {$+$};

\node at (\m-1.8+2*\x,0) {$\bullet$};
\node at (\m+2*\x,0) {$\bullet$};
\node at (\m+1.8+2*\x,0) {$\bullet$};
\node[above] at (\m-1.8+2*\x,0) {$v$};
\node[below] at (\m+2*\x,0) {$x{=}y$};
\node[above] at (\m+1.8+2*\x,0) {$v'$};
\draw(\m-1.8+2*\x,0)--(\m+2*\x,0);
\draw[bend left=30](\m+2*\x,0) to (\m+1.8+2*\x,0);
\draw[bend right=30](\m+2*\x,0) to (\m+1.8+2*\x,0);
\node at (\m+1.5*\x,0) {$+$};

\node[above] at (0.5*\m-1.5,-\h) {(7),(7),(4)};
\draw[->](0,-\h)--(\m-3,-\h);

\node at (\m-1,-\h) {$\bullet$};
\node at (\m+1,-\h) {$\bullet$};
\node[above] at (\m-1,-\h) {$v,x,y$};
\node[below] at (\m+1,-\h) {$v'$};
\draw(\m-1,-\h)--(\m+1,-\h);

\node at (\m-1+\y,-\h) {$\bullet$};
\node at (\m+1+\y,-\h) {$\bullet$};
\node[above] at (\m-1+\y,-\h) {$v,x$};
\node[below] at (\m+1+\y,-\h) {$v',y$};
\draw(\m-1+\y,-\h)--(\m+1+\y,-\h);
\node at (\m+0.5*\y,-\h) {$+$};

\node at (\m+2*\y,1-\h) {$\bullet$};
\node at (\m+2*\y,-1-\h) {$\bullet$};
\node[above] at (\m+2*\y,1-\h)  {$v,x,v'$};
\node[below] at (\m+2*\y,-1-\h) {$y$};
\draw(\m+2*\y,1-\h) --(\m+2*\y,-1-\h);
\node at (\m+1.5*\y,-\h) {$+$};

\node at (\m-1+3*\y,-\h) {$\bullet$};
\node at (\m+1+3*\y,-\h) {$\bullet$};
\node[above] at (\m-1+3*\y,-\h) {$v,x,y$};
\node[below] at (\m+1+3*\y,-\h) {$v'$};
\draw(\m-1+3*\y,-\h)--(\m+1+3*\y,-\h);
\node at (\m+2.5*\y,-\h) {$+$};

\node at (\m-1+4*\y,-\h) {$\bullet$};
\node at (\m+1+4*\y,-\h) {$\bullet$};
\node[above] at (\m-1+4*\y,-\h) {$v,y$};
\node[below] at (\m+1+4*\y,-\h) {$v',x$};
\draw(\m-1+4*\y,-\h)--(\m+1+4*\y,-\h);
\node at (\m+3.5*\y,-\h) {$+$};

\node at (\m+5*\y,1-\h) {$\bullet$};
\node at (\m+5*\y,-1-\h) {$\bullet$};
\node[above] at (\m+5*\y,1-\h)  {$x$};
\node[below] at (\m+5*\y,-1-\h) {$v,y,v'$};
\draw(\m+5*\y,1-\h) --(\m+5*\y,-1-\h);
\node at (\m+4.5*\y,-\h) {$+$};

\node at (\m-1+6*\y+\eps,-\h) {$\bullet$};
\node at (\m+1+6*\y+\eps,-\h) {$\bullet$};
\node[above] at (\m-1+6*\y+\eps,-\h) {$v$};
\node[below] at (\m+1+6*\y+\eps,-\h) {$v',x,y$};
\draw(\m-1+6*\y+\eps,-\h)--(\m+1+6*\y+\eps,-\h);
\node at (\m+5.5*\y+0.5*\eps,-\h) {$+\ 2\cdot$};

\node at (\m-1,-2*\h) {$\bullet$};
\node at (\m+1,-2*\h) {$\bullet$};
\node[above] at (\m-1,-2*\h) {$v,x$};
\node[below] at (\m+1,-2*\h) {$v',y$};
\draw(\m-1,-2*\h)--(\m+1,-2*\h);
\node[above] at (0.5*\m,-2*\h) {$=$};

\node at (\m-1+\y,-2*\h) {$\bullet$};
\node at (\m+1+\y,-2*\h) {$\bullet$};
\node[above] at (\m-1+\y,-2*\h) {$v,y$};
\node[below] at (\m+1+\y,-2*\h) {$v',x$};
\draw(\m-1+\y,-2*\h)--(\m+1+\y,-2*\h);
\node at (\m+0.5*\y,-2*\h) {$+$};

\node at (\m+2*\y,1-2*\h) {$\bullet$};
\node at (\m+2*\y,-1-2*\h) {$\bullet$};
\node[above] at (\m+2*\y,1-2*\h)  {$v,x,v'$};
\node[below] at (\m+2*\y,-1-2*\h) {$y$};
\draw(\m+2*\y,1-2*\h) --(\m+2*\y,-1-2*\h);
\node at (\m+1.5*\y,-2*\h) {$+$};

\node at (\m+3*\y,1-2*\h) {$\bullet$};
\node at (\m+3*\y,-1-2*\h) {$\bullet$};
\node[above] at (\m+3*\y,1-2*\h)  {$x$};
\node[below] at (\m+3*\y,-1-2*\h) {$v,y,v'$};
\draw(\m+3*\y,1-2*\h) --(\m+3*\y,-1-2*\h);
\node at (\m+2.5*\y,-2*\h) {$+$};

\node at (\m+4*\y+\eps+1,-2*\h) {$\bullet$};
\node at (\m+4*\y+\eps-1,-2*\h) {$\bullet$};
\node[below] at (\m+4*\y+\eps+1,-2*\h)  {$v'$};
\node[above] at (\m+4*\y+\eps-1,-2*\h) {$v,x,y$};
\draw(\m+4*\y+\eps+1,-2*\h) --(\m+4*\y+\eps-1,-2*\h);
\node at (\m+3.5*\y+0.5*\eps,-2*\h) {$+\ 2\cdot$};

\node at (\m+5*\y+2*\eps+1,-2*\h) {$\bullet$};
\node at (\m+5*\y+2*\eps-1,-2*\h) {$\bullet$};
\node[below] at (\m+5*\y+2*\eps+1,-2*\h)  {$v',x,y$};
\node[above] at (\m+5*\y+2*\eps-1,-2*\h) {$v$};
\draw(\m+5*\y+2*\eps+1,-2*\h) --(\m+5*\y+2*\eps-1,-2*\h);
\node at (\m+4.5*\y+1.5*\eps,-2*\h) {$+\ 2\cdot$};
\end{tikzpicture}
\caption{Case 2 of \cref{lem:alpha-H-well-defined}}
\label{fig:case2-of-alpha-H-well-defined}
\end{figure}

Case 3: $m$ is a move of type (7) applied to a vertex $v$ and $m'$ is a move of type (8) applied to a triangular face $v,x,y$. We consider the reductions separately for $m$ and $m'$ in \cref{fig:case3-of-alpha-H-well-defined} and observe the local confluence as desired. Notice that $v$ has degree $3$, which is adjacent to the vertices $x,y,z$, while the vertices $x,y,z$ can have more edges incident to them. 
\begin{figure}[h!]
\centering
\begin{tikzpicture}[scale=0.3]
\def\x{5}
\def\m{8}
\def\eps{2}
\draw (0,0) circle (2.5);
\node at (0.2,0) {$\bullet$};
\node at (-1,1) {$\bullet$};
\node at (-1,-1) {$\bullet$};
\node at (1.7,0) {$\bullet$};
\node[above] at (-1,1) {$x$};
\node[below] at (-1,-1) {$y$};
\node[above] at (0.2,0) {$v$};
\node[above] at (1.7,0) {$z$};
\draw(-1,1)--(-1,-1)--(0.2,0)--(-1,1);
\draw(0.2,0)--(1.7,0);
\draw[->](3,0)--(\m-3,0);
\node[above] at (0.5*\m,0) {$m$};

\node at (-1+\m,1) {$\bullet$};
\node at (-1+\m,-1) {$\bullet$};
\node at (1+\m,0) {$\bullet$};
\node[left] at (-1+\m,1) {$x$};
\node[left] at (-1+\m,-1) {$y$};
\node[above] at (1+\m,0) {$z$};
\draw(-1+\m,-1)--(-1+\m,1)--(1+\m,0);

\node at (-1+\m+\x,1) {$\bullet$};
\node at (-1+\m+\x,-1) {$\bullet$};
\node at (1+\m+\x,0) {$\bullet$};
\node[left] at (-1+\m+\x,1) {$x$};
\node[left] at (-1+\m+\x,-1) {$y$};
\node[above] at (1+\m+\x,0) {$z$};
\draw(-1+\m+\x,1)--(-1+\m+\x,-1)--(1+\m+\x,0);
\node at (\m+0.5*\x,0) {$+$};

\node at (-1+\m+2*\x,1) {$\bullet$};
\node at (-1+\m+2*\x,-1) {$\bullet$};
\node at (1+\m+2*\x,0) {$\bullet$};
\node[left] at (-1+\m+2*\x,1) {$x$};
\node[left] at (-1+\m+2*\x,-1) {$y$};
\node[above] at (1+\m+2*\x,0) {$z$};
\draw[bend left=30](-1+\m+2*\x,1)to(-1+\m+2*\x,-1);
\draw[bend right=30](-1+\m+2*\x,1)to(-1+\m+2*\x,-1);
\node at (\m+1.5*\x,0) {$+$};

\node at (-1+\m+3*\x+\eps,1) {$\bullet$};
\node at (-1+\m+3*\x+\eps,-1) {$\bullet$};
\node at (1+\m+3*\x+\eps,0) {$\bullet$};
\node[left] at (-1+\m+3*\x+\eps,1) {$x$};
\node[left] at (-1+\m+3*\x+\eps,-1) {$y$};
\node[above] at (1+\m+3*\x+\eps,0) {$z$};
\draw(-1+\m+3*\x+\eps,-1)--(-1+\m+3*\x+\eps,1)--(1+\m+3*\x+\eps,0);
\draw[->](\m+2.5*\x,0)--(\m+2.5*\x+\eps,0);
\node[above] at (\m+2.5*\x+0.5*\eps,0) {(4)};

\node at (-1+\m+4*\x+\eps,1) {$\bullet$};
\node at (-1+\m+4*\x+\eps,-1) {$\bullet$};
\node at (1+\m+4*\x+\eps,0) {$\bullet$};
\node[left] at (-1+\m+4*\x+\eps,1) {$x$};
\node[left] at (-1+\m+4*\x+\eps,-1) {$y$};
\node[above] at (1+\m+4*\x+\eps,0) {$z$};
\draw(-1+\m+4*\x+\eps,1)--(-1+\m+4*\x+\eps,-1)--(1+\m+4*\x+\eps,0);
\node at (\m+3.5*\x+\eps,0) {$+$};

\node at (-1+\m+5*\x+2*\eps,0) {$\bullet$};
\node at (1+\m+5*\x+2*\eps,0) {$\bullet$};
\node[above] at (-1+\m+5*\x+2*\eps,0) {$x,y$};
\node[above] at (1+\m+5*\x+2*\eps,0) {$z$};
\node at (\m+4.5*\x+1.5*\eps,0) {$+\ 2\cdot$};
\end{tikzpicture}

\begin{tikzpicture}[scale=0.3]
\def\x{5}
\def\m{8}
\def\eps{2}
\draw (0,0) circle (2.5);
\node at (0.2,0) {$\bullet$};
\node at (-1,1) {$\bullet$};
\node at (-1,-1) {$\bullet$};
\node at (1.7,0) {$\bullet$};
\node[above] at (-1,1) {$x$};
\node[below] at (-1,-1) {$y$};
\node[above] at (0.2,0) {$v$};
\node[above] at (1.7,0) {$z$};
\draw(-1,1)--(-1,-1)--(0.2,0)--(-1,1);
\draw(0.2,0)--(1.7,0);
\draw[->](3,0)--(\m-3,0);
\node[above] at (0.5*\m,0) {$m'$};

\node at (-1+\m,1) {$\bullet$};
\node at (-1+\m,-1) {$\bullet$};
\node at (1+\m,0) {$\bullet$};
\node[above] at (-1+\m,1) {$x,v$};
\node[left] at (-1+\m,-1) {$y$};
\node[above] at (1+\m,0) {$z$};
\draw(-1+\m,-1)--(-1+\m,1)--(1+\m,0);

\node at (-1+\m+\x,1) {$\bullet$};
\node at (-1+\m+\x,-1) {$\bullet$};
\node at (1+\m+\x,0) {$\bullet$};
\node[left] at (-1+\m+\x,1) {$x$};
\node[below] at (-1+\m+\x,-1) {$y,v$};
\node[above] at (1+\m+\x,0) {$z$};
\draw(-1+\m+\x,1)--(-1+\m+\x,-1)--(1+\m+\x,0);
\node at (\m+0.5*\x,0) {$+$};

\node at (-1+\m+2*\x,0) {$\bullet$};
\node at (\m+2*\x,0) {$\bullet$};
\node at (1+\m+2*\x,0) {$\bullet$};
\node[above] at (-1+\m+2*\x,0) {$x,y$};
\node[below] at (\m+2*\x,0) {$v$};
\node[above] at (1+\m+2*\x,0) {$z$};
\draw(-1+\m+2*\x,0)--(\m+2*\x,0)--(1+\m+2*\x,0);
\node at (\m+1.5*\x,0) {$+$};

\node at (-1+\m+3*\x+\eps,1) {$\bullet$};
\node at (-1+\m+3*\x+\eps,-1) {$\bullet$};
\node at (1+\m+3*\x+\eps,0) {$\bullet$};
\node[left] at (-1+\m+3*\x+\eps,1) {$x$};
\node[left] at (-1+\m+3*\x+\eps,-1) {$y$};
\node[above] at (1+\m+3*\x+\eps,0) {$z$};
\draw(-1+\m+3*\x+\eps,-1)--(-1+\m+3*\x+\eps,1)--(1+\m+3*\x+\eps,0);
\draw[->](\m+2.5*\x,0)--(\m+2.5*\x+\eps,0);
\node[above] at (\m+2.5*\x+0.5*\eps,0) {(6)};

\node at (-1+\m+4*\x+\eps,1) {$\bullet$};
\node at (-1+\m+4*\x+\eps,-1) {$\bullet$};
\node at (1+\m+4*\x+\eps,0) {$\bullet$};
\node[left] at (-1+\m+4*\x+\eps,1) {$x$};
\node[left] at (-1+\m+4*\x+\eps,-1) {$y$};
\node[above] at (1+\m+4*\x+\eps,0) {$z$};
\draw(-1+\m+4*\x+\eps,1)--(-1+\m+4*\x+\eps,-1)--(1+\m+4*\x+\eps,0);
\node at (\m+3.5*\x+\eps,0) {$+$};

\node at (-1+\m+5*\x+2*\eps,0) {$\bullet$};
\node at (1+\m+5*\x+2*\eps,0) {$\bullet$};
\node[above] at (-1+\m+5*\x+2*\eps,0) {$x,y$};
\node[above] at (1+\m+5*\x+2*\eps,0) {$z$};
\node at (\m+4.5*\x+1.5*\eps,0) {$+\ 2\cdot$};
\end{tikzpicture}
\caption{Case 3 of \cref{lem:alpha-H-well-defined}}
\label{fig:case3-of-alpha-H-well-defined}
\end{figure}

The other cases involving some combinations of moves (0) to (6) of \cref{def:alpha-H} are tedious and straightforward to check so we omit them here.
\end{proof}

Recall that for an edge-weighted graph $H$ whose edges have multiplicities, its weight $\wt(H)$ is given by the product of weights on the edges with multiplicity. For example, if some edge $e$ has multiplicity $2$, then it contributed $\wt(e)^2$ to $\wt(H)$. 
\begin{defin}\label{def:B-basis}
For $\xi\in\TNC_n$, define
\[B_{\xi}(\Gamma):=\sum_{H\subset 2\Gamma}\alpha(H)_{\xi}\wt(H)\]
for any cactus network $\Gamma$, where the sum is over double groves $H$ whose edges set is contained in that of $\Gamma$ and $\alpha(H)_{\xi}$ is defined in \cref{def:alpha-H}.
\end{defin}
\begin{ex}
Consider the electrical network $\Gamma$ shown here with $n=3$ that consists of an interior vertex $x$ connected to $\bar1,\bar2,\bar3$, with weights $a,b,c$ respectively:
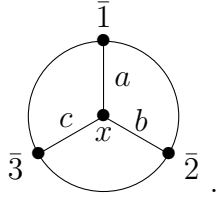
\[\begin{tikzpicture}[scale=0.5]
\draw (0,0) circle (2);
\node at (90:2) {$\bullet$};
\node at (210:2) {$\bullet$};
\node at (330:2) {$\bullet$};
\node at (90:2.7) {$\bar1$};
\node at (210:2.7) {$\bar3$};
\node at (330:2.7) {$\bar2$};
\node at (0,0) {$\bullet$};
\node[below] at (0,0) {$x$};
\draw(0,0)--(90:2);
\draw(0,0)--(210:2);
\draw(0,0)--(330:2);
\node[right] at (0,1) {$a$};
\node at (-1,-0.1) {$c$};
\node at (1,-0.1) {$b$};
\end{tikzpicture}.\]
Consider $\xi=\{(1,5),(2,6),(3,4)\}$, for which we denote as $(15|26|34)$ for simplicity. In order to compute $B_{\xi}(\Gamma)$, we need to sum over all $H\subset 2\Gamma$, where there are $27$ possibilities. However, notice that if $H$ has $\leq1$ edges, then (5) of \cref{def:alpha-H} implies $\alpha(H)=0$; and if $H$ has $\geq5$ edges counting with multiplicity, then after contracting once, some edge will have multiplicity $\geq3$, implying that $\alpha(H)=0$ as well. Here, we list the calculation of $\alpha(H)$, for those $H\subset2\Gamma$ such that $\alpha(H)_{\xi}\neq0$. 
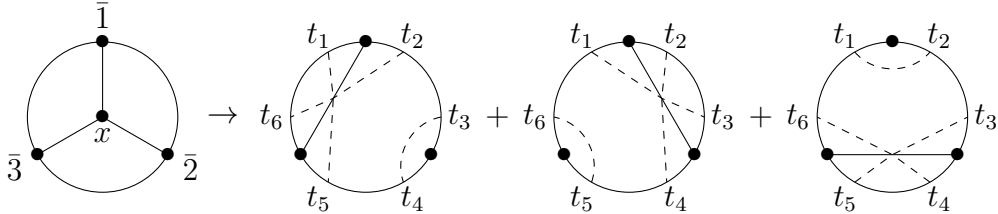
\[\begin{tikzpicture}[scale=0.5]
\def\x{7}
\draw (0,0) circle (2);
\node at (90:2) {$\bullet$};
\node at (210:2) {$\bullet$};
\node at (330:2) {$\bullet$};
\node at (90:2.7) {$\bar1$};
\node at (210:2.7) {$\bar3$};
\node at (330:2.7) {$\bar2$};
\node at (0,0) {$\bullet$};
\node[below] at (0,0) {$x$};
\draw(0,0)--(90:2);
\draw(0,0)--(210:2);
\draw(0,0)--(330:2);
\draw (\x,0) circle (2);
\node at (\x,2) {$\bullet$};
\node at (\x-2*0.86602540378,-1) {$\bullet$};
\node at (\x+2*0.86602540378,-1) {$\bullet$};
\draw(\x-2*0.86602540378,-1)--(\x,2);
\node at (\x-2.5*0.5,2.5*0.86602540378) {$t_1$};
\node at (\x+2.5*0.5,2.5*0.86602540378) {$t_2$};
\node at (\x+2.5,0) {$t_3$};
\node at (\x+2.5*0.5,-2.5*0.86602540378) {$t_4$};
\node at (\x-2.5*0.5,-2.5*0.86602540378) {$t_5$};
\node at (\x-2.5,0) {$t_6$};
\draw[dashed](\x-2*0.5,2*0.86602540378)--(\x-0.86602540378,0.5)--(\x-2*0.5,-2*0.86602540378);
\draw[dashed](\x+2*0.5,2*0.86602540378)--(\x-0.86602540378,0.5)--(\x-2,0);
\draw[bend right=50, dashed] (\x+2,0) to (\x+2*0.5,-2*0.86602540378);

\draw (2*\x,0) circle (2);
\node at (2*\x,2) {$\bullet$};
\node at (2*\x-2*0.86602540378,-1) {$\bullet$};
\node at (2*\x+2*0.86602540378,-1) {$\bullet$};
\draw(2*\x,2)--(2*\x+2*0.86602540378,-1);
\node at (2*\x-2.5*0.5,2.5*0.86602540378) {$t_1$};
\node at (2*\x+2.5*0.5,2.5*0.86602540378) {$t_2$};
\node at (2*\x+2.5,0) {$t_3$};
\node at (2*\x+2.5*0.5,-2.5*0.86602540378) {$t_4$};
\node at (2*\x-2.5*0.5,-2.5*0.86602540378) {$t_5$};
\node at (2*\x-2.5,0) {$t_6$};
\draw[dashed](2*\x+2*0.5,2*0.86602540378)--(2*\x+0.86602540378,0.5)--(2*\x+2*0.5,-2*0.86602540378);
\draw[dashed](2*\x-2*0.5,2*0.86602540378)--(2*\x+0.86602540378,0.5)--(2*\x+2,0);
\draw[bend right=50, dashed] (2*\x-2*0.5,-2*0.86602540378) to (2*\x-2,0);

\draw (3*\x,0) circle (2);
\node at (3*\x,2) {$\bullet$};
\node at (3*\x-2*0.86602540378,-1) {$\bullet$};
\node at (3*\x+2*0.86602540378,-1) {$\bullet$};
\draw(3*\x-2*0.86602540378,-1)--(3*\x+2*0.86602540378,-1);
\node at (3*\x-2.5*0.5,2.5*0.86602540378) {$t_1$};
\node at (3*\x+2.5*0.5,2.5*0.86602540378) {$t_2$};
\node at (3*\x+2.5,0) {$t_3$};
\node at (3*\x+2.5*0.5,-2.5*0.86602540378) {$t_4$};
\node at (3*\x-2.5*0.5,-2.5*0.86602540378) {$t_5$};
\node at (3*\x-2.5,0) {$t_6$};
\draw[dashed](3*\x+2*0.5,-2*0.86602540378)--(3*\x,-1)--(3*\x-2,0);
\draw[dashed](3*\x-2*0.5,-2*0.86602540378)--(3*\x,-1)--(3*\x+2,0);
\draw[bend right=50, dashed] (3*\x-2*0.5,2*0.86602540378) to (3*\x+2*0.5,2*0.86602540378);

\node at (0.5*\x-0.3,0) {$\rightarrow$};
\node at (1.5*\x,0) {$+$};
\node at (2.5*\x,0) {$+$};
\end{tikzpicture}.\]
This calculation shows that $\alpha(\Gamma)=(15|26|34)+(13|24|56)+(12|35|46)$ so $\alpha(\Gamma)_\xi=1$. We also have
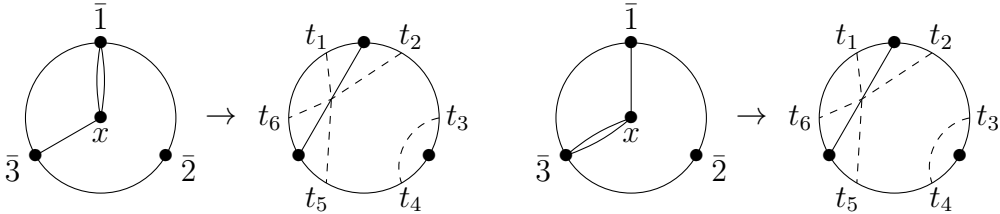
\[\begin{tikzpicture}[scale=0.5]
\def\x{7}
\draw (0,0) circle (2);
\node at (90:2) {$\bullet$};
\node at (210:2) {$\bullet$};
\node at (330:2) {$\bullet$};
\node at (90:2.7) {$\bar1$};
\node at (210:2.7) {$\bar3$};
\node at (330:2.7) {$\bar2$};
\node at (0,0) {$\bullet$};
\node[below] at (0,0) {$x$};
\draw[bend right=10](0,0)to(90:2);
\draw[bend left=10](0,0)to(90:2);
\draw(0,0)--(210:2);
\draw (\x,0) circle (2);
\node at (\x,2) {$\bullet$};
\node at (\x-2*0.86602540378,-1) {$\bullet$};
\node at (\x+2*0.86602540378,-1) {$\bullet$};
\draw(\x-2*0.86602540378,-1)--(\x,2);
\node at (\x-2.5*0.5,2.5*0.86602540378) {$t_1$};
\node at (\x+2.5*0.5,2.5*0.86602540378) {$t_2$};
\node at (\x+2.5,0) {$t_3$};
\node at (\x+2.5*0.5,-2.5*0.86602540378) {$t_4$};
\node at (\x-2.5*0.5,-2.5*0.86602540378) {$t_5$};
\node at (\x-2.5,0) {$t_6$};
\draw[dashed](\x-2*0.5,2*0.86602540378)--(\x-0.86602540378,0.5)--(\x-2*0.5,-2*0.86602540378);
\draw[dashed](\x+2*0.5,2*0.86602540378)--(\x-0.86602540378,0.5)--(\x-2,0);
\draw[bend right=50, dashed] (\x+2,0) to (\x+2*0.5,-2*0.86602540378);
\node at (0.5*\x-0.3,0) {$\rightarrow$};
\end{tikzpicture},\quad
\begin{tikzpicture}[scale=0.5]
\def\x{7}
\draw (0,0) circle (2);
\node at (90:2) {$\bullet$};
\node at (210:2) {$\bullet$};
\node at (330:2) {$\bullet$};
\node at (90:2.7) {$\bar1$};
\node at (210:2.7) {$\bar3$};
\node at (330:2.7) {$\bar2$};
\node at (0,0) {$\bullet$};
\node[below] at (0,0) {$x$};
\draw[bend right=10](0,0)to(210:2);
\draw[bend left=10](0,0)to(210:2);
\draw(0,0)--(90:2);
\draw (\x,0) circle (2);
\node at (\x,2) {$\bullet$};
\node at (\x-2*0.86602540378,-1) {$\bullet$};
\node at (\x+2*0.86602540378,-1) {$\bullet$};
\draw(\x-2*0.86602540378,-1)--(\x,2);
\node at (\x-2.5*0.5,2.5*0.86602540378) {$t_1$};
\node at (\x+2.5*0.5,2.5*0.86602540378) {$t_2$};
\node at (\x+2.5,0) {$t_3$};
\node at (\x+2.5*0.5,-2.5*0.86602540378) {$t_4$};
\node at (\x-2.5*0.5,-2.5*0.86602540378) {$t_5$};
\node at (\x-2.5,0) {$t_6$};
\draw[dashed](\x-2*0.5,2*0.86602540378)--(\x-0.86602540378,0.5)--(\x-2*0.5,-2*0.86602540378);
\draw[dashed](\x+2*0.5,2*0.86602540378)--(\x-0.86602540378,0.5)--(\x-2,0);
\draw[bend right=50, dashed] (\x+2,0) to (\x+2*0.5,-2*0.86602540378);
\node at (0.5*\x-0.3,0) {$\rightarrow$};
\end{tikzpicture}.\]
One can check that $\alpha(H)_{\xi}=0$ if $H$ is not the above three choices. As a result, we compute that $B_{\xi}(\Gamma)=abc+a^2c+ac^2$. 
\end{ex}

\subsection{Expansion of $\{L_{\sigma}L_{\sigma'}\}$ into $\{B_{\xi}\}$}

We now work towards lifting $B_{\xi}$'s to $G_{n,2}$ and showing that they form a basis of $G_{n,2}$. Recall that for a matching $M\in\TNC_n$ and a vector $v\in\{0,1\}^{\CR(M)}$, $M(v)\in\NCM_n$ is the noncrossing matching obtained from resolving the crossings in $M$ in the way specified by $v$, called a \textit{resolution of crossings} (see \cref{sec:combinatorics-3noncrossing}). We say that such a resolution $M(v)$ is \textit{valid} if no internal loops result from the uncrossings, i.e. all segments are used.
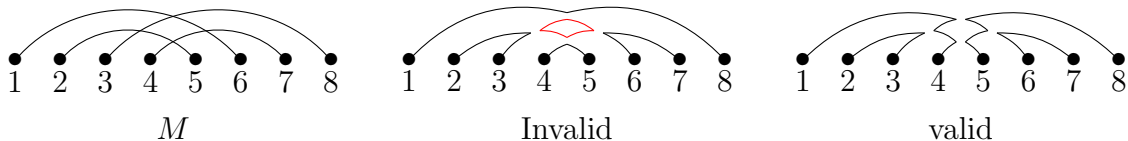
\begin{center}
\begin{tikzpicture}[scale=0.6]
\node at (1,0) {$\bullet$};
\node at (2,0) {$\bullet$};
\node at (3,0) {$\bullet$};
\node at (4,0) {$\bullet$};
\node at (5,0) {$\bullet$};
\node at (6,0) {$\bullet$};
\node at (7,0) {$\bullet$};
\node at (8,0) {$\bullet$};
\node[below] at (1,0) {$1$};
\node[below] at (2,0) {$2$};
\node[below] at (3,0) {$3$};
\node[below] at (4,0) {$4$};
\node[below] at (5,0) {$5$};
\node[below] at (6,0) {$6$};
\node[below] at (7,0) {$7$};
\node[below] at (8,0) {$8$};
\draw[bend left=50] (1,0) to (6,0);
\draw[bend left=50] (2,0) to (5,0);
\draw[bend left=50] (3,0) to (8,0);
\draw[bend left=50] (4,0) to (7,0);
\node at (4.5, -1.5) {$M$};
\end{tikzpicture} 
\quad
\begin{tikzpicture}[scale=0.6]
\node at (1,0) {$\bullet$};
\node at (2,0) {$\bullet$};
\node at (3,0) {$\bullet$};
\node at (4,0) {$\bullet$};
\node at (5,0) {$\bullet$};
\node at (6,0) {$\bullet$};
\node at (7,0) {$\bullet$};
\node at (8,0) {$\bullet$};
\node[below] at (1,0) {$1$};
\node[below] at (2,0) {$2$};
\node[below] at (3,0) {$3$};
\node[below] at (4,0) {$4$};
\node[below] at (5,0) {$5$};
\node[below] at (6,0) {$6$};
\node[below] at (7,0) {$7$};
\node[below] at (8,0) {$8$};
\draw[bend left=30] (1,0) to (4.5,1.05);
\draw[bend right=30] (8,0) to (4.5, 1.05);
\draw[bend left=20] (2,0) to (3.7,0.6);
\draw[bend left=15] (3,0) to (3.7, 0.6);
\draw[bend right=20] (7,0) to (5.3,0.6);
\draw[bend right=15] (6,0) to (5.3, 0.6);
\draw[bend left=15,color=red] (3.9,0.65) to (4.5, 0.9);
\draw[bend right=15,color=red] (5.1, 0.65) to (4.5, 0.9);
\draw[bend left=15,color=red] (3.9, 0.65) to (4.5, 0.5);
\draw[bend left=15,color=red] (4.5, 0.5) to (5.1, 0.65);
\draw[bend left=15] (4,0) to (4.5,0.35);
\draw[bend right=15] (5,0) to (4.5,0.35);
\node at (4.5, -1.5) {Invalid};
\end{tikzpicture}
\quad
\begin{tikzpicture}[scale=0.6]
\node at (1,0) {$\bullet$};
\node at (2,0) {$\bullet$};
\node at (3,0) {$\bullet$};
\node at (4,0) {$\bullet$};
\node at (5,0) {$\bullet$};
\node at (6,0) {$\bullet$};
\node at (7,0) {$\bullet$};
\node at (8,0) {$\bullet$};
\node[below] at (1,0) {$1$};
\node[below] at (2,0) {$2$};
\node[below] at (3,0) {$3$};
\node[below] at (4,0) {$4$};
\node[below] at (5,0) {$5$};
\node[below] at (6,0) {$6$};
\node[below] at (7,0) {$7$};
\node[below] at (8,0) {$8$};
\draw[bend left=30] (1,0) to (4.4,0.9);
\draw[bend right=30] (8,0) to (4.6, 0.9);
\draw[bend left=15] (3.9,0.65) to (4.4, 0.9);
\draw[bend right=15] (5.1, 0.65) to (4.6, 0.9);
\draw[bend left=15] (3.9, 0.65) to (4.4, 0.4);
\draw[bend left=15] (4.6, 0.4) to (5.1, 0.65);
\draw[bend left=15] (4,0) to (4.4, 0.4);
\draw[bend left=15] (4.6,0.4) to (5,0);
\draw[bend left=20] (2,0) to (3.7,0.6);
\draw[bend left=15] (3,0) to (3.7, 0.6);
\draw[bend right=20] (7,0) to (5.3,0.6);
\draw[bend right=15] (6,0) to (5.3, 0.6);
\node at (4.5, -1.5) {valid};
\end{tikzpicture}

\end{center}

Two loopless resolutions $M(v)$ and $M(v')$ are called \textit{opposite} if $v+v'=\mathbf{1}$. Note that $v+v'=\mathbf{1}$ means that the choices of uncrossings at every intersection in $M$ for $M(v)$ and $M(v')$ are different. 
\begin{defin}
For a $3$-noncrossing matching $\xi\in\TNC_n$, and a pair of noncrossing partitions $(\sigma,\sigma')\in\NP_n\times\NP_n$, define $a_{\xi,(\sigma,\sigma')}$ to be the number of valid opposite loopless resolution of $\xi$ that results in $(\sigma,\sigma')$. In other words,
\[a_{\xi,(\sigma,\sigma')}:=\#\{v\in\{0,1\}^{\CR(\xi)}\:|\: \xi(v)=\tau(\sigma),\ \xi(\mathbf{1}{-}v)=\tau(\sigma')\text{ are both valid}\}\]
where $\tau(\sigma)$ is the noncrossing matching that corresponds to the noncrossing partition $\sigma$ as in \cref{sec:catalan}.
\end{defin}

\begin{lemma}\label{lem:a-interpretation}
For $\xi\in\TNC_n$ and a pair of noncrossing partitions $(\sigma,\sigma')\in\NP_n\times\NP_n$, $a_{\xi,(\sigma,\sigma')}$ equals the number of ways of splitting $\Gamma(\xi)$ into two groves $F$ and $F'$ with boundary partition $\sigma$ and $\sigma'$ respectively.
\end{lemma}
\begin{proof}
Consider the cactus network $\Gamma=\Gamma(\xi)$ and the medial graph $G=G(\Gamma)$ of $\xi$ together. For each crossing $x\in\CR(\xi)$, there is a corresponding edge $e_x$ in $\Gamma$ that passes through the this crossing. Consider two opposite resolution $v$ and $v'$ such that $v+v'=\mathbf{1}^{\CR(\xi)}$. Correspondingly, we partition the edges of $\Gamma$ into $F\sqcup F'$ such that at each crossing $x\in\CR(\xi)$, if the resolution $\xi(v)$ does not intersect $e_x$, which looks like \begin{tikzpicture}[scale=0.4]\draw[bend right=50](0,0.4)to(1,0.4);\draw(0,0)--(1,0);\draw[bend left=50](0,-0.4)to(1,-0.4);\end{tikzpicture} , then we assign $e_x$ to $F$, otherwise assign $e_x$ to $F'$. This procedure is reversible, as any partition of the edge set of $\Gamma$ gives an opposite resolution.

We also see that if $F$ has a cycle, then $\xi(v)$ contains an interior loop inside this cycle. And if there is a connected component $C$ of $F$ not connected to the boundary, then $\xi(v)$ also contains an interior loop following edges not assigned to $F$ around $C$. Both situations are shown in Figure\cref{fig:invalid-resolution}. Conversely, if $\xi(v)$ contains an interior loop, by tracing through the crossings in $\xi$ and the corresponding edges in $\Gamma$ around this loop, we recover one of the two situations described previously. As a result, the opposite resolutions are valid if and only if both $F$ and $F'$ are groves. So we obtain the desired statement. 
\begin{figure}[h!]
\centering
\begin{tikzpicture}[scale=0.8]
\node (a) at (0,0) {$\bullet$};
\node (b) at (-0.2,3) {$\bullet$};
\node (c) at (1,4) {$\bullet$};
\node (d) at (3,3) {$\bullet$};
\node (e) at (4,3.5) {$\bullet$};
\node (f) at (4.5,0) {$\bullet$};
\draw(0,0)--(-0.2,3)--(1,4)--(3,3)--(4,3.5)--(4.5,0)--(0,0);
\node at (1.2,1.2) {$\bullet$};
\draw[dashed](0,0)--(1.2,1.2);
\node at (2,2.8) {$\bullet$};
\draw[dashed](3,3)--(2,2.8);
\node at (3,2) {$\bullet$};
\draw[dashed](3,3)--(3,2);
\node at (3,1) {$\bullet$};
\draw[dashed](4.5,0)--(3,1);

\coordinate (i1) at (0,1.5);
\coordinate (i2) at (0.6,0.6);
\coordinate (i3) at (2.25,0.1);
\coordinate (i4) at (3.75,0.5);
\coordinate (i5) at (4.15,1.75);
\coordinate (i6) at (3.5,3.15);
\coordinate (i7) at (3,2.5);
\coordinate (i8) at (2.5,2.9);
\coordinate (i9) at (1.9,3.4);
\coordinate (i10) at (0.4,3.4);
\draw[bend right=30,red] (i1) to (i2);
\draw[bend right=30,red] (i2) to (i3);
\draw[bend right=30,red] (i3) to (i4);
\draw[bend right=30,red] (i4) to (i5);
\draw[bend right=30,red] (i5) to (i6);
\draw[bend right=30,red] (i6) to (i7);
\draw[bend right=30,red] (i7) to (i8);
\draw[bend right=30,red] (i8) to (i9);
\draw[bend right=30,red] (i9) to (i10);
\draw[bend right=30,red] (i10) to (i1);
\end{tikzpicture}
\qquad
\begin{tikzpicture}[scale=0.8]
\coordinate (a) at (0,0);
\coordinate (b) at (1,-1);
\coordinate (c) at (2,-0.5);
\coordinate (d) at (2.5,0.5);
\coordinate (e) at (3,-1.5);
\node at (a) {$\bullet$};
\node at (b) {$\bullet$};
\node at (c) {$\bullet$};
\node at (d) {$\bullet$};
\node at (e) {$\bullet$};
\draw(a)--(b)--(c)--(d);
\draw(c)--(e);
\node at (-1,0) {$\bullet$};
\draw[dashed](a)--(-1,0);
\node at (0.5,1) {$\bullet$};
\draw[dashed](a)--(0.5,1);
\node at (0.5,-2) {$\bullet$};
\draw[dashed](b)--(0.5,-2);
\node at (2.5,1.5) {$\bullet$};
\draw[dashed](d)--(2.5,1.5);
\node at (3.5,0.5) {$\bullet$};
\draw[dashed](d)--(3.5,0.5);
\node at (4,-2) {$\bullet$};
\draw[dashed](e)--(4,-2);

\coordinate (i1) at (-0.5,0);
\coordinate (i2) at (0.4,-0.6);
\coordinate (i3) at (0.75,-1.5);
\coordinate (i4) at (1.6,-0.85);
\coordinate (i5) at (2.4,-1.1);
\coordinate (i6) at (3.5,-1.75);
\coordinate (i7) at (2.6,-0.9);
\coordinate (i8) at (2.35,-0.1);
\coordinate (i9) at (3,0.5);
\coordinate (i10) at (2.5,1);
\coordinate (i11) at (2.15,0.1);
\coordinate (i12) at (1.4,-0.65);
\coordinate (i13) at (0.6,-0.4);
\coordinate (i14) at (0.25,0.5);

\draw[bend right=30,red] (i1) to (i2);
\draw[bend right=30,red] (i2) to (i3);
\draw[bend right=30,red] (i3) to (i4);
\draw[bend right=30,red] (i4) to (i5);
\draw[bend right=50,red] (i5) to (i6);
\draw[bend right=50,red] (i6) to (i7);
\draw[bend right=30,red] (i7) to (i8);
\draw[bend right=30,red] (i8) to (i9);
\draw[bend right=30,red] (i9) to (i10);
\draw[bend right=30,red] (i10) to (i11);
\draw[bend right=30,red] (i11) to (i12);
\draw[bend right=30,red] (i12) to (i13);
\draw[bend right=30,red] (i13) to (i14);
\draw[bend right=30,red] (i14) to (i1);
\end{tikzpicture}
\caption{Invalid resolutions, where edges assigned to $F$ are solid, edges assigned to $F'$ are dashed, and internal loops in $\xi(v)$ are in red.}
\label{fig:invalid-resolution}
\end{figure}
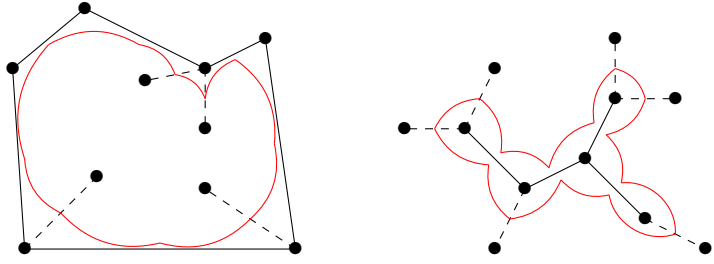
\end{proof}

\begin{theorem}\label{thm:L-into-B}
For $\sigma,\sigma'\in\NP_n$ and any cactus network $\Gamma$,
\begin{equation}\label{eq:L-into-B}
L_{\sigma}(\Gamma)L_{\sigma'}(\Gamma)=\sum_{\xi\in\TNC_n}a_{\xi,(\sigma,\sigma')}B_{\xi}(\Gamma).
\end{equation}
\end{theorem}
\begin{proof}
By definition of $B_{\xi}(\Gamma)$, the right hand side of \cref{eq:L-into-B} equals
\[\sum_{\xi\in\TNC_n}a_{\xi,(\sigma,\sigma')}\sum_{H\subset 2\Gamma}\alpha(H)_{\xi}\wt(H)=\sum_{H\subset2\Gamma}\wt(H)\sum_{\xi\in\TNC_n}a_{\xi,(\sigma,\sigma')}\alpha(H)_{\xi}.\]
Viewing every edge weight in $\Gamma$ as an indeterminate, we compare the coefficient of $\wt(H)$ on both sides for every $H$. For a fixed $H\subset 2\Gamma$, the coefficient of $\wt(H)$ on the left hand side $L_{\sigma}(\Gamma)L_{\sigma'}(\Gamma)$ is by definition the number of ways to partition $H$ into $F\sqcup F'$ such that $F$ is a grove with boundary partition $\sigma$ and $F'$ is a grove with boundary partition $\sigma'$. It suffices to show that this number equals $\sum_{\xi\in\TNC_n}a_{\xi,(\sigma,\sigma')}\alpha(H)_{\xi}$. The proof also explains the origin of \cref{def:alpha-H}.

To split $H$ into two groves $F\sqcup F'$, there are a few easy steps of reduction. If some edge of $H$ has multiplicity $\geq3$ (which cannot happen at the very beginning) or $H$ has a loop or $H$ has an interior leaf, then clearly no such splittings are possible. These situations correspond to moves (0), (2) and (5) of \cref{def:alpha-H}. Additionally, we can do the following reductions corresponding to moves (3), (4) and (6) of \cref{def:alpha-H}:
\begin{itemize}
    \item[-] If $H$ has a double edge $e$, then both $F$ and $F'$ need to contain $e$, so we can simply assign $e$ to both $F$ and $F'$ and then contract $e$. (Corresponds to move (3).)
    \item[-] If $H$ has two edges $e_1$ and $e_2$ between the same pair of vertices, then one needs to be assigned to $F$ and the other one to $F'$, so we can contract this edges and multiply the result by $2$. (Corresponds to move (4).)
    \item[-] If $H$ has more than two edges between the same pair of vertices, then no assignment is possible. (Corresponds to move (4).)
    \item[-] If $H$ has an interior vertex of degree $2$ with distinct neighbors, then these two edges must be assigned to $F$ and $F'$ respectively with both choices being symmetric and allowed. Therefore we can remove this interior vertex and multiply the result by $2$. (Corresponds to move (6).)
\end{itemize}

Next, if $H$ contains an interior vertex $v$ of degree $3$ connected to $v_1,v_2,v_3$, then to assign these three edges to $F$ and $F'$, we need to partition these three edges into a group of one and a group of two. There are three ways of doing so. If $(v,v_1)$ is assigned to one of $F$ and $F'$ and $(v,v_2)$ and $(v,v_3)$ to the other, then effectively we are removing $v$ (together with the three edges incident to it) from $H$, and adding an edge $(v_2,v_3)$ to it, to obtain a smaller graph $H_2$. For a partition of $H_2$ into two groves $F_2\sqcup F_2'$, we can recover a partition of $H$ into two groves $F\sqcup F'$ by assigning $(v,v_2)$ and $(v,v_3)$ to $F$ if $(v_2,v_3)$ in $H_2$ is assigned to $F_2$ and vice versa. This justifies move (7) of \cref{def:alpha-H}. The situation where $H$ contains a triangular face, which is move (8), is similar.

After all the reduction steps are performed, by \cref{prop:3-noncrossing}, we are left with a graph $H$ whose edges have multiplicity $1$ and $H$ is the medial graph of some $3$-noncrossing matching $\xi$. This is move (1) of \cref{def:alpha-H}. The integer $\alpha(H)_{\xi}$ is the number of ways that we end up in this situation. And the integer $a_{\xi,(\sigma,\sigma')}$ is the number of ways to do the partition into two groves as desired, by \cref{lem:a-interpretation}. Summing over $\xi\in\TNC_n$, we obtain the coefficient on the right hand side of \cref{eq:L-into-B}, as desired. 
\end{proof}

\begin{prop}\label{prop:B-basis}
For each $\xi\in\TNC_n$, we can lift $B_{\xi}$ to $G_{n,2}$. Moreover, $\{B_\xi\:|\: \xi\in\TNC_n\}$ is a basis of $G_{n,2}$.
\end{prop}
\begin{proof}
Fixing $n$, we first show that $B_{\xi}$'s are linearly independent. Recall that for each $\xi\in\TNC_n$, there is a corresponding cactus network $\Gamma(\xi)$ whose medial pairing is $\xi$, and that the number of edges of $\Gamma(\xi)$ equals $\#\CR(\xi)$, the number of crossings of $\xi$. Consider the computation of $\alpha(H)$ (\cref{def:alpha-H}) for some $H\subset 2\Gamma(\xi)$. Each reduction step in \cref{def:alpha-H} does not increase the number of edges. Moreover, if $H$ contains any double edges, then those edges can be removed by step (3) of \cref{def:alpha-H}. As a result, unless $H=\Gamma(\xi)$, $\alpha(H)\in\Z \TNC_n$ is a linear combination of 3-noncrossing matchings with strictly fewer crossings than $\xi$. And when $H=\Gamma(\xi)$, we have $\alpha(H)=\xi$. This means that $B_{\mu}(\Gamma(\xi))=0$ if $\#\CR(\xi)\leq\#\CR(\mu)$ and $B_{\xi}(\Gamma(\xi))=\wt(\Gamma(\xi))\neq0$. 

If we have a nontrivial relation $\sum c_{\xi}B_{\xi}=0$, by choosing the $\xi$ with the smallest number of crossings with $c_{\xi}\neq0$ and evaluating this relation $\sum c_{\xi}B_{\xi}=0$ on the cactus network $\Gamma(\xi)$, we obtain a contradiction. Thus, $\{B_\xi\:|\: \xi\in\TNC_n\}$ is linearly independent. In light of \cref{thm:L-into-B}, this means that the matrix $M=\{a_{\xi,(\sigma,\sigma')}\}$ whose rows are indexed by $\xi\in\TNC_n$ and columns indexed by $(\sigma,\sigma')\in\NP_n\times\NP_n$ (which has way more columns than rows) has full rank. We can pick any representative for $B_{\xi}\in G_{n,2}$ as a linear combination of $L_{\sigma}L_{\sigma}$'s by choosing a maximal invertible submatrix of $M$ and invert it. 

As elements in $G_{n,2}$, $\{B_\xi\:|\: \xi\in\TNC_n\}$ is a basis since they are linearly independent and they span $G_{n,2}$ (\cref{thm:L-into-B}).
\end{proof}

\section{Connections to the Temperley–Lieb immanant}\label{sec:F-basis}
The \emph{Temperley-Lieb immanant} $F_{\tau,T}$ is defined by the second author \cite{lam2015dimers} as a function on a planar bipartite graph, or equivalently a function on the affine cone over the Grassmannian $\Gr(k,n)$. However, we take an orthogonal and self-contained approach and define $F_{\tau,T}\in G_{n,2}$ as a linear combination of the Bush basis $\{B_{\xi}\:|\: \xi\in\TNC_n\}$.

\subsection{Temperley-Lieb immanants in terms of the Bush basis}
\begin{defin}
A $(k,m)$-\emph{partial noncrossing matching} is a pair $(\tau,T)$, such that $\tau$ is a matching of a subset $S(\tau)=S\subset\{1,2,\ldots,m\}$ of even size that is noncrossing, i.e. if $[m]$ are arranged on a circle in order, then edges formed by $\tau$ won't cross, and that $T\subset[m]\setminus S$ with $|S|+2|T|=2k$. 
\end{defin}
Denote the set of $(k,m)$-partial noncrossing matchings as $\mathcal{A}_{k,m}$.  In the following, we take the parameters $k=n-1$ and $m=2n$.

The framework of the next construction is similar to Section 5 of \cite{lam2004electroid}. For each $3$-noncrossing matching $\xi\in\TNC_n$, we define $\beta(\xi)\in\mathbb{Z}\mathcal{A}_{n-1,2n}$ as follows.

Draw $\xi$ on a disk as a medial graph with $n$ strands on $t_1,\ldots,t_{2n}$. We recover the information of cactus networks $\Gamma=\Gamma(\xi)$ and $\Gamma^{\vee}$ together in a single cactus, with a procedure similar to the process described in \cref{sub:medial}. Label the boundary vertices as $1,2,\ldots,2n$ such that $i$ is between $t_{i}$ and $t_{i+1}$ in clockwise order. These boundary vertices are colored black. The strands in $\xi$ divide the disk into regions. For each region, if there are boundary vertices in this region, we identify them together; and if there are no boundary vertices in this region, we add a black vertex for this interior region. These identifications result in a cactus $S_{\zeta}$. Make the intersections of strands in $\xi$ into white vertices and connect each white vertex to the four black vertices that represent the four faces incident to it. Denote this cactus graph as $N(\xi)$. See \cref{fig:N-xi} for an example.
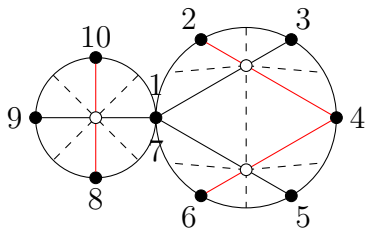
\begin{figure}[h!]
\centering
\begin{tikzpicture}[scale=0.4]
\def\r{3};
\def\rr{2};
\draw (0,0) circle (\r);
\draw (-\r-\rr,0) circle (\rr);
\node[circle, draw=black, fill=black, minimum size=1, scale=0.4,] (2) at (120:\r) {};
\node[circle, draw=black, fill=black, minimum size=1, scale=0.4,] (1) at (180:\r) {};
\node[circle, draw=black, fill=black, minimum size=1, scale=0.4,] (6) at (240:\r) {};
\node[circle, draw=black, fill=black, minimum size=1, scale=0.4,] (5) at (300:\r) {};
\node[circle, draw=black, fill=black, minimum size=1, scale=0.4,] (4) at (0:\r) {};
\node[circle, draw=black, fill=black, minimum size=1, scale=0.4,] (3) at (60:\r) {};
\node[circle, draw=black, fill=black, minimum size=1, scale=0.4,] (10) at (-\r-\rr,\rr) {};
\node[circle, draw=black, fill=black, minimum size=1, scale=0.4,] (8) at (-\r-\rr,-\rr) {};
\node[circle, draw=black, fill=black, minimum size=1, scale=0.4,] (9) at (-\r-2*\rr,0) {};
\draw[red](10)--(8);
\draw(9)--(1);
\draw(5)--(1)--(3);
\draw[red](2)--(4)--(6);
\draw[dashed](-\r-\rr-0.70710678118*\rr,0.70710678118*\rr)to(-\r-\rr+0.70710678118*\rr,-0.70710678118*\rr);
\draw[dashed](-\r-\rr+0.70710678118*\rr,0.70710678118*\rr)to(-\r-\rr-0.70710678118*\rr,-0.70710678118*\rr);
\draw[dashed](0,0.57735026919*\r)to(30:\r);
\draw[dashed](0,0.57735026919*\r)to(150:\r);
\draw[dashed](0,-0.57735026919*\r)to(210:\r);
\draw[dashed](0,-0.57735026919*\r)to(330:\r);
\draw[dashed](0,\r)--(0,-\r);

\node[circle, draw=black, fill=white, minimum size=1, scale=0.4,] at (-\r-\rr,0) {};
\node[circle, draw=black, fill=white, minimum size=1, scale=0.4,] (i1) at (0,0.57735026919*\r) {};
\node[circle, draw=black, fill=white, minimum size=1, scale=0.4,] (i2) at (0,-0.57735026919*\r) {};

\node at (-\r,1.1) {$1$};
\node at (-\r,-1.1) {$7$};
\node at (0:\r+0.7) {$4$};
\node at (60:\r+0.8) {$3$};
\node at (120:\r+0.8) {$2$};
\node at (240:\r+0.8) {$6$};
\node at (300:\r+0.8) {$5$};
\node at (-\r-2*\rr-0.7,0) {$9$};
\node at (-\r-\rr,\rr+0.7) {$10$};
\node at (-\r-\rr,-\rr-0.7) {$8$};
\end{tikzpicture}
\caption{For $\xi=\{(1,9),(2,4),(3,6),(5,7),(8,10)\}$, the bipartite graph $N(\xi)$ with $\xi$ draw in dashed lines, $\Gamma$ in black and $\Gamma^{\vee}$ in red}
\label{fig:N-xi}
\end{figure}

Note that the graph $N(\xi)$ can be viewed as $\Gamma=\Gamma(\xi)$ and $\Gamma^{\vee}$ lying on top of each other, with intersections being the white vertices in $N(\xi)$. Each edge of $\Gamma$ and $\Gamma^{\vee}$ consists of two edges in $N(\xi)$ with endpoints having different colors, called ``half-edges" in $\Gamma$ and $\Gamma^{\vee}$

Let $Z_1\sqcup\cdots\sqcup Z_r$ be the corresponding noncrossing partition of $[2n]$ so that we have $r$ distinct vertices on the boundary of the cactus $S_{\zeta}$. Consider all the possible choices of $A$ that consists of the following data:
\begin{enumerate}
\item For every edge in $\Gamma$ and $\Gamma^{\vee}$, choose one of the two half-edges so that every interior vertex has degree $2$ and every new boundary vertex $Z_i$ has degree at most $2$.
\item The chosen edges can be partitioned into vertex-disjoint cycles and paths, whose endpoints are on the boundary. For each new boundary vertex $Z_i$ of degree $0$, choose one vertex $j\in Z_i$ and put the rest $Z_i\setminus\{j\}$ in $T(A)$; for each $Z_i$ and $Z_{i'}$ of degree $1$ that are the endpoints of a path, choose $j\in Z_i$ and $j'\in Z_{i'}$, put this pair $(j,j')$ in $\tau(A)$, and put the rest $Z_i\setminus\{j\}$ and $Z_{i'}\setminus\{j'\}$ in $T(A)$; for each $Z_i$ of degree $2$, put $Z_i$ in $T(A)$. Let $\cyc(A)$ be the number of cycles in $A$.
\end{enumerate}

\begin{lemma}
For a choice $A$ as above, $|\tau(A)|+2|T(A)|=2n-2$.
\end{lemma}
\begin{proof}
Keep the notations as above. The cactus $S_{\zeta}$ contains $1+2n-r$ circles, split into $(1+2n-r)+n+\#\CR(\xi)$ regions by the $n$ strands in $\xi$. Among these regions, there are $2n$ of them coming from the original boundary vertices, so we have $1+n+\#\CR(\xi)-r$ interior regions, giving rise to $1+n+\#\CR(\xi)-r$ interior black vertices in $N(\xi)$. There are $\#\CR(\xi)$ interior white vertices, so we have $1+n+2\#\CR(\xi)-r$ interior vertices and $r$ boundary vertices in $N(\xi)$ in total. From $A$, we selected $2\#\CR(\xi)$ edges, contributing $4\#\CR(\xi)$ to the total degree of all vertices. All interior vertices have degree $2$, so the sum of degrees of boundary vertices is $2r-2n-2$. Assume that $2s$ of them have degree $1$, that $r-n-s-1$ of them have degree $2$, and that $n-s+1$ of them have degree $0$. We then have $|\tau(A)|=2s$. Also by construction, we put every boundary vertex in $T(A)$ except one vertex of our choice from every new boundary vertex $Z_i$ with degree $\leq1$. Thus, $T(A)=(2n)-(n+s+1)=n-s-1$. This means $|\tau(A)|+2|T(A)|=2n-2$ as desired.
\end{proof}

Then we define \[\beta(\xi):=\sum_{A}2^{\cyc(A)}(\tau(A),T(A))\in\mathbb{Z}\mathcal{A}_{n-1,2n}\] where the sum is over all selections $A$ as above.

\begin{ex}
First consider an extreme example of $\xi=\{(1,2),(3,4),(5,6)\}$ with $n=3$. The cactus $S_{\zeta}$ identifies $2,4,6$ together. Since $\xi$ is noncrossing, we do not need to select any half-edges, and the only choice that we make is choosing a $2$ element subset of $\{2,4,6\}$ to be $T(A)$. As a result,
\[\beta(\{(1,2),(3,4),(5,6)\})=(\emptyset,\{2,4\})+(\emptyset,\{2,6\})+(\emptyset,\{4,6\}).\]
\end{ex}

\begin{ex}
Consider $\xi=\{(1,3),(2,5),(4,6)\}$ with $2$ crossings, with $N(\xi)$ shown here
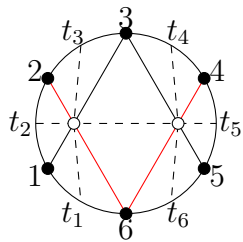
\begin{center}
\begin{tikzpicture}[scale=0.4]
\def\r{3};
\draw (0,0) circle (\r);
\node[circle, draw=black, fill=black, minimum size=1, scale=0.4,] (4) at (30:\r) {};
\node[circle, draw=black, fill=black, minimum size=1, scale=0.4,] (3) at (90:\r) {};
\node[circle, draw=black, fill=black, minimum size=1, scale=0.4,] (2) at (150:\r) {};
\node[circle, draw=black, fill=black, minimum size=1, scale=0.4,] (1) at (210:\r) {};
\node[circle, draw=black, fill=black, minimum size=1, scale=0.4,] (6) at (270:\r) {};
\node[circle, draw=black, fill=black, minimum size=1, scale=0.4,] (5) at (330:\r) {};
\draw(1)--(3)--(5);
\draw[red](2)--(6)--(4);
\coordinate (i2) at (\r*0.57735026919,0);
\coordinate (i1) at (-\r*0.57735026919,0);
\draw[dashed](120:\r)to(i1)to(240:\r);
\draw[dashed](60:\r)to(i2)to(300:\r);
\draw[dashed](-\r,0)to(\r,0);
\node at (30:\r+0.5) {$4$};
\node at (90:\r+0.5) {$3$};
\node at (150:\r+0.5) {$2$};
\node at (210:\r+0.5) {$1$};
\node at (270:\r+0.5) {$6$};
\node at (330:\r+0.5) {$5$};
\node at (0:\r+0.5) {$t_5$};
\node at (60:\r+0.5) {$t_4$};
\node at (120:\r+0.5) {$t_3$};
\node at (180:\r+0.5) {$t_2$};
\node at (240:\r+0.5) {$t_1$};
\node at (300:\r+0.5) {$t_6$};

\node[circle, draw=black, fill=white, minimum size=1, scale=0.4,] at (i1) {};
\node[circle, draw=black, fill=white, minimum size=1, scale=0.4,] at (i2) {};
\end{tikzpicture}
\end{center}
All the $2^4=16$ selection of the half-edges are valid and we pick a few of them here:
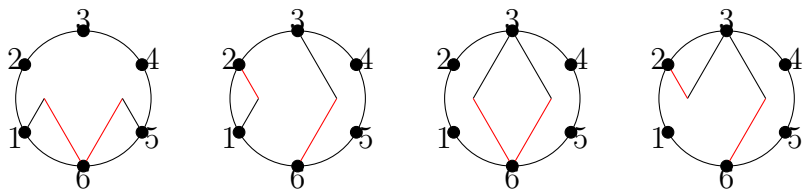
\begin{center}
\begin{tikzpicture}[scale=0.3]
\def\r{3};
\draw (0,0) circle (\r);
\node[circle, draw=black, fill=black, minimum size=1, scale=0.4,] (4) at (30:\r) {};
\node[circle, draw=black, fill=black, minimum size=1, scale=0.4,] (3) at (90:\r) {};
\node[circle, draw=black, fill=black, minimum size=1, scale=0.4,] (2) at (150:\r) {};
\node[circle, draw=black, fill=black, minimum size=1, scale=0.4,] (1) at (210:\r) {};
\node[circle, draw=black, fill=black, minimum size=1, scale=0.4,] (6) at (270:\r) {};
\node[circle, draw=black, fill=black, minimum size=1, scale=0.4,] (5) at (330:\r) {};
\node at (30:\r+0.5) {$4$};
\node at (90:\r+0.5) {$3$};
\node at (150:\r+0.5) {$2$};
\node at (210:\r+0.5) {$1$};
\node at (270:\r+0.5) {$6$};
\node at (330:\r+0.5) {$5$};
\coordinate (i2) at (\r*0.57735026919,0);
\coordinate (i1) at (-\r*0.57735026919,0);
\draw(1)--(i1);
\draw[red](i1)--(6);
\draw[red](i2)--(6);
\draw(i2)--(5);
\end{tikzpicture}
\quad
\begin{tikzpicture}[scale=0.3]
\def\r{3};
\draw (0,0) circle (\r);
\node[circle, draw=black, fill=black, minimum size=1, scale=0.4,] (4) at (30:\r) {};
\node[circle, draw=black, fill=black, minimum size=1, scale=0.4,] (3) at (90:\r) {};
\node[circle, draw=black, fill=black, minimum size=1, scale=0.4,] (2) at (150:\r) {};
\node[circle, draw=black, fill=black, minimum size=1, scale=0.4,] (1) at (210:\r) {};
\node[circle, draw=black, fill=black, minimum size=1, scale=0.4,] (6) at (270:\r) {};
\node[circle, draw=black, fill=black, minimum size=1, scale=0.4,] (5) at (330:\r) {};
\node at (30:\r+0.5) {$4$};
\node at (90:\r+0.5) {$3$};
\node at (150:\r+0.5) {$2$};
\node at (210:\r+0.5) {$1$};
\node at (270:\r+0.5) {$6$};
\node at (330:\r+0.5) {$5$};
\coordinate (i2) at (\r*0.57735026919,0);
\coordinate (i1) at (-\r*0.57735026919,0);
\draw(1)--(i1);
\draw[red](i1)--(2);
\draw[red](i2)--(6);
\draw(i2)--(3);
\end{tikzpicture}
\quad
\begin{tikzpicture}[scale=0.3]
\def\r{3};
\draw (0,0) circle (\r);
\node[circle, draw=black, fill=black, minimum size=1, scale=0.4,] (4) at (30:\r) {};
\node[circle, draw=black, fill=black, minimum size=1, scale=0.4,] (3) at (90:\r) {};
\node[circle, draw=black, fill=black, minimum size=1, scale=0.4,] (2) at (150:\r) {};
\node[circle, draw=black, fill=black, minimum size=1, scale=0.4,] (1) at (210:\r) {};
\node[circle, draw=black, fill=black, minimum size=1, scale=0.4,] (6) at (270:\r) {};
\node[circle, draw=black, fill=black, minimum size=1, scale=0.4,] (5) at (330:\r) {};
\node at (30:\r+0.5) {$4$};
\node at (90:\r+0.5) {$3$};
\node at (150:\r+0.5) {$2$};
\node at (210:\r+0.5) {$1$};
\node at (270:\r+0.5) {$6$};
\node at (330:\r+0.5) {$5$};
\coordinate (i2) at (\r*0.57735026919,0);
\coordinate (i1) at (-\r*0.57735026919,0);
\draw(3)--(i1);
\draw[red](i1)--(6);
\draw[red](i2)--(6);
\draw(i2)--(3);
\end{tikzpicture}
\quad
\begin{tikzpicture}[scale=0.3]
\def\r{3};
\draw (0,0) circle (\r);
\node[circle, draw=black, fill=black, minimum size=1, scale=0.4,] (4) at (30:\r) {};
\node[circle, draw=black, fill=black, minimum size=1, scale=0.4,] (3) at (90:\r) {};
\node[circle, draw=black, fill=black, minimum size=1, scale=0.4,] (2) at (150:\r) {};
\node[circle, draw=black, fill=black, minimum size=1, scale=0.4,] (1) at (210:\r) {};
\node[circle, draw=black, fill=black, minimum size=1, scale=0.4,] (6) at (270:\r) {};
\node[circle, draw=black, fill=black, minimum size=1, scale=0.4,] (5) at (330:\r) {};
\node at (30:\r+0.5) {$4$};
\node at (90:\r+0.5) {$3$};
\node at (150:\r+0.5) {$2$};
\node at (210:\r+0.5) {$1$};
\node at (270:\r+0.5) {$6$};
\node at (330:\r+0.5) {$5$};
\coordinate (i2) at (\r*0.57735026919,0);
\coordinate (i1) at (-\r*0.57735026919,0);
\draw(3)--(i1);
\draw[red](i1)--(2);
\draw[red](i2)--(6);
\draw(i2)--(3);
\end{tikzpicture}
\end{center}
with contributions
\[\beta(\xi)=(\{1,5\},\{6\})+(\{1,2,3,6\},\emptyset)+2(\emptyset,\{3,6\})+(\{2,6\},\{3\})+\cdots\]
where the third term has a coefficient $2$ because there is a cycle.
\end{ex}

\begin{defin}\label{def:F_tau_T}
Let $(\tau,T)\in\mathcal{A}_{n-1,2n}$. Define
\[F_{\tau,T}=\sum_{\xi\in\TNC_n}\beta(\xi)_{\tau,T}B_{\xi}\]
where $\beta(\xi)_{\tau,T}\in\mathbb{Z}_{\geq0}$ is the coefficient of $(\tau,T)$ in $\beta(\xi)$.
\end{defin}

\subsection{Concordance}
We use the notion of concordance from Section 5 of \cite{lam2004electroid}, to be explained further in \cref{sec:geometry}.  In \cref{prop:tfy_relations}, we will give a new derivation of the ``electrical" Pl\"ucker relations (Proposition 5.35 of \cite{lam2004electroid}). 

We say that an $(n-1)$-element subset $I\in \binom{[2n]}{n-1}$ is \emph{concordant} with a noncrossing partition $\sigma$ if each part of $\sigma$ and each part of the dual partition $\Tilde{\sigma}$ contains exactly one element not in $I$, viewing $\sigma$ as a noncrossing partition on $1,3,\ldots,2n-1$ and $\Tilde{\sigma}$ on $2,4,\ldots,2n$. In general, we also say $\sigma$ is concordant with $I$, or $(\sigma, \Tilde{\sigma})$ is concordant with $I$. Given $I\in \binom{[2n]}{n-1}$, we use $\cE(I)\subset \NP_n$ to denote the set of noncrossing partitions concordant with $I$. 
\begin{ex}
Let $\sigma = (1,2,5\mid 3,4\mid 6)$, $\Tilde{\sigma} = (\Tilde{1}\mid \Tilde{2}, \Tilde{4}\mid \Tilde{3}\mid \Tilde{4},\Tilde{5})$. Then $\sigma$ is concordant with the subset $\{3,4,7,9,12\}$, but not concordant with the subset $\{1,3,7,9,12\}$. See \cref{fig:concordant_subsets}. We also draw the Dyck path interpretation for later purposes (\cref{sec:tableaux}).
\begin{figure}[h!]
    \centering
    \begin{tikzpicture}[scale=0.4]
    \def\r{4};
    \draw (0,0) circle (\r);
    \node[circle, draw=black, fill=black, minimum size=1, scale=0.4,] (7) at (0:\r) {};
    \node[draw, regular polygon,regular polygon sides=4] at (0:\r) {};
    \node[circle, draw=black, fill=white, minimum size=1, scale=0.4,] (6) at (30:\r) {};
    \node[circle, draw=black, fill=black, minimum size=1, scale=0.4,] (5) at (60:\r) {};
    \node[circle, draw=black, fill=white, minimum size=1, scale=0.4,] (4) at (90:\r) {};
    \node[draw, regular polygon,regular polygon sides=4] at (90:\r) {};
    \node[circle, draw=black, fill=black, minimum size=1, scale=0.4,] (3) at (120:\r) {};
    \node[draw, regular polygon,regular polygon sides=4] at (120:\r) {};
    \node[circle, draw=black, fill=white, minimum size=1, scale=0.4,] (2) at (150:\r) {};
    \node[circle, draw=black, fill=black, minimum size=1, scale=0.4,] (1) at (180:\r) {};
    \node[circle, draw=black, fill=white, minimum size=1, scale=0.4,] (12) at (210:\r) {};
    \node[draw,regular polygon,regular polygon sides=4] at (210:\r) {};
    \node[circle, draw=black, fill=black, minimum size=1, scale=0.4,] (11) at (240:\r) {};
    \node[circle, draw=black, fill=white, minimum size=1, scale=0.4,] (10) at (270:\r) {};
    \node[circle, draw=black, fill=black, minimum size=1, scale=0.4,] (9) at (300:\r) {};
    \node[draw, regular polygon,regular polygon sides=4] at (300:\r) {};
    \node[circle, draw=black, fill=white, minimum size=1, scale=0.4,] (8) at (330:\r) {};
    
    \node at (0:\r+1) {7};
    \node at (30:\r+1) {6};
    \node at (60:\r+1) {5};
    \node at (90:\r+1) {4};
    \node at (120:\r+1) {3};
    \node at (150:\r+1) {2};
    \node at (180:\r+1) {1};
    \node at (210:\r+1) {12};
    \node at (240:\r+1) {11};
    \node at (270:\r+1) {10};
    \node at (300:\r+1) {9};
    \node at (330:\r+1) {8};
    
    \draw (1) -- (3) -- (9) -- (1);
    \draw (5) -- (7);
    \draw (12) -- (10);
    \draw (4) -- (8);
    \node at (0,-\r-2.5) {Concordant};
    \end{tikzpicture}
    $\quad\quad$
    \begin{tikzpicture}[scale=0.4]
    \def\r{4};
    \draw (0,0) circle (\r);
    \node[circle, draw=black, fill=black, minimum size=1, scale=0.4,] (7) at (0:\r) {};
    \node[draw, regular polygon,regular polygon sides=4] at (0:\r) {};
    \node[circle, draw=black, fill=white, minimum size=1, scale=0.4,] (6) at (30:\r) {};
    \node[circle, draw=black, fill=black, minimum size=1, scale=0.4,] (5) at (60:\r) {};
    \node[circle, draw=black, fill=white, minimum size=1, scale=0.4,] (4) at (90:\r) {};
    \node[circle, draw=black, fill=black, minimum size=1, scale=0.4,] (3) at (120:\r) {};
    \node[draw, regular polygon,regular polygon sides=4] at (120:\r) {};
    \node[circle, draw=black, fill=white, minimum size=1, scale=0.4,] (2) at (150:\r) {};
    \node[circle, draw=black, fill=black, minimum size=1, scale=0.4,] (1) at (180:\r) {};
    \node[draw, regular polygon,regular polygon sides=4] at (180:5) {};
    \node[circle, draw=black, fill=white, minimum size=1, scale=0.4,] (12) at (210:\r) {};
    \node[draw,regular polygon,regular polygon sides=4] at (210:\r) {};
    \node[circle, draw=black, fill=black, minimum size=1, scale=0.4,] (11) at (240:\r) {};
    \node[circle, draw=black, fill=white, minimum size=1, scale=0.4,] (10) at (270:\r) {};
    \node[circle, draw=black, fill=black, minimum size=1, scale=0.4,] (9) at (300:\r) {};
    \node[draw, regular polygon,regular polygon sides=4] at (300:\r) {};
    \node[circle, draw=black, fill=white, minimum size=1, scale=0.4,] (8) at (330:\r) {};
    
    \node at (0:\r+1) {7};
    \node at (30:\r+1) {6};
    \node at (60:\r+1) {5};
    \node at (90:\r+1) {4};
    \node at (120:\r+1) {3};
    \node at (150:\r+1) {2};
    \node at (180:\r+1) {1};
    \node at (210:\r+1) {12};
    \node at (240:\r+1) {11};
    \node at (270:\r+1) {10};
    \node at (300:\r+1) {9};
    \node at (330:\r+1) {8};
    
    \draw (1) -- (3) -- (9) -- (1);
    \draw (5) -- (7);
    \draw (12) -- (10);
    \draw (4) -- (8);
    \node at (0,-\r-2.5) {\textbf{NOT} Concordant};
    \end{tikzpicture}
    
    \begin{tikzpicture}[scale=0.4]
    \def\x{0.7}
     \draw(0,0)--(2,2)--(3,1)--(6,4)--(10,0)--(11,1)--(12,0);
    \node at (0,0) {$\bullet$};
    \node[circle, draw=black, fill=white, minimum size=1, scale=0.4,] at (1,1) {};
    \node[circle, draw=black, fill=black, minimum size=1, scale=0.4,] at (2,2) {};
    \node[circle, draw=black, fill=white, minimum size=1, scale=0.4,] at (3,1) {};
    \node[draw, regular polygon,regular polygon sides=4,scale=0.8] at (3,1) {};
    \node[circle, draw=black, fill=black, minimum size=1, scale=0.4,] at (4,2) {};
    \node[draw, regular polygon,regular polygon sides=4,scale=0.8] at (4,2) {};
    \node[circle, draw=black, fill=white, minimum size=1, scale=0.4,] at (5,3) {};
    \node[circle, draw=black, fill=black, minimum size=1, scale=0.4,] at (6,4) {};
    \node[circle, draw=black, fill=white, minimum size=1, scale=0.4,] at (7,3) {};
    \node[draw, regular polygon,regular polygon sides=4,scale=0.8] at (7,3) {};
    \node[circle, draw=black, fill=black, minimum size=1, scale=0.4,] at (8,2) {};
    \node[circle, draw=black, fill=white, minimum size=1, scale=0.4,] at (9,1) {};
    \node[draw, regular polygon,regular polygon sides=4,scale=0.8] at (9,1) {};
    \node[circle, draw=black, fill=black, minimum size=1, scale=0.4,] at (10,0) {};
    \node[circle, draw=black, fill=white, minimum size=1, scale=0.4,] at (11,1) {};
    \node[circle, draw=black, fill=black, minimum size=1, scale=0.4,] at (12,0) {};
    \node[draw, regular polygon,regular polygon sides=4,scale=0.8] at (12,0) {};
   
    \node at (1,1+\x) {1};
    \node at (2,2+\x) {2};
    \node at (3,1+\x) {3};
    \node at (4,2+\x) {4};
    \node at (5,3+\x) {5};
    \node at (6,4+\x) {6};
    \node at (7,3+\x) {7};
    \node at (8,2+\x) {8};
    \node at (9,1+\x) {9};
    \node at (10,0+\x) {10};
    \node at (11,1+\x) {11};
    \node at (12,0+\x) {12};
    
    \draw[dotted] (0,0) -- (12,0);
    \draw[dotted] (1,1) -- (9,1);
    \draw[dotted] (4,2) -- (8,2);
    \draw[dotted] (5,3) -- (7,3);
    
    \node at (6.5,-1) {Concordant};
    \end{tikzpicture}
    $\quad\quad$
     \begin{tikzpicture}[scale=0.4]
    \def\x{0.7}
     \draw(0,0)--(2,2)--(3,1)--(6,4)--(10,0)--(11,1)--(12,0);
    \node at (0,0) {$\bullet$};
    \node[circle, draw=black, fill=white, minimum size=1, scale=0.4,] at (1,1) {};
    \node[draw, regular polygon,regular polygon sides=4,scale=0.8] at (1,1) {};
    \node[circle, draw=black, fill=black, minimum size=1, scale=0.4,] at (2,2) {};
    \node[circle, draw=black, fill=white, minimum size=1, scale=0.4,] at (3,1) {};
    \node[draw, regular polygon,regular polygon sides=4,scale=0.8] at (3,1) {};
    \node[circle, draw=black, fill=black, minimum size=1, scale=0.4,] at (4,2) {};
    
    \node[circle, draw=black, fill=white, minimum size=1, scale=0.4,] at (5,3) {};
    \node[circle, draw=black, fill=black, minimum size=1, scale=0.4,] at (6,4) {};
    \node[circle, draw=black, fill=white, minimum size=1, scale=0.4,] at (7,3) {};
    \node[draw, regular polygon,regular polygon sides=4,scale=0.8] at (7,3) {};
    \node[circle, draw=black, fill=black, minimum size=1, scale=0.4,] at (8,2) {};
    \node[circle, draw=black, fill=white, minimum size=1, scale=0.4,] at (9,1) {};
    \node[draw, regular polygon,regular polygon sides=4,scale=0.8] at (9,1) {};
    \node[circle, draw=black, fill=black, minimum size=1, scale=0.4,] at (10,0) {};
    \node[circle, draw=black, fill=white, minimum size=1, scale=0.4,] at (11,1) {};
    \node[circle, draw=black, fill=black, minimum size=1, scale=0.4,] at (12,0) {};
    \node[draw, regular polygon,regular polygon sides=4,scale=0.8] at (12,0) {};
   
    \node at (1,1+\x) {1};
    \node at (2,2+\x) {2};
    \node at (3,1+\x) {3};
    \node at (4,2+\x) {4};
    \node at (5,3+\x) {5};
    \node at (6,4+\x) {6};
    \node at (7,3+\x) {7};
    \node at (8,2+\x) {8};
    \node at (9,1+\x) {9};
    \node at (10,0+\x) {10};
    \node at (11,1+\x) {11};
    \node at (12,0+\x) {12};
    
    \draw[dotted] (0,0) -- (12,0);
    \draw[dotted] (1,1) -- (9,1);
    \draw[dotted] (4,2) -- (8,2);
    \draw[dotted] (5,3) -- (7,3);
    \node at (6.5,-1) {\textbf{NOT} Concordant};
    \end{tikzpicture}
    \caption{Example of concordant and not concordant subsets}
    \label{fig:concordant_subsets}
\end{figure}
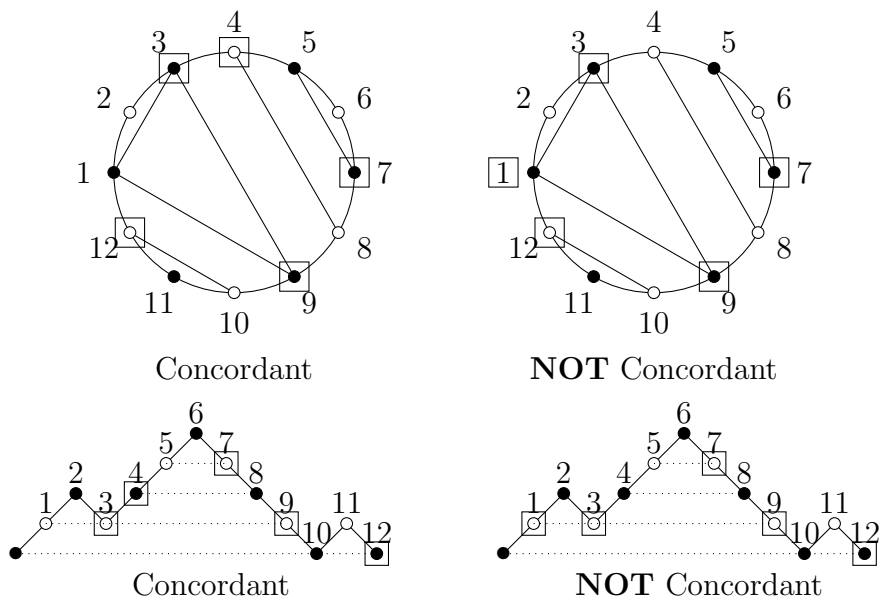
\end{ex}

For $I\in{2n\choose n-1}$, define \begin{equation}
    \label{eq:delta-to-L}
\Delta_I=\sum_{\sigma\in\cE(I)}L_{\sigma}\in G_{n,1}.\end{equation}

We say that a $(k,m)$-partial noncrossing matching $(\tau,T)$ is \emph{compatible} with $I,J\in{m\choose k}$ if
\begin{enumerate}
\item $\tau$ is a matching on $(I\setminus J)\cup(J\setminus I)$ such that each edges contains one vertex in $I\setminus J$ and the other vertex in $J\setminus I$;
\item $T=I\cap J$.
\end{enumerate}
\begin{theorem}\label{thm:delta-to-F}
For $I,J\in {2n\choose n-1}$,
\begin{equation}\label{eq:delta-to-F}
\Delta_I\Delta_J=\sum_{\tau,T}F_{\tau,T}
\end{equation}
where the sum is over all $(\tau,T)\in\mathcal{A}_{n-1,2n}$ compatible with $I,J$.
\end{theorem}
\begin{proof}
Both sides are in $G_{n,2}$. By \cref{prop:B-basis}, it suffices to compare the coefficient of $B_{\xi}$ on both sides, for $\xi\in\TNC_n$. We fix $I,J\in {2n\choose n-1}$ and $\xi\in\TNC_n$.

By \cref{thm:L-into-B}, \[\Delta_I\Delta_J=\sum_{\sigma\in \cE(I),\sigma'\in\cE(J)}L_{\sigma}L_{\sigma'}=\sum_{\xi\in\TNC_n}\sum_{\sigma\in \cE(I),\sigma'\in\cE(J)}a_{\xi,(\sigma,\sigma')}B_{\xi}.\]
The coefficient $a_{\xi,(I,J)}:=\sum_{\sigma\in \cE(I),\sigma'\in\cE(J)}a_{\xi,(\sigma,\sigma')}$, by \cref{lem:a-interpretation}, equals the number of ways of splitting $\Gamma:=\Gamma(\xi)$ into two groves $F_I\sqcup F_J$ such that the boundary partition $\sigma(F_I)$ is concordant with $I$ and $\sigma(F_J)$ is concordant with $J$. Note that a splitting of $\Gamma$ corresponds to a splitting $\tilde F_I\sqcup \tilde F_J$ of $\Gamma^{\vee}$ such that if $e\in F_I$, then its corresponding edge $e^{\vee}$ in $\Gamma^{\vee}$ belongs to $F_J$. The boundary partitions of $\tilde F_I$ and $\tilde F_J$ is exactly the dual noncrossing partition $\tilde\sigma(F_I)$ and $\tilde\sigma(F_J)$ respectively, by definition. We will think about $F_I\sqcup F_J$ and $\tilde F_I\sqcup \tilde F_J$ simultaneously.

For the right hand side of \cref{eq:delta-to-F}, summing over $(\tau,T)$ compatible with $I,J$,
\[\sum_{\tau, T}F_{\tau,T}=\sum_{\xi\in\TNC_n}\sum_{\tau,T}\beta(\xi)_{\tau,T}B_{\xi}=\sum_{\xi\in\TNC_n}\sum_{A}2^{\cyc(A)}B_{\xi}\]
where the final sum is over choices of $A$, with each choice $A$ consists of a selection of half-edges of $\Gamma$ and $\Gamma^{\vee}$ and a certain selection of vertices in each group of boundary vertices $Z_i$ identified together (see the beginning of this section), such that $\tau(A),T(A)$ is compatible with $I,J$. 

To establish equality bijectively, we map each partition of groves $F_I\sqcup F_J$ to a choice of $A$ described above, then we count the cardinality of the preimage of each $A$. This map $\psi$ works as follows. The connected components of $F_I\cup \tilde F_I$ give a partition $[2n]$ into $n+1$ parts (Lemma 2.2 of \cite{lam2004electroid}) and since $\sigma(F)$ is concordant with $I$, there is exactly one vertex for each part that does not lie in $I$, called the \emph{root}. We orient each tree in $F_I\cup \tilde F_I$ towards the root (called a \emph{rooting}) and for each edge $e\in F_I\cup \tilde F_I$, choose the half-edge in $N(\xi)$ containing the source. Now, the selection $A$ consists of these half-edges $H$, together with $T(A)=I\cap J$ and $\tau(A)$ pairs up $I\setminus J$ and $J\setminus I$ based on the paths formed by these half-edges. See \cref{fig:proof-Delta-into-F}.
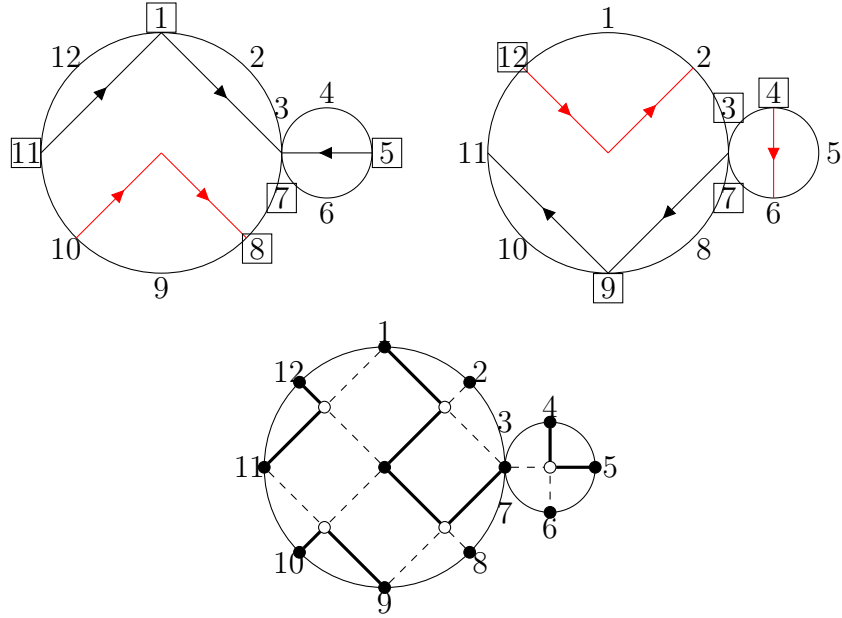
\begin{figure}[h!]
\centering
\begin{tikzpicture}[scale=0.4]
\def\r{4};
\def\rr{1.5};
\draw (0,0) circle (\r);
\draw (\r+\rr,0) circle (\rr);
\node at (90:\r+0.5) {$1$}; 
\node at (135:\r+0.5) {$12$}; 
\node at (180:\r+0.5) {$11$}; 
\node at (225:\r+0.5) {$10$}; 
\node at (270:\r+0.5) {$9$}; 
\node at (315:\r+0.5) {$8$}; 
\node at (45:\r+0.5) {$2$}; 
\node at (\r,1.5) {$3$};
\node at (\r,-1.5) {$7$};
\node at (\r+\rr,\rr+0.5) {$4$};
\node at (\r+\rr,-\rr-0.5) {$6$};
\node at (\r+2*\rr+0.5,0) {$5$};
\draw(\r+2*\rr,0)--(\r,0)node[currarrow,sloped,allow upside down,pos=0.5] {};
\draw(-\r,0)--(0,\r)node[currarrow,sloped,allow upside down,pos=0.5] {};
\draw(0,\r)--(\r,0)node[currarrow,sloped,allow upside down,pos=0.5] {};
\draw[red](225:\r)--(0,0)node[color=red,currarrow,sloped,allow upside down,pos=0.5] {};
\draw[red](0,0)--(315:\r)node[color=red,currarrow,sloped,allow upside down,pos=0.5] {};
\node[draw, regular polygon,regular polygon sides=4] at (90:\r+0.5) {};
\node[draw, regular polygon,regular polygon sides=4] at (\r+2*\rr+0.5,0) {};
\node[draw, regular polygon,regular polygon sides=4] at (\r,-1.5) {};
\node[draw, regular polygon,regular polygon sides=4] at (315:\r+0.5) {};
\node[draw, regular polygon,regular polygon sides=4] at (180:\r+0.5) {};
\end{tikzpicture}
\quad
\begin{tikzpicture}[scale=0.4]
\def\r{4};
\def\rr{1.5};
\draw (0,0) circle (\r);
\draw (\r+\rr,0) circle (\rr);
\node at (90:\r+0.5) {$1$}; 
\node at (135:\r+0.5) {$12$}; 
\node at (180:\r+0.5) {$11$}; 
\node at (225:\r+0.5) {$10$}; 
\node at (270:\r+0.5) {$9$}; 
\node at (315:\r+0.5) {$8$}; 
\node at (45:\r+0.5) {$2$}; 
\node at (\r,1.5) {$3$};
\node at (\r,-1.5) {$7$};
\node at (\r+\rr,\rr+0.5) {$4$};
\node at (\r+\rr,-\rr-0.5) {$6$};
\node at (\r+2*\rr+0.5,0) {$5$};
\draw(0,-\r)--(-\r,0)node[currarrow,sloped,allow upside down,pos=0.5] {};
\draw(\r,0)--(0,-\r)node[currarrow,sloped,allow upside down,pos=0.5] {};
\draw[red](135:\r)--(0,0)node[color=red,currarrow,sloped,allow upside down,pos=0.5] {};
\draw[red](0,0)--(45:\r)node[color=red,currarrow,sloped,allow upside down,pos=0.5] {};
\draw[red](\r+\rr,\rr)--(\r+\rr,-\rr)node[color=red,currarrow,sloped,allow upside down,pos=0.5] {};
\node[draw, regular polygon,regular polygon sides=4] at (\r,1.5) {};
\node[draw, regular polygon,regular polygon sides=4] at (\r,-1.5) {};
\node[draw, regular polygon,regular polygon sides=4] at (\r+\rr,\rr+0.5) {};
\node[draw, regular polygon,regular polygon sides=4] at (270:\r+0.5) {};
\node[draw, regular polygon,regular polygon sides=4] at (135:\r+0.5) {};
\end{tikzpicture}

\begin{tikzpicture}[scale=0.4]
\def\r{4};
\def\rr{1.5};
\draw (0,0) circle (\r);
\draw (\r+\rr,0) circle (\rr);
\node at (90:\r+0.5) {$1$}; 
\node at (135:\r+0.5) {$12$}; 
\node at (180:\r+0.5) {$11$}; 
\node at (225:\r+0.5) {$10$}; 
\node at (270:\r+0.5) {$9$}; 
\node at (315:\r+0.5) {$8$}; 
\node at (45:\r+0.5) {$2$}; 
\node at (\r,1.5) {$3$};
\node at (\r,-1.5) {$7$};
\node at (\r+\rr,\rr+0.5) {$4$};
\node at (\r+\rr,-\rr-0.5) {$6$};
\node at (\r+2*\rr+0.5,0) {$5$};
\node[circle, draw=black, fill=black, minimum size=1, scale=0.4,] (1) at (90:\r) {};
\node[circle, draw=black, fill=black, minimum size=1, scale=0.4,] (12) at (135:\r) {};
\node[circle, draw=black, fill=black, minimum size=1, scale=0.4,] (11) at (180:\r) {};
\node[circle, draw=black, fill=black, minimum size=1, scale=0.4,] (10) at (225:\r) {};
\node[circle, draw=black, fill=black, minimum size=1, scale=0.4,] (9) at (270:\r) {};
\node[circle, draw=black, fill=black, minimum size=1, scale=0.4,] (8) at (315:\r) {};
\node[circle, draw=black, fill=black, minimum size=1, scale=0.4,] (37) at (0:\r) {};
\node[circle, draw=black, fill=black, minimum size=1, scale=0.4,] (2) at (45:\r) {};
\node[circle, draw=black, fill=black, minimum size=1, scale=0.4,] (4) at (\r+\rr,\rr) {};
\node[circle, draw=black, fill=black, minimum size=1, scale=0.4,] (5) at (\r+2*\rr,0) {};
\node[circle, draw=black, fill=black, minimum size=1, scale=0.4,] (6) at (\r+\rr,-\rr) {};
\node[circle, draw=black, fill=black, minimum size=1, scale=0.4,] (0) at (0,0) {};
\draw[dashed](1)--(11)--(9)--(37)--(1);
\draw[dashed](2)--(0)--(10);
\draw[dashed](8)--(0)--(12);
\draw[dashed](4)--(6);
\draw[dashed](37)--(5);

\node[circle, draw=black, fill=white, minimum size=1, scale=0.4,] (i1) at (135:0.70710678118*\r) {};
\node[circle, draw=black, fill=white, minimum size=1, scale=0.4,] (i2) at (45:0.70710678118*\r) {};
\node[circle, draw=black, fill=white, minimum size=1, scale=0.4,] (i3) at (315:0.70710678118*\r) {};
\node[circle, draw=black, fill=white, minimum size=1, scale=0.4,] (i4) at (225:0.70710678118*\r) {};
\node[circle, draw=black, fill=white, minimum size=1, scale=0.4,] (i5) at (\r+\rr,0) {};

\draw [very thick] (11) -- (i1);
\draw [very thick] (1) -- (i2);
\draw [very thick] (5) -- (i5);
\draw [very thick] (10) -- (i4);
\draw [very thick] (0) -- (i3);
\draw [very thick] (12) -- (i1);
\draw [very thick] (0) -- (i2);
\draw [very thick] (37) -- (i3);
\draw [very thick] (9) -- (i4);
\draw [very thick] (4) -- (i5);
\end{tikzpicture}
\caption{The map $\psi$ in the proof of \cref{thm:delta-to-F} with $\xi=(1,10|2,9|3,12|4,6|5,7|8,11)$, $I=\{1,5,7,8,11\}$, $J=\{3,4,7,9,12\}.$ Top left: $F_I\cup \tilde F_I$; top right: $F_J\cup\tilde F_J$; bottom: the selection $A$ in $N(\xi)$ with edges in $\Gamma$ colored in black and edges in $\Gamma^{\vee}$ colored in red.}
\label{fig:proof-Delta-into-F}
\end{figure}

We first justify that this map $\psi$ is well-defined, i.e. the selection $A$ obtained satisfies that every internal vertex has degree $2$ in $H$, every boundary vertex has degree at most $2$ in $H$, and that the pair $(\tau(A),T(A))$ is compatible with $I,J$. Recall that $H\subset N(\xi)$ is the set of half-edges that we selected. To begin with, each white interior vertex has degree $2$. Each black interior vertex is not a a root in its connected component in $F_I\cup \tilde F_I$ so it has outdegree $1$ in the rooting, gaining one degree in $H$; similarly, it gains one degree in $H$ from $F_J\cup\tilde F_J$ so it has degree $2$ in total. Likewise, each boundary black vertex is either a root in $F_I\cup\tilde F_I$, in which case it gains no degree in $H$, or not a root in $F_I\cup\tilde F_I$, in which case it gains one degree in $H$. So considering $F_J\cup\tilde F_J$ as well, each boundary vertex has degree $\leq2$.

The selection $A$ also consists of the data of $\tau(A)$ and $T(A)$ and we need to show the following: First, from $\psi$, $\tau(A)$ can be chosen so that $I\setminus J$ is paired with $J\setminus I$ given by the paths formed by the half-edges described above; and second, $T(A)=I\cap J$ consists of $|Z_k|-1$ vertices from $Z_k$ if $Z_k$ is of degree $0$ or $1$ in $H$ and consists of all vertices from $Z_k$ if $Z_k$ is of degree $2$ in $H$, where $Z_1,\ldots,Z_m$ is the set of boundary vertices in the cactus. The claim concerning $T(A)$ is clear by counting: $v\in[2n]$ contributes degree $x$ to $Z_k\ni v$ if $v$ appears $x$ times in $I\cup J$ as a multiset. As for $\tau(A)$, it suffices to show that we cannot form a path with the two endpoints both in $I\setminus J$. This is done by a parity argument. Notice that as we traverse half-edges consecutively in a path or a cycle of $H$, we alternate between black vertices and white vertices, and alternate between edges in $F_I\cup\tilde F_I$ and edges in $F_J\cup\tilde F_J$. Thus, if a path in $H$ starts with a vertex $v$ and an edge in $F_I\cup\tilde F_I$, ends with a vertex $w$ and an edge in $F_J\cup\tilde F_J$, then we know that $v\notin I$ and $w\notin J$, meaning that $v\in J$ and $w\in I$ as they need to have degree $1$. This shows that $(\tau(A),T(A))$ is compatible with $I,J$.

Now fix a selection $A$ such that $(\tau(A),T(A))$ is compatible with $I,J$. We count the number of partitions of $\Gamma$ into groves $F_I\sqcup F_J$ such that $\sigma(F_I),\sigma(F_J)$ are concordant with $I,J$ respectively. For each path $P$ in $A$, we can without loss of generality assume that it starts with $v\in I\setminus J$ and ends with $w\in J\setminus I$, as $\tau(A)$ is a matching that pairs $I\setminus J$ with $J\setminus I$. As we traverse the half-edges in $P$, the first edge incident to $v$ needs to be assigned to $F_J\cup \tilde F_J$ for a rooting to exist. Then the half-edges must alternate between $F_I\cup\tilde F_I$ and $F_J\cup\tilde F_J$ since for a rooting to exist, each internal black vertex has outdegree $1$ in both $F_I\cup\tilde F_I$ and $F_J\cup\tilde F_J$. This alternating property allows us to assign all edges along this path $P$. Likewise, every cycle $C$ in $A$ has two possible assignments to $F_I\cup\tilde F_I$ and $F_J\cup\tilde F_J$ via the alternating property (see \cref{fig:proof-delta-into-F-cycle}). Thus, every edge in $\Gamma$ and $\Gamma^{\vee}$ is now assigned to either $F_I\cup\tilde F_I$ or $F_J\cup\tilde F_J$. Notice that every directed tree whose internal vertices have outdegree $1$ has a leaf that is the unique source, i.e. a root. So the concordant property with $I$ and $J$ is exactly the same as the existence of a rooting. All $2^{\cyc(A)}$ choices satisfy the requirement. This finishes the comparison of coefficients of $B_{\xi}$ on both sides of \cref{eq:delta-to-F} as desired.
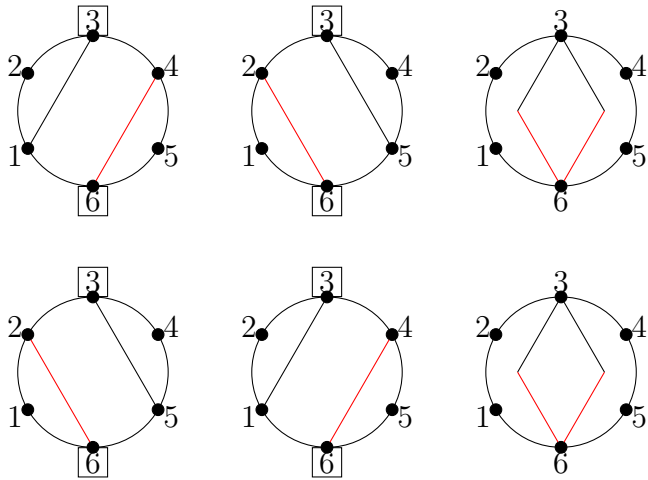
\begin{figure}[h!]
\centering
\begin{tikzpicture}[scale=0.4]
\def\r{2.5};
\draw (0,0) circle (\r);
\node[circle, draw=black, fill=black, minimum size=1, scale=0.4,] (4) at (30:\r) {};
\node[circle, draw=black, fill=black, minimum size=1, scale=0.4,] (3) at (90:\r) {};
\node[circle, draw=black, fill=black, minimum size=1, scale=0.4,] (2) at (150:\r) {};
\node[circle, draw=black, fill=black, minimum size=1, scale=0.4,] (1) at (210:\r) {};
\node[circle, draw=black, fill=black, minimum size=1, scale=0.4,] (6) at (270:\r) {};
\node[circle, draw=black, fill=black, minimum size=1, scale=0.4,] (5) at (330:\r) {};
\node at (30:\r+0.5) {$4$};
\node at (90:\r+0.5) {$3$};
\node at (150:\r+0.5) {$2$};
\node at (210:\r+0.5) {$1$};
\node at (270:\r+0.5) {$6$};
\node at (330:\r+0.5) {$5$};
\node[draw, regular polygon,regular polygon sides=4] at (0,\r+0.5) {};
\node[draw, regular polygon,regular polygon sides=4] at (0,-\r-0.5) {};
\draw(1)--(3);
\draw[red](4)--(6);
\end{tikzpicture}
\quad
\begin{tikzpicture}[scale=0.4]
\def\r{2.5};
\draw (0,0) circle (\r);
\node[circle, draw=black, fill=black, minimum size=1, scale=0.4,] (4) at (30:\r) {};
\node[circle, draw=black, fill=black, minimum size=1, scale=0.4,] (3) at (90:\r) {};
\node[circle, draw=black, fill=black, minimum size=1, scale=0.4,] (2) at (150:\r) {};
\node[circle, draw=black, fill=black, minimum size=1, scale=0.4,] (1) at (210:\r) {};
\node[circle, draw=black, fill=black, minimum size=1, scale=0.4,] (6) at (270:\r) {};
\node[circle, draw=black, fill=black, minimum size=1, scale=0.4,] (5) at (330:\r) {};
\node at (30:\r+0.5) {$4$};
\node at (90:\r+0.5) {$3$};
\node at (150:\r+0.5) {$2$};
\node at (210:\r+0.5) {$1$};
\node at (270:\r+0.5) {$6$};
\node at (330:\r+0.5) {$5$};
\node[draw, regular polygon,regular polygon sides=4] at (0,\r+0.5) {};
\node[draw, regular polygon,regular polygon sides=4] at (0,-\r-0.5) {};
\draw(5)--(3);
\draw[red](2)--(6);
\end{tikzpicture}
\quad
\begin{tikzpicture}[scale=0.4]
\def\r{2.5};
\draw (0,0) circle (\r);
\node[circle, draw=black, fill=black, minimum size=1, scale=0.4,] (4) at (30:\r) {};
\node[circle, draw=black, fill=black, minimum size=1, scale=0.4,] (3) at (90:\r) {};
\node[circle, draw=black, fill=black, minimum size=1, scale=0.4,] (2) at (150:\r) {};
\node[circle, draw=black, fill=black, minimum size=1, scale=0.4,] (1) at (210:\r) {};
\node[circle, draw=black, fill=black, minimum size=1, scale=0.4,] (6) at (270:\r) {};
\node[circle, draw=black, fill=black, minimum size=1, scale=0.4,] (5) at (330:\r) {};
\node at (30:\r+0.5) {$4$};
\node at (90:\r+0.5) {$3$};
\node at (150:\r+0.5) {$2$};
\node at (210:\r+0.5) {$1$};
\node at (270:\r+0.5) {$6$};
\node at (330:\r+0.5) {$5$};
\coordinate (i2) at (\r*0.57735026919,0);
\coordinate (i1) at (-\r*0.57735026919,0);
\draw(3)--(i1);
\draw[red](i1)--(6);
\draw[red](i2)--(6);
\draw(i2)--(3);
\end{tikzpicture}

\

\begin{tikzpicture}[scale=0.4]
\def\r{2.5};
\draw (0,0) circle (\r);
\node[circle, draw=black, fill=black, minimum size=1, scale=0.4,] (4) at (30:\r) {};
\node[circle, draw=black, fill=black, minimum size=1, scale=0.4,] (3) at (90:\r) {};
\node[circle, draw=black, fill=black, minimum size=1, scale=0.4,] (2) at (150:\r) {};
\node[circle, draw=black, fill=black, minimum size=1, scale=0.4,] (1) at (210:\r) {};
\node[circle, draw=black, fill=black, minimum size=1, scale=0.4,] (6) at (270:\r) {};
\node[circle, draw=black, fill=black, minimum size=1, scale=0.4,] (5) at (330:\r) {};
\node at (30:\r+0.5) {$4$};
\node at (90:\r+0.5) {$3$};
\node at (150:\r+0.5) {$2$};
\node at (210:\r+0.5) {$1$};
\node at (270:\r+0.5) {$6$};
\node at (330:\r+0.5) {$5$};
\node[draw, regular polygon,regular polygon sides=4] at (0,\r+0.5) {};
\node[draw, regular polygon,regular polygon sides=4] at (0,-\r-0.5) {};
\draw(5)--(3);
\draw[red](2)--(6);
\end{tikzpicture}
\quad
\begin{tikzpicture}[scale=0.4]
\def\r{2.5};
\draw (0,0) circle (\r);
\node[circle, draw=black, fill=black, minimum size=1, scale=0.4,] (4) at (30:\r) {};
\node[circle, draw=black, fill=black, minimum size=1, scale=0.4,] (3) at (90:\r) {};
\node[circle, draw=black, fill=black, minimum size=1, scale=0.4,] (2) at (150:\r) {};
\node[circle, draw=black, fill=black, minimum size=1, scale=0.4,] (1) at (210:\r) {};
\node[circle, draw=black, fill=black, minimum size=1, scale=0.4,] (6) at (270:\r) {};
\node[circle, draw=black, fill=black, minimum size=1, scale=0.4,] (5) at (330:\r) {};
\node at (30:\r+0.5) {$4$};
\node at (90:\r+0.5) {$3$};
\node at (150:\r+0.5) {$2$};
\node at (210:\r+0.5) {$1$};
\node at (270:\r+0.5) {$6$};
\node at (330:\r+0.5) {$5$};
\node[draw, regular polygon,regular polygon sides=4] at (0,\r+0.5) {};
\node[draw, regular polygon,regular polygon sides=4] at (0,-\r-0.5) {};
\draw(1)--(3);
\draw[red](4)--(6);
\end{tikzpicture}
\quad
\begin{tikzpicture}[scale=0.4]
\def\r{2.5};
\draw (0,0) circle (\r);
\node[circle, draw=black, fill=black, minimum size=1, scale=0.4,] (4) at (30:\r) {};
\node[circle, draw=black, fill=black, minimum size=1, scale=0.4,] (3) at (90:\r) {};
\node[circle, draw=black, fill=black, minimum size=1, scale=0.4,] (2) at (150:\r) {};
\node[circle, draw=black, fill=black, minimum size=1, scale=0.4,] (1) at (210:\r) {};
\node[circle, draw=black, fill=black, minimum size=1, scale=0.4,] (6) at (270:\r) {};
\node[circle, draw=black, fill=black, minimum size=1, scale=0.4,] (5) at (330:\r) {};
\node at (30:\r+0.5) {$4$};
\node at (90:\r+0.5) {$3$};
\node at (150:\r+0.5) {$2$};
\node at (210:\r+0.5) {$1$};
\node at (270:\r+0.5) {$6$};
\node at (330:\r+0.5) {$5$};
\coordinate (i2) at (\r*0.57735026919,0);
\coordinate (i1) at (-\r*0.57735026919,0);
\draw(3)--(i1);
\draw[red](i1)--(6);
\draw[red](i2)--(6);
\draw(i2)--(3);
\end{tikzpicture}
\caption{An example of two partitions achieving the same cycle in $N(\xi)$ with $\xi=(13|25|46)$, $I,J=\{3,6\}$. Left: $F_I\cup \tilde F_I$; middle: $F_J\cup\tilde F_J$; right: the same selection $A$ of $N(\xi)$}
\label{fig:proof-delta-into-F-cycle}
\end{figure}
\end{proof}

We have now established the following commutative diagram of various expansions:
\begin{equation}
\begin{tikzcd}[column sep=huge]
\Delta_I\Delta_J \arrow[r,"\text{Theorem~}\ref{thm:delta-to-F}"] \arrow[d,"\text{concordance}"] & F_{\tau,T} \arrow[d,"\text{Definition~}\ref{def:F_tau_T}"]  \\
L_{\sigma}L_{\sigma'} \arrow[r,"\text{Theorem~}\ref{thm:L-into-B}"] & B_{\xi}
\end{tikzcd}.
\end{equation}
The existence of $F_{\tau,T}$'s satisfying \cref{thm:delta-to-F} is equivalent to the Pl\"ucker relations of the Pl\"ucker ring (see Theorem 3.10 and Theorem 3.1 of \cite{lam2015dimers}).  Using \cref{eq:delta-to-L}, we deduce the ``electrical" Pl\"ucker relations.
\begin{prop}[Proposition 5.35 of \cite{lam2004electroid}]\label{prop:tfy_relations}
For each $1\leq k < n-1$ and each $I = \{i_1 < i_2 < \dots < i_{n-1}\}, J = \{j_1 < j_2 < \dots  < j_{n-1}\}$, we have
\[\sum_{\sigma\in \cE(I),\kappa\in \cE(J)}L_{\sigma}L_{\kappa} = \sum_{I', J'}(-1)^{a(I', J')}\sum_{\sigma\in \cE(I'), \kappa\in \cE(J')}L_\sigma  L_\kappa\]
where:
\begin{enumerate}
\item the summation is over $I', J'$ obtained from swapping the last $k$ indices in $J$ with any $k$ indices in $I$, keeping the order in both;
\item $a(I', J')$ is the total number of swaps needed to put both subsets $I'$ and $J'$ in order.
\end{enumerate}
\end{prop}

\section{Tableaux basis of the Grove algebra}\label{sec:tableaux}
In this section, we provide a basis of $G_n$, building on results in \cite{bychkov2021electrical,chepuri2021electrical}. We will select a set of relations $R$ from \cref{prop:tfy_relations} and show that $R$ is a Gr\"obner basis of the ideal $\II_n$, which is defined to be the ideal of polynomials in the variables $\{L_{\sigma}\}_{\sigma\in\NP_n}$ that vanish on all electrical networks $\Gamma$.

\def\C{{\mathbb C}}
We work over $\C$ in this section, and $G_n$ denotes $G_n \otimes \C$.

\subsection{Embedding into the Grassmannian}\label{sec:geometry}

Let $R(n-1,2n)$ denote the Pl\"ucker ring, the homogeneous coordinate ring of the Grassmannian $\Gr(n-1,2n)$.
\begin{theorem}\label{thm:embed}
The formula
\[\Delta_I=\sum_{\sigma\in\cE(I)}L_{\sigma}\in G_{n,1} \tag{\cref{eq:delta-to-L}}\]
induces a graded ring homomorphism $\iota^*: R(n-1,2n) \to G_n$.
\end{theorem}
\begin{proof}
\def\tpi{{\tilde \pi}}
Let $\pi: \C[\Delta_I \mid I \in \binom{[2n]}{n-1}] \to G_n$ be the ring homomorphism given by the formula \cref{eq:delta-to-L}.  The ring $R(n-1,2n)$ is the quotient of $\C[\Delta_I]$ by the Pl\"ucker ideal generated by the Pl\"ucker relations.  We need to check that the Pl\"ucker ideal belongs to the kernel of $\pi$.

It is shown in Theorem 3.10 of \cite{lam2015dimers} that the existence of elements $F_{\tau,T}$ satisfying our \cref{thm:delta-to-F} implies the Pl\"ucker relations, at least set-theoretically.  Since the ring $G_n$ has no nilpotents, it follows that the entire Pl\"ucker ideal belongs to the kernel of $\pi$.
\end{proof}

\def\Proj{{\rm Proj}}
The ring homomorphism $\iota^*: R(n-1,2n) \to G_n$ induces a morphism between projective varieties $$\Proj(G_n) \to \Proj(R(n-1,2n)).$$  Indeed, this is exactly the embedding $\iota$ of electrical networks into the Grassmannian $\Gr(n-1,2n)$ constructed in \cite{lam2004electroid}, and our results give a new proof of this embedding.  Note that in \cite{lam2004electroid}, the formula \cref{eq:delta-to-L} is a theorem, whereas for us it is a definition.

The embedding is constructed in \cite{lam2004electroid} as follows.  Recall that $C_n$ denotes the $n$-th Catalan number.  Let $M = (M_{I\sigma})$ be the $\binom{2n}{n-1} \times C_n$ matrix of $1$-s and $0$-s, with $1$-s in the entries corresponding to pairs $(I,\sigma) \in \binom{[2n]}{n-1} \times \NP_n$ that are concordant.  Let $\X_n = H \cap \Gr(n-1,2n)$ be the linear slice of the Grassmannian, where $H \subset {\mathbb P}^{\binom{2n}{n-1}-1}$ is the projective subspace given by the image of the matrix $M$.  It is shown in \cite{lam2004electroid} that $\X_{n,\geq 0}:= \X_n \cap \Gr(n-1,2n)_{\geq 0}$ is isomorphic to the compactified space of electrical networks.  Since $\X_{n,\geq 0}$ is Zariski-dense in $\X$, we have the following result.

\begin{prop}\label{prop:grove}
The homogeneous coordinate ring of $\X$ is naturally isomorphic to the grove algebra $G_n$.  The ring homorphism $\pi:R(n-1,2n) \to G_n$ in \cref{thm:embed} induces the embedding $\iota: \X_n \hookrightarrow \Gr(n-1,2n)$ of \cite{lam2004electroid}.
\end{prop}

\subsection{Electrical networks and the Lagrangian Grassmannian}\label{sec:Lagrangian}
The subvariety $\X_n$ was further studied in the works of Chepuri, George, and Speyer \cite{chepuri2021electrical} and by Bychkov, Gorbounov, Kazakov, and Talalaev \cite{bychkov2021electrical}, who established the following result. 

\begin{theorem}\label{thm:LG}
The subvariety $\X_n$ is isomorphic to the Lagrangian Grassmannian $\LG(n-1,2n-2)$.
\end{theorem}

The appearance of the symplectic group in the study of electrical networks earlier appeared in \cite{Lam-Pylyavskyy}.  Combining \cref{thm:LG} and \cref{prop:grove}, we have the following result.

\begin{prop}\label{prop:dim_G_nd}
The dimension the graded piece $G_{n,d}$ of $G_{n}$ is $\dim G_{n,d}=\#\mathcal{C}_n^{(k)}$.
\end{prop}
\begin{proof}
By Borel-Weil theory, the embedding $\X_n$ of $\LG(n-1,2n-2)$ into the Grassmannian $\Gr(n-1,2n)$, and thus the projective space ${\mathbb P}^{\binom{2n}{n-1}-1}$ corresponds to a choice of line bundle on $\LG(n-1,2n-2)$, or equivalently, to the choice of an irreducible representation $V_\omega$ of the symplectic group $\Sp(2n-2)$.  Indeed, it is shown in \cite{bychkov2021electrical} (see also \cite{chepuri2021electrical}) that $\omega = \omega_n$, where $n$ indexes the long simple root in the root system of type $C_n$.  Indeed, $V_\omega$ can be identified with the kernel of the natural map $\varphi: \bigwedge^{n-1} \C^{2n-2} \to \bigwedge^{n-3} \C^{2n-2}$; see Theorem 17.5 of \cite{Fulton-Harris}.

Each graded piece $G_{n,d}$ is itself an irreducible representation of $\Sp(2n-2)$.  Namely, we have $G_{n,d} \cong V_{d\omega}$.  The dimension of $V_{d\omega}$ is given by the Weyl character formula.  Explicitly, taking $\lambda = (d^{n-1})$ in Lemma 4.1 of \cite{Campbell-Stokke}, we get that this dimension is equal to
\begin{equation}\label{eq:WCF}
\frac{ \prod_{i=1}^{n-1} a_i \prod_{1\leq i<j \leq n-1} (a_i-a_j)(a_i+a_j)}{\prod_{i=1}^{n-1} (2i-1)!}
\end{equation}
where $(a_1,a_2,\ldots,a_{n-1}) = (d+n-1,d+n-2,\ldots,d+1)$.
We manipulate \cref{eq:WCF}:
\begin{align*}
&=\prod_{1 \leq i \leq n-1}(i+d) \prod_{1 \leq i < j \leq n-1}(i+j+2d) \left(\frac{\prod_{1 \leq i < j \leq n-1}(j-i)}{\prod_{i=1}^{n-1}(2i-1)!} \right) \\
&=\prod_{1 \leq i \leq n-1}(i+d) \prod_{1 \leq i < j \leq n-1}(i+j+2d) \frac{2^{n-1}}{\prod_{1\leq i \leq n-1}(i+i)(i+i+1) \cdots (i+n-1)}\\
&= \prod_{1\leq i \leq j \leq n-1} \frac{i+j+2d}{i+j}.
\end{align*}
The result follows from \cref{prop:prodformula}.

\end{proof}

The following result follows from standard techniques in the enumeration of noncrossing lattice paths; see \cite[Section 3.1.6]{ardila}.
\begin{prop}\label{prop:prodformula}
We have 
$$\#\mathcal{C}_n^{(k)} = \prod_{1 \leq i \leq j \leq n-1} \frac{i+j+2k}{i+j}.
$$
\end{prop}

Recall that in \cref{sec:catalan}, we defined a partial order $\leq$ and a total order $\preccurlyeq$ on Catalan objects. In this section, we use Dyck paths, and for notations, we write $L_P$ where $P\in\mathcal{C}_n$, instead of $L_{\sigma}$ where $\sigma\in\NP_n$, with the bijection between Dyck paths and noncrossing partitions shown in \cref{sec:catalan}. We label the nodes of each Dyck path by $0,1,\ldots,2n$ and typically ignore the $0$.

We define a total ordering on monomials in $G_n$ as follows. For monomials $L_{\mathbf{P}}=L_{P_1}\cdots L_{P_d}$ with $P_1\preccurlyeq\cdots\preccurlyeq P_d$ and $L_{\mathbf{Q}}=L_{Q_1}\cdots L_{Q_d}$ with $Q_1\preccurlyeq\cdots\preccurlyeq Q_d$, we define $L_{\mathbf{P}}\prec L_{\mathbf{Q}}$ if for some $1\leq k\leq d$, $P_i=Q_i$ for all $k+1\leq i\leq d$ and $P_k\prec Q_k$. This is a lexicographic order where we compare the larger sides first. Also, a monomial of lower degree always precedes a monomial of higher degree. For a polynomial $f$ in $L_{\sigma}$'s, its \emph{leading term} is the monomial with nonzero coefficient that is smallest in the total order $\preccurlyeq$.

We say that a monomial $L_{\mathbf{P}}=L_{P_1}\cdots L_{P_d}$ is \emph{standard} if $P_1\leq P_2\leq\cdots\leq P_d$, where we use the partial order here. Thus, standard monomials of degree $d$ are indexed by $\mathcal{C}_n^{(d)}$. We will soon see that standard monomials form a basis of $G_n$.

\begin{defin}\label{def:catalan-subset}
For a Dyck path $P$, its corresponding \emph{Catalan subset} $I(P)$ is an $(n-1)$-element subset of $[2n]$ such that $\{1\}\cup\{a+1\:|\: a\in I(P)\}$ is the positions of the upsteps in $P$. A subset $I\subset{2n\choose n-1}$ is called a \emph{Catalan subset} if $I$ equals the Catalan subset of a Dyck path $P$. 
\end{defin}
For a Catalan subset $I$, let $P(I)$ be the corresponding Dyck path. It is also straightforward to see that $P\leq Q$ if and only if the $k$-th largest element in $I(P)$ is greater than or equal to the $k$-th largest element in $I(Q)$, for all $k=1,\ldots,n-1$. The following lemma is immediate from the definition of Dyck paths.
\begin{lemma}\label{lem:catalan-subset}
The subset $I=\{a_1<\cdots<a_{n-1}\}$ is a Catalan subset if and only if $a_i\leq 2i$ for all $1\leq i\leq n-1$.
\end{lemma}

\begin{remark}
We see that $P\preccurlyeq Q$ if and only if $I(P)$ is lexicographically larger than $I(Q)$. So for the remainder of the section, we will be extra careful regarding whether a path is smaller than another path, or a subset is smaller than another subset.
\end{remark}

Observe that $I(P)$ is concordant with the noncrossing partition $\sigma(P)$. See \cref{fig:upstep_concordant} for an example of $I=\{1,3,4,5,10\}$. 
\begin{figure}[h!]
    \centering
     \begin{tikzpicture}[scale=0.7]
    \def\x{0.7}
     \draw(0,0)--(2,2)--(3,1)--(6,4)--(10,0)--(11,1)--(12,0);
    \node at (0,0) {$\bullet$};
    \node[circle, draw=black, fill=white, minimum size=1, scale=0.4,] at (1,1) {};
    \node[draw, regular polygon,regular polygon sides=4,scale=0.8] at (1,1) {};
    \node[circle, draw=black, fill=black, minimum size=1, scale=0.4,] at (2,2) {};
    \node[circle, draw=black, fill=white, minimum size=1, scale=0.4,] at (3,1) {};
    \node[draw, regular polygon,regular polygon sides=4,scale=0.8] at (3,1) {};
    \node[circle, draw=black, fill=black, minimum size=1, scale=0.4,] at (4,2) {};
    \node[draw, regular polygon,regular polygon sides=4,scale=0.8] at (4,2) {};
    \node[circle, draw=black, fill=white, minimum size=1, scale=0.4,] at (5,3) {};
    \node[draw, regular polygon,regular polygon sides=4,scale=0.8] at (5,3) {};
    \node[circle, draw=black, fill=black, minimum size=1, scale=0.4,] at (6,4) {};
    \node[circle, draw=black, fill=white, minimum size=1, scale=0.4,] at (7,3) {};
    \node[circle, draw=black, fill=black, minimum size=1, scale=0.4,] at (8,2) {};
    \node[circle, draw=black, fill=white, minimum size=1, scale=0.4,] at (9,1) {};
    \node[circle, draw=black, fill=black, minimum size=1, scale=0.4,] at (10,0) {};
    \node[draw, regular polygon,regular polygon sides=4,scale=0.8] at (10,0) {};
    \node[circle, draw=black, fill=white, minimum size=1, scale=0.4,] at (11,1) {};
    \node[circle, draw=black, fill=black, minimum size=1, scale=0.4,] at (12,0) {};
   
    \node at (1,1+\x) {1};
    \node at (2,2+\x) {2};
    \node at (3,1+\x) {3};
    \node at (4,2+\x) {4};
    \node at (5,3+\x) {5};
    \node at (6,4+\x) {6};
    \node at (7,3+\x) {7};
    \node at (8,2+\x) {8};
    \node at (9,1+\x) {9};
    \node at (10,0+\x) {10};
    \node at (11,1+\x) {11};
    \node at (12,0+\x) {12};
    
    \draw[dotted] (0,0) -- (12,0);
    \draw[dotted] (1,1) -- (9,1);
    \draw[dotted] (4,2) -- (8,2);
    \draw[dotted] (5,3) -- (7,3);
    \end{tikzpicture}
    \caption{Concordant subset obtained from taking the upsteps.}
    \label{fig:upstep_concordant}
\end{figure}
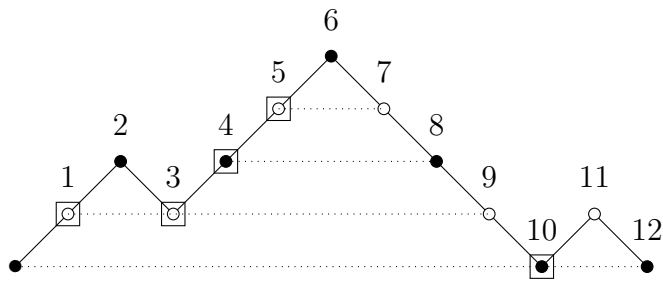

\begin{lemma}\label{lem:P-minimal-in-delta}
For a Catalan subset $I$, $P(I)$ is the minimum in the total order $\preccurlyeq$ on $\cE(I)$. 
\end{lemma}
\begin{proof}
Fix a Catalan subset $I$ and its Dyck path $P=P(I)$, and consider another Dyck path $Q\preccurlyeq P$ which differs from $P$ at step $a+1$ where $Q$ goes down and $P$ goes up and agrees with $P$. This means $a\in I$ by construction. Consider the node right before step $a+1$, which is labeled by $a$ in $Q$ and look horizontally to the left. Since $Q$ goes down at step $a+1$, we have that $a$ is the maximum element in this connected component $Z$. Since $Q$ agrees with $P$ before step $a$, every node in $Z$ is followed by an upstep, meaning that $(Z\setminus\{a\})\subset I$. As $a\in I$, $Z\subset I$, and $I$ cannot be concordant with $Q$.  
\end{proof}

\begin{defin}\label{def:noncomparable_relation}
For a pair of noncomparable (under $\leq$) Dyck path $P\succcurlyeq Q$, let $I=I(P)=\{a_1<\cdots<a_{n-1}\}$ and $J=I(Q)=\{b_1<\cdots<b_{n-1}\}$ so that $I(P)$ is lexicographically smaller than $I(Q)$. As $P\succcurlyeq Q$ and $P,Q$ are noncomparable in $\leq$, let $k$ be the smallest index such that $a_k>b_k$. Then define the Pl\"ucker relation of the form
\[r_{P,Q}=\Delta_I\Delta_J-\sum_{I',J'}(-1)^{a(I',J')}\Delta_{I'}\Delta_{J'}\]
where the sum is over $I',J'$ obtained from $I,J$ by swapping the last $n-k$ entries $\{a_{n-k},\ldots,a_{n-1}\}$ of $I$ with all possible choices of $(n-k)$-subsets of $J$, with $\Delta_I=\sum_{K\in\cE(I)}L_{K}$ and $a(I',J')$ is the total number of swaps needed to put both $I'$ and $J'$ in order, defined above.
\end{defin}
Let $R=\{r_{P,Q}\:|\: P\succcurlyeq Q\text{ noncomparable}\}$ be the set of all such relations.

\begin{ex}\label{ex:noncomparable_relation}
Let $P$ be the blue Dyck path and $Q$ be the black Dyck path in \cref{fig:noncomparable_relation} with $I=I(P)=\{1,2,5\}$ and $J=I(Q)=\{1,3,4\}$. Then the first index $k$ such that $a_k > b_k$ is $k=3$ with $a_k=5>b_k=4$. We have the relation
\[r_{P,Q}=\sum_{\substack{\sigma\in \cE(\{1,2,5\})\\ \kappa\in \cE(\{1,3,4\})}}L_\sigma L_\kappa + \sum_{\substack{\sigma\in \cE(\{1,2,3\})\\ \kappa\in \cE(\{1,4,5\})}}L_\sigma L_\kappa -\sum_{\substack{\sigma\in \cE(\{1,2,4\})\\ \kappa\in \cE(\{1,3,5\})}}L_\sigma L_\kappa.\]
\begin{figure}[h!]
\centering
 \begin{tikzpicture}[scale=0.7]
    \def\x{-0.5}
    \def\y{0.5}
    \def\z{1}
     \draw(0,0)--(2,2)--(3,1)--(5,3)--(8,0);
    \node at (0,0) {$\bullet$};
    \node[circle, draw=black, fill=white, minimum size=1, scale=0.4,] at (1,1) {};
    \node[draw, regular polygon,regular polygon sides=4,scale=0.8] at (1,1) {};
    \node[circle, draw=black, fill=black, minimum size=1, scale=0.4,] at (2,2) {};
    \node[circle, draw=black, fill=white, minimum size=1, scale=0.4,] at (3,1) {};
    \node[draw, regular polygon,regular polygon sides=4,scale=0.8] at (3,1) {};
    \node[circle, draw=black, fill=black, minimum size=1, scale=0.4,] at (4,2) {};
    \node[draw, regular polygon,regular polygon sides=4,scale=0.8] at (4,2) {};
    \node[circle, draw=black, fill=white, minimum size=1, scale=0.4,] at (5,3) {};
    \node[circle, draw=black, fill=black, minimum size=1, scale=0.4,] at (6,2) {};
    \node[circle, draw=black, fill=white, minimum size=1, scale=0.4,] at (7,1) {};
    \node[circle, draw=black, fill=black, minimum size=1, scale=0.4,] at (8,0) {};
   
    \node at (1,1+\x) {1};
    \node at (2,2+\x) {2};
    \node at (3,1+0.5) {3};
    \node at (4,2+\x) {4};
    \node at (5,3+0.5) {5};
    \node at (6,2+\x) {6};
    \node at (7,1+\x) {7};
    \node at (8,0+\x) {8};
    
    \draw[blue] (0,0+\y) -- (3,3+\y) -- (5,1+\y) -- (6,2+\y) -- (8,0+\y);
    \node[circle, draw=blue, fill=blue, minimum size=1, scale=0.4] at (0,0+\y) {};
    \node[circle, draw=blue, fill=white, minimum size=1, scale=0.4,] at (1,1+\y) {};
    \node[draw = blue, regular polygon,regular polygon sides=4,scale=0.8] at (1,1+\y) {};
    \node[circle, draw=blue, fill=blue, minimum size=1, scale=0.4,] at (2,2+\y) {};
    \node[draw=blue, regular polygon,regular polygon sides=4,scale=0.8] at (2,2+\y) {};
    \node[circle, draw=blue, fill=white, minimum size=1, scale=0.4,] at (3,3+\y) {};
    \node[circle, draw=blue, fill=blue, minimum size=1, scale=0.4,] at (4,2+\y) {};
    \node[draw, regular polygon,regular polygon sides=4,scale=0.8] at (4,2) {};
    \node[circle, draw=blue, fill=white, minimum size=1, scale=0.4,] at (5,1+\y) {};
    \node[draw=blue, regular polygon,regular polygon sides=4,scale=0.8] at (5,1+\y) {};
    \node[circle, draw=blue, fill=blue, minimum size=1, scale=0.4,] at (6,2+\y) {};
    \node[circle, draw=blue, fill=white, minimum size=1, scale=0.4,] at (7,1+\y) {};
    \node[circle, draw=blue, fill=blue, minimum size=1, scale=0.4,] at (8,0+\y) {};

    \node[blue] at (1,1+\z) {1};
    \node[blue] at (2,2+\z) {2};
    \node[blue] at (3,3+\z) {3};
    \node[blue] at (4,2+\z) {4};
    \node[blue] at (5,1+\z) {5};
    \node[blue] at (6,2+\z) {6};
    \node[blue] at (7,1+\z) {7};
    \node[blue] at (8,0+\z) {8};
\end{tikzpicture}
\caption{A pair of noncomparable Dyck paths}
\label{fig:noncomparable_relation}
\end{figure}
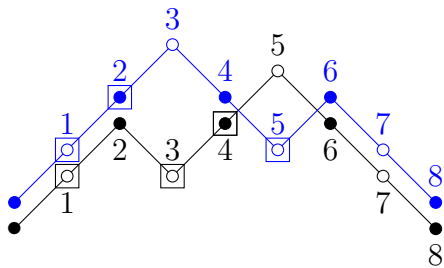
\end{ex}

The following key lemma is our ``straightening algorithm".
\begin{lemma}\label{lem:straightening}
For $P \succcurlyeq Q$, the leading term of $r_{P,Q}$ is $L_QL_P$. 
\end{lemma}
\begin{proof}
Recall that the leading term is the smallest (in terms of Dyck paths) monomial, compared from large to small. By \cref{lem:P-minimal-in-delta}, the leading term of $\Delta_I$ is $L_{P(I)}$, if $I$ is a Catalan subset. And for a pair of noncomparable Dyck paths $P \succcurlyeq Q$ with $I=I(P)$ and $J=I(Q)$, the leading term of $\Delta_I\Delta_J$ is $L_QL_P$.

Keep the notations as in this section. Let's examine the set \[I'=\{a_1,\ldots,a_{k-1},b_{m_1},\ldots,b_{m_{n-k}}\}\] obtained from $I$ in the definition of $r_{P,Q}$ (\cref{def:noncomparable_relation}), where $m_1<\cdots<m_{n-k}$. Without loss of generality, assume all elements in $I'$ are different, as we have $\Delta_{I'}=0$ if otherwise. For each $i=1,2,\ldots,n-1$, if $i\leq k-1$, we have $a_1,\ldots,a_{i}\leq 2i$; and if $i\geq k$, we have $a_1,\ldots,a_{k-1}\leq 2i$, and at least $i-k+1$ number of elements from $\{b_{m_1},\ldots,b_{m_{n-k}}\}$ that are less than or equal to $2i$ since $J=\{b_1<\cdots<b_{n-1}\}$ is a Catalan subset. This means that for each $i$, there are at least $i$ elements from $I'$ that do not exceed $2i$, and by \cref{lem:catalan-subset}, $I'$ is a Catalan subset. Moreover, by definition of $k$, the smallest number among $\{b_{m_1},\ldots,b_{m_{n-k}}\}$ is at most $b_k<a_k$. Thus, the $i$-th smallest number of $I'$ is smaller than or equal to the $i$-th smallest number of $I$ for $i\leq k$; and the $k$-th smallest number of $I'$ is strictly smaller than the $k$-th smallest number of $I$. This is saying $P(I')\succcurlyeq P(I)=P$ and $P(I')\neq P$. 

For each term $L_{P'}L_{Q'}$ in $\Delta_{I'}\Delta_{J'}$ in the relation $r_{P,Q}$, by \cref{lem:P-minimal-in-delta} and the argument above, $P'\succcurlyeq P(I')\succcurlyeq P\succcurlyeq Q$ and $P'\neq P$. So the monomial $L_{P'}L_{Q'}$ (no matter which one of $P',Q'$ is larger) is strictly larger than $L_PL_Q$ in our term order. As a result, $L_PL_Q$ is the leading term of $r_{P,Q}$ as desired, with coefficient $1$ from $\Delta_I\Delta_J$ that cannot be canceled. 
\end{proof}

\begin{theorem}\label{thm:tableaux-basis}
The set of relations $R$ is a Gr\"obner basis of the ideal $\II_n$. The graded piece $G_{n,d}$ of the Grove algebra has a linear basis given by $\{L_{\mathbf{P}}\:|\: \mathbf{P}\in\mathcal{C}_n^{(d)}\}$.
\end{theorem}
\begin{proof}
We have $R\subset\II_n$ by construction. To show that $\{L_{\mathbf{P}}\:|\: \mathbf{P}\in\mathcal{C}_n^{(d)}\}$ spans $G_{n,d}$, we use induction on the total order on the monomials. If a monomial $L_{\mathbf{Q}}$ is not standard, where $\mathbf{Q}=(Q_1\preccurlyeq\cdots \preccurlyeq Q_d)$, then there exists some $k$ such that $Q_k\nleq Q_{k+1}$. We now have a noncomparable pair of Dyck paths. Subtracting $L_{\mathbf{Q}}$ by the relation $r_{Q_{k+1},Q_k}\cdot L_{Q_1}\cdots L_{Q_{k-1}}L_{Q_{k+2}}\cdots L_{Q_d}$, we obtain a polynomial whose leading term is greater than $L_{\mathbf{Q}}$. By the induction hypothesis, this difference can be written as a linear combination of $\{L_{\mathbf{P}}\:|\: \mathbf{P}\in\mathcal{C}_n^{(d)}\}$, so $L_{\mathbf{Q}}$ can also be written as a linear combination of $\{L_{\mathbf{P}}\:|\: \mathbf{P}\in\mathcal{C}_n^{(d)}\}$ in $G_n$. Thus, standard monomials span. By \cref{prop:dim_G_nd} on the dimensions, we conclude that the standard monomials form a basis of $G_n$. Moreover, this means that $R$ generates the ideal $\II_n$, since any additional relation reduces the dimension of some graded component $G_{n,d}$.

To see that $R$ is in fact a Gr\"obner basis of $\II_n$, we need to show that the leading terms of polynomials in $R$, which are $\{L_PL_Q\:|\: P,Q\text{ are noncomparable}\}$, generate the initial ideal $\mathrm{init}_{\preccurlyeq}(\II_n)$. In other words, it suffices to show that for any $f\in\II_n$, the leading term of $f$ is not standard. Assume the opposite that the leading term of $f$ is $L_{\mathbf{Q}}$ with coefficient $1$, where $\mathbf{Q}=(Q_1\leq\cdots\leq Q_d)$. As explained in the previous paragraph, for any other monomial $L_{\mathbf{Q}'}$ in $f$, we can write it as a linear combination of standard monomials which are weakly greater than $\mathbf{Q}'$, thus strictly greater than $\mathbf{Q}$. Applying this procedure, we obtain a nonzero relation on the standard monomials, contradicting that standard monomials form a basis. 
\end{proof}

\newcommand{\etalchar}[1]{$^{#1}$}

\end{document}